\documentclass[a4paper, reqno, twoside, 10pt]{amsart}

\usepackage{amsmath,amssymb,amscd,mathtools}
\usepackage{booktabs,mathrsfs,array}
\usepackage{enumitem}

\usepackage[T1]{fontenc}
\usepackage{microtype}
\usepackage{newtxtext,newtxmath}
\usepackage{helvet}% selects Helvetica as sans-serif font
\usepackage{courier}% selects Courier as typewriter font

\usepackage[hidelinks]{hyperref}

\newtheorem{theorem}{Theorem}[section]
\newtheorem{maintheorem}{Theorem}

\newtheorem{lemma}[theorem]{Lemma}
\newtheorem{proposition}[theorem]{Proposition}
\newtheorem*{proposition*}{Proposition}
\newtheorem{corollary}[theorem]{Corollary}

\newtheorem*{conjecture*}{Conjecture}

\theoremstyle{definition}

\newtheorem{definition}[theorem]{Definition}

\newtheorem{example}[theorem]{Example}

\newtheorem{remark}[theorem]{Remark}
\newtheorem*{remark*}{Remark}

\numberwithin{equation}{section}

\DeclareMathOperator{\Hom}{Hom}
\DeclareMathOperator{\Ext}{Ext}
\DeclareMathOperator{\dep}{depth}
\DeclareMathOperator{\indx}{index}
\DeclareMathOperator{\coker}{coker}
\DeclareMathOperator{\Tot}{Tot}
\DeclareMathOperator{\Spec}{Spec}
\DeclareMathOperator{\End}{End}
\DeclareMathOperator{\length}{length}
\DeclareMathOperator{\Der}{Der}
\DeclareMathOperator{\Res}{Res}
\DeclareMathOperator{\divr}{div}
\DeclareMathOperator{\im}{im}
\DeclareMathOperator{\Specan}{Specan}
\DeclareMathOperator{\tr}{tr}
\DeclareMathOperator{\mult}{mult}
\DeclareMathOperator{\SL}{SL}
\DeclareMathOperator{\GL}{GL}
\DeclareMathOperator{\wt}{wt}
\DeclareMathOperator{\id}{id}
\DeclareMathOperator{\Sym}{Sym}
\DeclareMathOperator{\Supp}{Supp}
\DeclareMathOperator{\rank}{rank}
\DeclareMathOperator{\Aut}{Aut}

\hypersetup{pdftitle={Positivity of Wahl defects and representations from index-one covers}, pdfauthor={Heesang Park and Dongsoo Shin}, pdfsubject={Wahl defect; differential forms and residues; canonical representation; geometric genus; specialization of reflexive canonical powers}, pdfkeywords={normal surface singularity, Wahl defect, rational surface singularity, Q-Gorenstein singularity, canonical cover, rational homology disk smoothing}}

\begin{document}
%\title[Wahl defects and representations]{Positivity of Wahl defects and representations from index-one covers}

\title[Wahl defects and specialization]{Positivity of Wahl defects and specialization of reflexive powers of the canonical sheaf}

\author[H. Park]{Heesang Park}

\address{Department of Mathematics, Konkuk University, Seoul, Republic of Korea}

\email{HeesangPark@konkuk.ac.kr}

\author[D. Shin]{Dongsoo Shin}
%\thanks{Corresponding author: Dongsoo Shin (dsshin@cnu.ac.kr).}

\address{Department of Mathematics, Chungnam National University, Daejeon, Republic of Korea}

\email{dsshin@cnu.ac.kr}

\subjclass[2020]{14J17, 14B07, 32S25, 32S30}

\keywords{canonical cover, normal surface singularity, Q-Gorenstein surface singularity, rational homology disk smoothing, rational surface singularity, Wahl defect}

\begin{abstract}
Let \((X,E) \to (V,0)\) be the minimal good resolution of a normal complex surface singularity, with \(E\) reduced. The Wahl defect is \(\Delta(V,0) :=h^1(\mathcal{O}_X)-h^1(S_X)+h^1(\bigwedge^2 S_X)\), where \(S_X=\Theta_X(-\log E)\). For non-Gorenstein \(\mathbb{Q}\)-Gorenstein singularities \((V,0)\), we prove that \(\Delta(V,0)\geq0\), with equality precisely in the quasi-homogeneous case. This proves Wahl's Main Conjecture in this class and his \(\mathbb{Q}\)-Gorenstein and Rational Conjectures. The proof expresses \(\Delta\) as the multiplicity of the canonical character in a representation constructed from differential forms and residues. We also prove \(\Delta+4p_g\geq0\) for arbitrary normal surface singularities, with equality precisely for rational quasi-homogeneous singularities. For smoothings of rational surface singularities, we give a Riemann--Roch formula for the lengths of the cokernels of the restriction maps to the central fiber for all reflexive canonical powers, in terms of the Milnor number and periodic local Riemann--Roch corrections. This yields numerical and representation-theoretic criteria for \(\mathbb{Q}\)-Gorenstein smoothings, defined as quotients of equivariant smoothings of the index-one cover with trivial group action on the base. If a rational singularity admits a \(\mathbb{Q}\)-Gorenstein smoothing, such smoothings are precisely those with minimal Milnor number, or equivalently those whose smoothing components have minimal dimension. Combining these results, we prove that every rational homology disk smoothing is \(\mathbb{Q}\)-Gorenstein, has a rational weighted homogeneous central fiber, and occurs over a one-dimensional smoothing component. Together with the known weighted homogeneous classification, this completes the analytic classification of normal surface singularities admitting rational homology disk smoothings.
\end{abstract}

\maketitle

\section{Introduction}
\label{sec:introduction}

This paper studies positivity of the Wahl defect and specialization of reflexive canonical powers under smoothings of normal complex surface singularities. Our first main result proves positivity with a quasi-homogeneity equality criterion for \(\mathbb{Q}\)-Gorenstein surface singularities. Our second main result computes the specialization defects of all reflexive canonical powers under smoothings of rational surface singularities. Together with the known weighted homogeneous classification, these results complete the analytic classification of normal surface singularities admitting rational homology disk smoothings.

Let \((V,0)\) be a normal complex surface singularity, and let \(\pi\colon(\mathcal{V},0)\to(T,0)\) be a one-parameter smoothing over a smooth base. Write \(M_\pi\) for its Milnor fiber and \(\mu_\pi=b_2(M_\pi)\) for its \emph{Milnor number}. Following Wahl~\cite{Wahl2015}, we write \(\tau_\pi\) for the \emph{Tjurina number} of the smoothing, namely the dimension of the corresponding smoothing component in the base of a semi-universal deformation of \((V,0)\). Both invariants may depend on \(\pi\).

Let \(p\colon(X,E)\to(V,0)\) be the minimal good resolution, with \(E\) reduced, and put
\[
 S_X=\Theta_X(-\log E),\quad
 L_X=\bigwedge\nolimits^2S_X=\mathcal{O}_X(-(K_X+E)).
\]
The \emph{Wahl defect} is
\[
 \Delta(V,0):=h^1(X,\mathcal{O}_X)-h^1(X,S_X)+h^1(X,L_X).
\]
If \((V,0)\) is not a rational double point, Wahl~\cite[Corollary~2.5]{Wahl2015} gives
\begin{equation}\label{eq:introduction-wahl-smoothing}
 \Delta(V,0)=1+\mu_\pi-\tau_\pi+\alpha_\pi.
\end{equation}
Here \(\alpha_\pi\geq0\) is the length of the cokernel of the specialization map for the dual relative canonical module. The formula relates the singularity invariant \(\Delta(V,0)\) to the smoothing invariants \(\mu_\pi,\tau_\pi\) and the canonical-specialization correction \(\alpha_\pi\). We study both the defect and the specialization correction, extending the latter to every reflexive canonical power.

\subsection{Positivity of the Wahl defect}

The positivity question originates in the classical comparison of Milnor and Tjurina numbers. For an isolated hypersurface singularity, the Jacobian algebra surjects onto the Tjurina algebra, so \(\mu\geq\tau\); Saito~\cite{Saito1971} proved that equality characterizes quasi-homogeneity. The same inequality and equality criterion hold for isolated complete intersection singularities by work of Greuel~\cite{Greuel1980}, Looijenga--Steenbrink~\cite{LooijengaSteenbrink1985}, Wahl~\cite{Wahl1985}, and Vosegaard~\cite{Vosegaard2002}. For normal Gorenstein surface singularities other than rational double points, Wahl~\cite{Wahl1985} proved the corresponding statement \(\Delta(V,0)\geq1\), with equality if and only if \((V,0)\) is quasi-homogeneous. Here quasi-homogeneity means the existence of a good \(\mathbb{C}^*\)-action; equivalently, the singularity is analytically isomorphic to one defined by weighted homogeneous equations. We therefore use quasi-homogeneous and weighted homogeneous synonymously.

For non-Gorenstein quasi-homogeneous surface singularities, Wahl~\cite[Theorem~3.3(i)]{Wahl2015} proved that \(\Delta(V,0)=0\). His Main Conjecture~\cite[p.~317]{Wahl2015} predicts that every non-Gorenstein normal surface singularity satisfies \(\Delta(V,0)\geq0\), with equality precisely in the quasi-homogeneous case. The conjecture is false in general: Park--Shin~\cite{ParkShinNonlinear} construct families with arbitrarily negative Wahl defect. Our first main result proves the conjecture for non-Gorenstein \(\mathbb{Q}\)-Gorenstein surface singularities, in a form that also includes the Gorenstein case.

\begin{maintheorem}[Theorem~\ref{thm:main}]
Let \((V,0)\) be a normal \(\mathbb{Q}\)-Gorenstein complex surface singularity. Put
\[
 \varepsilon(V,0)=
 \begin{cases}
 1&\text{if \((V,0)\) is Gorenstein},\\
 0&\text{if \((V,0)\) is non-Gorenstein}.
 \end{cases}
\]
If \((V,0)\) is not a rational double point, then
\[
 \Delta(V,0)\geq\varepsilon(V,0),
\]
with equality if and only if \((V,0)\) is quasi-homogeneous. For a rational double point, \(\Delta(V,0)=0\).
\end{maintheorem}

The proof realizes the Wahl defect as the multiplicity of the canonical character in an actual finite-dimensional representation constructed from differential forms and residues on the index-one cover; see Theorem~\ref{thm:uniform}. Theorem~\ref{thm:main} requires neither smoothability of \((V,0)\) nor a complete intersection hypothesis on its index-one cover.

A \emph{strong \(\mathbb{Q}\)-Gorenstein smoothing} is a quotient of an equivariant smoothing of the index-one cover, with the cyclic group acting trivially on the base, as in Wahl~\cite[\S5]{Wahl2015}. We omit the qualifier when referring to this condition. The \emph{weak \(\mathbb{Q}\)-Gorenstein condition} requires only that \(K_{\mathcal{V}}\) be \(\mathbb{Q}\)-Cartier. These are the \(\mathbb{Q}\)-Gorenstein and naively \(\mathbb{Q}\)-Gorenstein conditions, respectively, of Lee--Nakayama~\cite[Definition~7.1 and Remark~7.6]{LeeNakayama2018}.

Theorem~\ref{thm:main} also proves Wahl's \(\mathbb{Q}\)-Gorenstein and Rational Conjectures~\cite[p.~318 and Conjecture~4.1]{Wahl2015}. For a strong \(\mathbb{Q}\)-Gorenstein smoothing, Wahl~\cite[Lemma~5.1]{Wahl2015} gives \(\alpha_\pi=0\), so \eqref{eq:introduction-wahl-smoothing} yields
\[
 1+\mu_\pi\geq\tau_\pi
\]
for a non-Gorenstein central singularity, with equality precisely in the quasi-homogeneous case; see Theorem~\ref{thm:QG-conjecture}. Every rational surface singularity is \(\mathbb{Q}\)-Gorenstein, and a rational Gorenstein surface singularity is a rational double point. Thus, for a rational singularity other than a rational double point, the same theorem gives
\[
 h^1(X,L_X)\geq h^1(X,S_X),
\]
again with equality precisely in the quasi-homogeneous case; see Theorem~\ref{thm:rational-conjecture}. These results do not settle Wahl's codimension-two conjecture~\cite[Conjecture~6.1]{Wahl2015}.

\subsection{Specialization of reflexive canonical powers}

Our second main result concerns the restriction of reflexive canonical powers under an arbitrary one-parameter smoothing of a rational surface singularity, without assuming that the smoothing is \(\mathbb{Q}\)-Gorenstein. Put \(R:=\mathcal{O}_{\mathcal{V},0}\), \(A:=\mathcal{O}_{V,0}\), and let \(t\) be the pullback of a coordinate on \(T\), so that \(A=R/(t)\). For \(i\in\mathbb{Z}\), define the \emph{specialization defect}
\[
 \delta_i(\pi):=
 \length_A\coker
 \left(\omega_R^{[i]}/t\omega_R^{[i]}\to\omega_A^{[i]}\right).
\]
Thus \(\delta_i(\pi)=0\) means that the \(i\)-th reflexive canonical power specializes, and Wahl's correction satisfies \(\alpha_\pi=\delta_{-1}(\pi)=\delta_2(\pi)\).

\begin{maintheorem}[Theorem~\ref{thm:rr} and Corollary~\ref{cor:periodic}]
Let \(\pi\) be a one-parameter smoothing of a rational normal complex surface singularity with local ring \(A\) and canonical index \(r>1\). Let \(n\) be the number of exceptional components of its minimal good resolution and \(K\) its numerical canonical cycle. There are rational numbers \(\gamma_i(A)\), depending only on \(A\), such that
\[
 \delta_i(\pi)=\binom{i}{2}(\mu_\pi-n-K^2)+\gamma_i(A)
 \quad(i\in\mathbb{Z}).
\]
They satisfy
\[
 \gamma_{i+r}(A)=\gamma_i(A),\quad
 \gamma_i(A)=\gamma_{1-i}(A),\quad
 \gamma_0(A)=\gamma_1(A)=0.
\]
\end{maintheorem}

The formula separates the dependence on the smoothing, through \(\mu_\pi\), from the local Riemann--Roch corrections \(\gamma_i(A)\) determined by the central singularity. The case \(i=-1\) recovers Wahl's anticanonical specialization formula~\cite[Theorem~1.1(i)]{Wahl2015}; the formula above applies to every reflexive canonical power.

Combining this formula with the finite specialization criterion of Lee--Nakayama~\cite[Lemma~7.3 and Proposition~7.11(i)]{LeeNakayama2018}, recalled in Proposition~\ref{thm:finite}, gives
\[
 \pi\text{ is strongly }\mathbb{Q}\text{-Gorenstein}
 \quad\Longleftrightarrow\quad
 \mu_\pi=n+K^2\text{ and }\gamma_j(A)=0\quad(2\leq j<r);
\]
see Corollary~\ref{cor:numerical}. Under the weak \(\mathbb{Q}\)-Gorenstein hypothesis, specialization of the compatible canonical algebras defines a finite-dimensional representation of the canonical covering group whose character multiplicities are the corrections \(\gamma_j(A)\). Its vanishing is equivalent to the strong condition; see Proposition~\ref{prop:canonical-specialization-representation}, the canonical-specialization representation criterion (CSRC). Example~\ref{ex:weak-qg} shows that the weak condition alone does not suffice for Wahl's inequality, even for certain weighted homogeneous rational surface singularities.

The specialization formula also compares all smoothings of a fixed rational singularity. If one strong \(\mathbb{Q}\)-Gorenstein smoothing exists, then the strong smoothings are precisely those with minimal Milnor number, or equivalently minimal smoothing-component dimension; see Proposition~\ref{prop:one-qg} and Corollary~\ref{cor:qg-minimal-components}. In particular, every weakly \(\mathbb{Q}\)-Gorenstein smoothing of the same singularity is then strong. Section~\ref{sec:consequences} applies this characterization to Gorenstein modifications of smoothing total spaces and to the variation of \(K^2\) in projective surface smoothings.

\subsection{Rational homology disk smoothings}

We apply the two main results together to rational homology disk (QHD) smoothings, whose Milnor fibers have the rational homology of the \(4\)-disk. Their classification is a longstanding problem, dating back to examples constructed in the early 1980s; see Looijenga--Wahl~\cite{LooijengaWahl1986} and Wahl~\cite{Wahl1981}.

A singularity admitting a QHD smoothing is rational by Durfee's formula; see also Stipsicz--Szab\'{o}--Wahl~\cite[Proposition~2.4(1)]{SSW2008}. The canonical-specialization formula, together with the resolution-lattice properties in Stipsicz--Szab\'{o}--Wahl~\cite[Proposition~2.4(2), (4)]{SSW2008}, forces every reflexive canonical power to specialize. Thus every QHD smoothing is strongly \(\mathbb{Q}\)-Gorenstein; see Theorem~\ref{thm:qhd-qg}. This extends Wahl's result for weighted homogeneous singularities~\cite[Theorem~3.4]{Wahl2013} without assuming weighted homogeneity. Applying positivity to \eqref{eq:introduction-wahl-smoothing} then gives the following joint consequence.

\begin{maintheorem}[Theorem~\ref{thm:qhd-qg} and Corollary~\ref{cor:qhd-components}]
Every rational homology disk smoothing \(\pi\) of a normal complex surface singularity \((V,0)\) is strongly \(\mathbb{Q}\)-Gorenstein. Its central singularity is rational and weighted homogeneous, and
\[
 \tau_\pi=1,\quad \alpha_\pi=0,\quad \Delta(V,0)=0.
\]
\end{maintheorem}

This reduction to weighted homogeneity is independent of the existing graph classification lists. Combining it with the weighted homogeneous classification of Stipsicz--Szab\'{o}--Wahl~\cite{SSW2008}, Bhupal--Stipsicz~\cite{BhupalStipsicz2011}, and Fowler~\cite{Fowler2013} gives the complete classification, up to analytic isomorphism, in Corollary~\ref{thm:qhd-classification}.

\begin{remark*}
Theorem~\ref{thm:qhd-qg} and Proposition~\ref{prop:canonical-specialization-representation} give the character equations \(\gamma_j(A)=0\) for QHD smoothings. In Park--Shin~\cite{ParkShinQHD}, these equations, together with numerical identities and Hodge inequalities, determine the possible Seifert data in the three-arm case with central self-intersection \(-2\). Marked projective compactifications and other algebraic-geometric arguments treat the remaining cases. Together with the structural properties established here, these arguments recover the graph classification of Bhupal--Stipsicz~\cite{BhupalStipsicz2011} and the corresponding analytic structures of Fowler~\cite{Fowler2013}, without using the existing graph classification lists as inputs; see Remark~\ref{rem:qhd-csrc-application}.
\end{remark*}

\begin{remark*}
Beke--Park--Shin--Stipsicz~\cite{BPSS2026} prove that a normal complex surface singularity whose minimal good resolution graph has at least two nodes admits no rational homology disk smoothing. Their proof combines Wahl's formula for the dimension of a rational homology disk smoothing component with estimates for logarithmic cohomology on the resolution. Together with the known classification in the quasi-homogeneous case, this yields the same complete classification.
\end{remark*}

\subsection{Function indices and methods of proof}

The positivity proof uses a formula expressing the Wahl defect in terms of a function on the singularity. Put \(A=\mathcal{O}_{V,0}\). For \(f\in\mathfrak{m}_A\) defining a reduced curve and having an isolated critical point at the origin, let \(F_f\) be its Milnor fiber and put
\[
 j_f=\dim_{\mathbb{C}}A/df(\Der_{\mathbb{C}}A),\quad
 \mu_f=b_1(F_f).
\]
Here \(df(\Der_{\mathbb{C}}A)\) is the ideal generated by \(\delta(f)\) for \(\delta\in\Der_{\mathbb{C}}A\).

\begin{maintheorem}[Proposition~\ref{prop:index}]
Let \((V,0)\) be a normal complex surface singularity which is not a rational double point. For every holomorphic function \(f\colon(V,0)\to(\mathbb{C},0)\) defining a reduced curve and having an isolated critical point at the origin,
\[
 \Delta(V,0)=1+j_f-\mu_f.
\]
\end{maintheorem}

This generalizes LP--NB--OO--T~\cite[Theorem~2.2]{LimaPereiraEtAl2024}, who prove the formula for isolated complete intersection singularities. On a principalization of the maximal ideal, logarithmic differentiation relates the critical algebra to the cohomology of the resolution. Riemann--Roch and the Euler characteristic of the Milnor fiber give the formula.

For a \(\mathbb{Q}\)-Gorenstein singularity, we pass to the index-one cover with local ring \(B\), covering group \(G\), and canonical character \(\kappa\). The function-index formula compares \(\Delta(A)\) with \(G\)-equivariant invariants of \(B\). Differential forms and residue maps on a resolution of \(B\) then give an actual finite-dimensional \(G\)-representation whose \(\kappa\)-multiplicity is \(\Delta(A)\); see Theorem~\ref{thm:uniform}. The residue construction proves effectivity of the underlying virtual representation; see Theorem~\ref{thm:effective}. This extends Wahl's representation-theoretic interpretation for quotients of equivariant hypersurface smoothings~\cite[Theorems~5.2 and~5.5 and Lemma~5.7]{Wahl2015}. For singular covers other than rational double points, a trivial summand accounts for the difference between the Gorenstein and non-Gorenstein lower bounds. Euler contraction gives equality for quasi-homogeneous covers. Conversely, the divergence argument using the Scheja--Wiebe criterion and the residue argument in the boundary case produce a nonzero canonical-character component of the residue contribution for non-quasi-homogeneous covers, giving strict inequality; M\"{u}ller's finite-quotient theorem transfers this conclusion to the base. Smooth and rational double point covers are treated separately by a determinant-character calculation.

The canonical-specialization formula uses Riemann--Roch on a resolution where the full pullback of the relevant reflexive canonical power is invertible. The calculation relates its local contribution to the Milnor number, and the finite canonical index makes the correction terms periodic. Nonnegative lengths then give numerical constraints on the smoothing. These constraints are independent of the positivity theorem; the QHD application uses both results, first to prove specialization and then to deduce quasi-homogeneity.

Park--Shin~\cite{ParkShinTaut} use quiver representations to study the logarithmic tangent cohomology \(H^1(X,S_X)\) appearing in the Wahl defect. The representations are defined by restriction maps on exceptional cycles, and their extension groups compute the corresponding restricted cohomology. This gives another proof of Laufer's classification of taut normal complex surface singularities~\cite{Laufer1973}.

\subsection{A genus bound for arbitrary normal surface singularities}

Although the Wahl defect can be arbitrarily negative, a correction involving the geometric genus \(p_g(V,0)=h^1(X,\mathcal{O}_X)\) gives a uniform lower bound.

\begin{maintheorem}[Theorem~\ref{genus:thm:main}]
For every normal complex surface singularity \((V,0)\),
\[
 \Delta(V,0)+4p_g(V,0)\geq0,
\]
with equality if and only if \((V,0)\) is rational and quasi-homogeneous.
\end{maintheorem}

The proof constructs a numerically trivial line bundle \(N\) for which
\[
 h^1(X,N)-h^1(X,S_X\otimes N)+h^1(X,L_X\otimes N)\geq0.
\]
An equivariant residue comparison on a resolution of a finite cyclic cover gives this inequality. Tensoring a vector bundle of rank \(s\) by \(N\) changes its first cohomology dimension by at most \(sp_g\). Applying this estimate to \(\mathcal{O}_X,S_X,L_X\), of ranks \(1,2,1\), gives the bound. A strict estimate for \(p_g>0\) shows that equality forces rationality, and the rational equality criterion then gives quasi-homogeneity.

\subsection*{Organization}

Sections~\ref{qg:sec:relative-cover}--\ref{genus:sec:bound} develop the positivity results: they fix the canonical-cover conventions, prove the function-index formula, construct the residue representations, establish positivity and its equality criterion, and prove the genus bound. Sections~\ref{sec:algebra}--\ref{sec:consequences} develop specialization of reflexive canonical powers, its Riemann--Roch formula and criteria, and applications to smoothing components. Section~\ref{sec:qhd} combines the two sets of results to complete the classification of singularities admitting rational homology disk smoothings. Throughout the paper, we work in the complex analytic category.

\subsection*{Statements and Declarations}

\begin{itemize}[labelindent=0pt, leftmargin=*, align=left]
\item \textbf{Funding.} Heesang Park was supported by the Basic Science Research Program through the National Research Foundation of Korea (NRF) (RS-2026-25590534). Dongsoo Shin was supported by the Basic Science Research Program through the National Research Foundation of Korea (NRF) (NRF-2023R1A2C1007293).

\item \textbf{Competing interests.} The authors declare that they have no financial or non-financial competing interests relevant to the content of this article.

\item \textbf{Use of artificial intelligence.} The authors used ChatGPT-6 Astra to search for and analyze earlier literature, in particular the work of Wahl~\cite{Wahl1985,Wahl2015} and Lima-Pereira, Nuño-Ballesteros, Oréfice-Okamoto, and Tomazella~\cite{LimaPereiraEtAl2024}. The authors proposed various approaches to generalizing the ideas in these papers and used ChatGPT-6 Astra to examine their feasibility. Other uses included language editing and identifying errors in the manuscript. The authors take full responsibility for the correctness of its contents.
\end{itemize}

\section{Index-one covers and characters}
\label{qg:sec:relative-cover}

We recall the index-one-cover construction and fix the character conventions used below. For the canonical-cover construction, see Matsuki~\cite[\S4.5]{Matsuki2002} and Wahl~\cite[\S5]{Wahl2015}. For isotypic decompositions and representation rings, see Serre~\cite[\S\S2.6, 9.1]{Serre1977}. Recall that every rational normal complex surface singularity is \(\mathbb{Q}\)-Gorenstein; see Lipman~\cite[Proposition~17.1]{Lipman1969} for instance.

Let \((V,0)\) be a normal \(\mathbb{Q}\)-Gorenstein surface singularity of index \(r>1\), put \(A=\mathcal{O}_{V,0}\), and choose an isomorphism \(\omega_A^{[r]} \simeq A\). Here, for a rank-one reflexive module \(M\) over the normal local ring \(A\), write \(M^\vee=\Hom_A(M,A)\). We define its reflexive powers by
\[
 M^{[0]}=A,
 M^{[m]}=(M^{\otimes m})^{\vee\vee},
 M^{[-m]}=((M^\vee)^{\otimes m})^{\vee\vee}
  (m \geq 1).
\]
Since \(M^\vee\) is reflexive, \(M^{[-1]}=M^\vee\). In particular,
\[
 \omega_A^{[-1]}=\Hom_A(\omega_A,A).
\]
The natural products of reflexive canonical powers, together with this isomorphism, give
\[
 B=\bigoplus_{i=0}^{r-1}\omega_A^{[i]}
\]
the structure of a finite \(A\)-algebra: products of degree at least \(r\) are identified with the corresponding summands of degree less than \(r\) using the chosen isomorphism. The finite morphism
\[
 q\colon (\widetilde{V},0)=\Specan B \to (V,0)
\]
induced by \(A \hookrightarrow B\) is called the index-one cover of \((V,0)\), and \(B=\mathcal{O}_{\widetilde{V},0}\) is its local ring. The complex analytic space \(\widetilde{V}\) is connected and normal. The grading defines an action of \(G=\mu_r\) on \(B\), and the induced action on \(\widetilde{V}\) is free away from the origin. Moreover, \(A=B^G\), so \(q\) is the quotient map; see Reid~\cite[\S\S3.5--3.6]{Reid1987} and Matsuki~\cite[\S4.5]{Matsuki2002} for details. Since \(B\) is a two-dimensional normal local ring, it is Cohen--Macaulay by Serre's criterion. Since \(A \to B\) is a finite local homomorphism of Cohen--Macaulay rings of dimension two, duality for finite morphisms gives an isomorphism of \(B\)-modules
\[
 \omega_B \simeq \Hom_A(B,\omega_A);
\]
see the Stacks Project~\cite[Tag~0AX1]{Stacks}. The \(B\)-module structure on the right is given by \((b\varphi)(c)=\varphi(bc)\).

Let \(p_1\colon B \to \omega_A\) denote the projection onto the degree-one summand. Multiplication in \(B\) defines a \(B\)-linear map
\[
 \Phi\colon B \to \Hom_A(B,\omega_A),
 \Phi(b)(c)=p_1(bc).
\]
Away from the origin, the summands of degrees \(i\) and \(1-i\) modulo \(r\) are invertible and are paired perfectly by multiplication. Thus \(\Phi\) is an isomorphism there. Both modules are reflexive over \(A\), so \(\Phi\) is an isomorphism everywhere; see the Stacks Project~\cite[Tag~0AVB]{Stacks}. Consequently, \(\omega_B \simeq B\) as \(B\)-modules, after forgetting the \(G\)-actions. Since \(B\) is Cohen--Macaulay, it is Gorenstein.

The covering group \(G=\mu_r\) acts on functions and differential forms by the left action
\[
 g\cdot\alpha=(g^{-1})^*\alpha
  (g \in G).
\]
We use the induced actions on canonical modules and cohomology. Since \(\omega_B\) is free of rank one, the quotient \(\omega_B/\mathfrak{m}_B\omega_B\) is a one-dimensional representation of \(G\). Denote its character by \(\kappa\).

Choose \(\eta_0 \in \omega_B\) whose image in this quotient is nonzero, and set \(\eta=P_\kappa(\eta_0)\), where
\[
 P_\kappa(v)=\frac1{|G|}
 \sum_{g \in G}\kappa(g)^{-1}(g\cdot v)
\]
is the projection onto the \(\kappa\)-eigenspace. The induced projection on \(\omega_B/\mathfrak{m}_B\omega_B\) is the identity, so \(\eta\) and \(\eta_0\) have the same nonzero image in this quotient. Nakayama's lemma therefore shows that \(\eta\) generates \(\omega_B\), and by construction
\[
 g\cdot\eta=\kappa(g)\eta
  (g \in G).
\]

Let \(d\) be the order of \(\kappa\). Then \(\eta^{\otimes d}\) is \(G\)-invariant. Since the covering is finite \'{e}tale away from the origin, this section descends to a nowhere-vanishing section of \(\omega_{V\setminus\{0\}}^{\otimes d}\). Extending the resulting trivialization across the origin as an isomorphism of reflexive sheaves gives \(\omega_A^{[d]} \simeq A\). Thus \(r \mid d\), since \(A\) has canonical index \(r\). On the other hand, \(d \mid r\) because \(G=\mu_r\). Hence \(d=r\), and \(\kappa\) is faithful, meaning that \(\kappa(g)=1\) only when \(g\) is the identity element of \(G\).

For a finite group \(G\), write \(R(G)\) for its complex representation ring, \([U]\) for the class of a representation, \(U^*\) for its dual, and \(\mathbf{1}\) for the trivial representation. We write tensor products with one-dimensional characters multiplicatively. For cyclic \(G\), the notation \(U_\chi\) denotes the \(\chi\)-eigenspace, and \(\mult_\chi\) denotes multiplicity, extended additively to \(R(G)\). For a coherent sheaf on a compact exceptional cycle, \(\chi_G\) denotes its equivariant Euler characteristic. Canonical bundles always carry their natural linearizations.

A \(\mathbb{Q}\)-Gorenstein smoothing is understood in the index-one-cover sense: it is induced by a \(G\)-equivariant smoothing of the index-one cover, with trivial action on the base. We use the standard equivalence with the relative canonical-cover formulation; see Koll\'{a}r--Shepherd-Barron~\cite{KollarShepherdBarron1988} and Wahl~\cite[\S5]{Wahl2015}. In particular its central singularity is \(\mathbb{Q}\)-Gorenstein.

\section{The function index}\label{sec:index}

Let \((V,0)\) be a normal complex surface singularity which is not a rational double point, and set \(A=\mathcal{O}_{V,0}\). We write
\[
 \omega_A^{\vee}=\Hom_A(\omega_A,A).
\]
No \(\mathbb{Q}\)-Gorenstein assumption is imposed in this section. For \(f \in \mathfrak{m}_A\) defining a reduced curve and having an isolated critical point at the origin, put
\[
 J_f=df(\Der_{\mathbb{C}} A),
 j_f=\length(A/J_f),
 \mu_f=1-\chi(F_f),
\]
where \(F_f\) is the Milnor fiber of the function \(f\colon(V,0) \to (\mathbb{C},0)\). Here \(df\colon\Der_{\mathbb{C}} A \to A\) denotes the \(A\)-linear map \(\delta \mapsto \delta(f)\). Thus \(J_f\) is the ideal generated by the derivatives of \(f\) along holomorphic vector fields on \(V\); when \(V\) is smooth, it is the usual Jacobian ideal. Under our hypotheses, \(A/J_f\) has finite length, and \(j_f=\dim_{\mathbb{C}}(A/J_f)\).

The invariant \(j_f\) coincides with the relative Bruce--Roberts number \(\mu_{\mathrm{BR}}^{-}(f,V)\); see LP--NB--OO--T~\cite[\S1]{LimaPereiraEtAl2024}. The Milnor fiber \(F_f\) is a connected compact real two-dimensional manifold with boundary. Moreover, \(\mu_f=1-\chi(F_f)=b_1(F_f)\) agrees with the Milnor number of the reduced curve \(f^{-1}(0)\); see Buchweitz--Greuel~\cite[Theorem~4.2.2 and Corollary~4.2.3(1)]{BuchweitzGreuel1980}.

Let \(p\colon(X,E) \to (V,0)\) be the minimal good resolution, with \(E\) reduced. As in the introduction, write
\[
 S_X=\Theta_X(-\log E),
 L_X=\mathcal{O}_X(-(K_X+E)),
\]
and
\[
 \Delta(V,0)=h^1(\mathcal{O}_X)-h^1(S_X)+h^1(L_X).
\]
When \((V,0)\) is an isolated complete intersection singularity, LP--NB--OO--T~\cite[Theorem~2.2]{LimaPereiraEtAl2024} prove that
\[
 j_f=\mu_f+\mu(V,0)-\tau(V,0),
\]
where \(\mu(V,0)\) and \(\tau(V,0)\) are the Milnor and Tjurina numbers of the surface singularity \((V,0)\). Since a complete intersection is Gorenstein, Wahl~\cite[Corollary~2.5]{Wahl2015} gives \(\Delta(V,0)=1+\mu(V,0)-\tau(V,0)\). Thus \(\Delta(V,0)=1+j_f-\mu_f\) in the complete intersection case. The following proposition shows that this identity holds without the complete intersection assumption.

\begin{proposition}\label{prop:index}
Let \((V,0)\) be a normal complex surface singularity which is not a rational double point. For every holomorphic function \(f\colon(V,0) \to (\mathbb{C},0)\) defining a reduced curve and having an isolated critical point at the origin,
\begin{equation*}%\label{eq:index}
 \Delta(V,0)=1+j_f-\mu_f.
\end{equation*}
\end{proposition}

We prove Proposition~\ref{prop:index} in Section~\ref{sec:proof-of-prop-index}.

\subsection{Resolution comparison and conservation}

All resolutions below are taken over sufficiently small Stein representatives. For any such resolution \(r\colon Z \to V\) and any coherent analytic sheaf \(\mathcal{F}\) on \(Z\), the group \(H^1(Z,\mathcal{F})\) is finite dimensional and \(H^2(Z,\mathcal{F})=0\).

We first recall the anticanonical extension property. By Okuma~\cite{Okuma1997} (see also Wahl~\cite[Lemma~2.1(i)]{Wahl2015}), one has \(H_E^1(X,L_X)=0\). The local cohomology sequence therefore identifies \(H^0(X,L_X)\) with \(H^0(X\setminus E,L_X)\). Put \(U=V\setminus\{0\}\), and let \(j\colon U \hookrightarrow V\) be the inclusion. Under the identification \(X\setminus E \simeq U\), the restriction of \(L_X\) is the anticanonical bundle \(\omega_U^{-1}\). The sheaf \(\omega_V^\vee\) is reflexive and restricts to \(\omega_U^{-1}\) on \(U\). Since \(V\) is normal and \(V\setminus U=\{0\}\) has codimension two, the extension property of reflexive sheaves gives
\[
 \omega_V^\vee \simeq j_*(\omega_U^{-1}).
\]
In other words, every section of \(\omega_U^{-1}\) extends uniquely across the origin as a section of \(\omega_V^\vee\). Combining this with the preceding restriction isomorphism yields
\[
 p_*L_X \simeq j_*(\omega_U^{-1}) \simeq \omega_V^\vee.
\]
Taking the stalk at \(0\), and writing \(H^0(X,L_X)\) for \((p_*L_X)_0\), we obtain
\begin{equation}\label{eq:anti}
 H^0(X,L_X) \simeq \omega_A^\vee.
\end{equation}

Let \(q\colon\widetilde{X} \to X\) be a principalization of \(\mathfrak{m}_A\mathcal{O}_X\) obtained by successive point blowups. At each stage \(Y\), write
\[
\mathfrak{m}_A\mathcal{O}_Y
=
\mathcal{O}_Y(-D_Y)\mathcal{J}_Y,
\]
where \(D_Y\) is the fixed divisorial part and \(\mathcal{J}_Y \subseteq \mathcal{O}_Y\) is obtained locally by dividing the generators by a local equation of \(D_Y\). We blow up the points of \(\Supp(\mathcal{O}_Y/\mathcal{J}_Y)\) and repeat until the residual ideal is the unit ideal. Thus \(\mathfrak{m}_A\mathcal{O}_{\widetilde{X}}\) is invertible.

Let \(\widetilde{E}=\sum_i\widetilde{E}_i\) be the reduced exceptional divisor of \(\widetilde{p}=p\circ q \colon \widetilde{X} \to V\), and write
\[
 \mathfrak{m}_A\mathcal{O}_{\widetilde{X}} =\mathcal{O}_{\widetilde{X}}(-M),
 M=\sum_i m_i\widetilde{E}_i,\quad m_i>0.
\]
Put
\[
 S_{\widetilde{X}} =\Theta_{\widetilde{X}}(-\log\widetilde{E}),
 L_{\widetilde{X}} =\mathcal{O}_{\widetilde{X}}(-(K_{\widetilde{X}}+\widetilde{E})),
\]
and define
\[
 \begin{aligned}
 a_{\widetilde{X}}
 &=\length(\omega_A^\vee/H^0(\widetilde{X},L_{\widetilde{X}})),\\
 \Delta_{\widetilde{X}}
 &=h^1(\mathcal{O}_{\widetilde{X}})-h^1(S_{\widetilde{X}})+h^1(L_{\widetilde{X}}).
 \end{aligned}
\]

The following lemma combines the standard lifting property of vector fields with the transformation formulas for logarithmic tangent sheaves under point blowups; see Bergner--Graf~\cite[\S2, Fact~1 and Proposition~1]{BergnerGraf2020} and Flenner--Zaidenberg~\cite[Lemma~1.5(3)--\allowbreak(4)]{FlennerZaidenberg1994}. For an equivariant log-resolution of the maximal ideal, see also Favre--Ruggiero~\cite[Proposition~3.1]{FavreRuggiero2014}. We include the proof to describe the lifting through the chosen principalization and to keep track of the anticanonical correction term.

\begin{lemma}\label{lem:resolution-comparison}
With the above notation, one has
\[
 H^0(\widetilde{X},S_{\widetilde{X}})=\Der_{\mathbb{C}} A
\]
and
\begin{equation}\label{eq:correction}
 \Delta(V,0) = \Delta_{\widetilde{X}}+a_{\widetilde{X}}.
\end{equation}
\end{lemma}

\begin{proof}
\emph{Lifting vector fields.} Every derivation of \(A\) preserves \(\mathfrak{m}_A\) and lifts uniquely to a logarithmic vector field on the minimal good resolution; see Bergner--Graf~\cite[\S2, Fact~1 and Proposition~1]{BergnerGraf2020}. We show that these vector fields lift through \(q\).

On an intermediate resolution \(Y\), write locally
\[
\mathfrak{m}_A\mathcal{O}_Y=gJ,
\]
where \((g)\) is the fixed divisorial part and \(J\) is the residual ideal. A lifted derivation \(\delta\) preserves \(gJ\), and logarithmicity gives \(\delta(g) \in (g)\). Hence
\[
\delta(J) \subseteq J.
\]
In characteristic zero, the radical of a derivation-invariant ideal is also invariant. At each residual base point \(y\), we have \(\sqrt{J_y}=\mathfrak{m}_y\), and therefore
\[
\delta(\mathfrak{m}_y^k) \subseteq \mathfrak{m}_y^k
\quad\text{for all } k \ge 0.
\]
Thus \(\delta\) induces a graded derivation of the Rees algebra and lifts to the blowup at \(y\), logarithmically along the reduced total exceptional divisor. Induction proves that every derivation of \(A\) lifts through \(q\). Conversely, a vector field on \(Y\) acts on \(H^0(Y,\mathcal{O}_Y)=A\). Consequently,
\[
H^0(\widetilde{X},S_{\widetilde{X}})
=\Der_{\mathbb{C}} A.
\]
The same argument applies to every intermediate resolution; in particular, all global logarithmic vector fields vanish at each subsequent blowup center.

\emph{Change under a point blowup.} Use the analogous notation \(S_Y,L_Y,a_Y,\Delta_Y\) on each intermediate resolution. Consider a blowup \(b\colon Y' \to Y\) at a smooth point \(y\) of the reduced exceptional divisor \(D\), and let \(F\) be its exceptional curve. By Flenner--Zaidenberg~\cite[Lemma~1.5(3)]{FlennerZaidenberg1994},
\[
0 \to b_*S_{Y'} \to S_Y \to T_yD \to 0,
  R^1b_*S_{Y'}=0,
\]
where \(T_yD\) is viewed as a skyscraper sheaf. The evaluation map on global sections is zero by the preceding paragraph, so
\[
h^1(S_{Y'})=h^1(S_Y)+1.
\]
Moreover,
\[
L_{Y'} \simeq b^*L_Y(-F),
  b_*L_{Y'}=L_Y\mathfrak{m}_y,
  R^1b_*L_{Y'}=0.
\]
Put
\[
c=\rank(
H^0(Y,L_Y) \to L_Y|_y
).
\]
The evaluation sequence for \(L_Y\) gives
\[
h^1(L_{Y'})=h^1(L_Y)+1-c,
  a_{Y'}=a_Y+c.
\]
Since \(h^1(\mathcal{O}_{Y'})=h^1(\mathcal{O}_Y)\), it follows that
\[
\Delta_{Y'}+a_{Y'}=\Delta_Y+a_Y.
\]

If the center is a node of \(D\), then \(S_{Y'} \simeq b^*S_Y\) by Flenner--Zaidenberg~\cite[Lemma~1.5(4)]{FlennerZaidenberg1994}, and \(L_{Y'} \simeq b^*L_Y\). Thus both \(\Delta_Y\) and \(a_Y\) remain unchanged. Applying these observations along \(q\), and using \(a_X=0\) from \eqref{eq:anti}, yields
\[
\Delta_{\widetilde{X}}+a_{\widetilde{X}}
=\Delta(V,0). \qedhere
\]
\end{proof}

Let \(f \in \mathfrak{m}_A\) define a reduced curve and have an isolated critical point at the origin. Choose a general linear combination \(\ell\) of a minimal set of generators of \(\mathfrak{m}_A\), and put \(f_t=f+t\ell\). For general sufficiently small \(t \neq 0\), the function \(f_t\) also defines a reduced curve and has an isolated critical point at the origin. We write
\[
j_{f_t}=\length(A/df_t(\Der_{\mathbb{C}} A)),
 \mu_{f_t}=1-\chi(F_{f_t,0}),
\]
where, for each fixed such \(t\), \(F_{f_t,0}\) is the local Milnor fiber of \(f_t\) at \(0 \in V\), taken in a sufficiently small neighborhood containing no other critical point of \(f_t\).

\begin{lemma}\label{lem:function-conservation}
For general sufficiently small \(t \neq 0\), one has
\[
j_f-\mu_f=j_{f_t}-\mu_{f_t}.
\]
\end{lemma}

\begin{proof}
We show that both \(j_f\) and \(\mu_f\) decrease by the sum of the Milnor numbers at the critical points of \(f_t\) away from the origin. For \(j_f\), we prove flatness of the relative Jacobian quotient; for \(\mu_f\), we use conservation of the radial index.

Set \(T=\mathbb{C}\{t\}\), \(R=A\{t\}\), and
\[
\mathscr{D}=(\Der_{\mathbb{C}} A)\otimes_A R,
 \mathscr{Q}=R/df_t(\mathscr{D}),
\]
Here \(df_t\) denotes differentiation along \(V\), and \(\mathscr{Q}\) is the relative Jacobian quotient. The modules \(\Der_{\mathbb{C}} A\) and \(\omega_A^\vee\) are reflexive over the normal local ring \(A\), hence maximal Cohen--Macaulay. Their extensions to the normal three-dimensional algebra \(R\) therefore have depth three. Since
\[
\mathscr{Q}/t\mathscr{Q} \simeq A/J_f
\]
has finite length, the analytic finiteness theorem implies that \(\mathscr{Q}\) is finite over \(T\).

The image of \(df_t\colon\mathscr{D} \to R\) is torsion-free, so the localized depth lemma shows that its kernel is reflexive. Outside the singular axis and \(\Supp\mathscr{Q}\), contraction with \(df_t\) identifies this kernel with \(\omega_A^\vee\otimes_A R\). Both excluded sets have codimension at least two. Reflexive extension therefore gives an exact sequence
\[
0 \to \omega_A^\vee\otimes_A R
 \to \mathscr{D}
 \xrightarrow{df_t} R
 \to \mathscr{Q} \to 0.
\]
Applying the depth lemma twice yields
\[
\dep_R\mathscr{Q} \ge 1.
\]
The \(T\)-torsion submodule of \(\mathscr{Q}\), if nonzero, would be finite-dimensional over \(\mathbb{C}\), since \(\mathscr{Q}\) is finite over \(T\). It would thus be a nonzero finite-length \(R\)-submodule, contrary to the positive depth of \(\mathscr{Q}\). Hence \(\mathscr{Q}\) is finite free over \(T\), of rank \(j_f\).

Choose sufficiently small representatives on which the corresponding quotient sheaf is finite over the parameter disc. After shrinking the disc, its direct image is locally free of rank \(j_f\). Let \(\Sigma_t\) be the set of critical points of \(f_t\) away from the origin in the chosen representative. For \(p \in \Sigma_t\), the surface \(V\) is smooth at \(p\), and the local quotient has length
\[
\dim_{\mathbb{C}}
\frac{\mathcal{O}_{V,p}}{df_t(\Theta_{V,p})}
=\mu(f_t,p),
\]
where \(\mu(f_t,p)\) is the Milnor number of \(f_t-f_t(p)\) at \(p\). Constancy of the total fiber length gives
\[
j_f=j_{f_t}+\sum_{p \in \Sigma_t}\mu(f_t,p).
\]

By Ebeling--Gusein-Zade--Seade~\cite[Remark~2.4 and Example~2.6]{EbelingGuseinZadeSeade2004}, the radial index of \(df\) on the complex surface \(V\) equals \(1-\chi(F_f)=\mu_f\) and satisfies conservation of number. At each \(p \in \Sigma_t\), the radial index of \(df_t\) equals \(\mu(f_t,p)\). Consequently,
\[
\mu_f=\mu_{f_t}+\sum_{p \in \Sigma_t}\mu(f_t,p).
\]
Subtracting the two identities proves
\[
j_f-\mu_f=j_{f_t}-\mu_{f_t}. \qedhere
\]
\end{proof}

\subsection{Logarithmic differentiation and the Milnor fiber}

We first describe the logarithmic differential on a resolution. The local calculation applies whenever the total transform of the zero divisor has the form below. The identification of its global sections uses the chosen principalization of the maximal ideal.

\begin{lemma}\label{lem:logarithmic-differential}
Let \(f \in \mathfrak{m}_A\) define a reduced curve and have an isolated critical point at the origin.
\begin{enumerate}[label=\textup{(\roman*)}]
\item Let \(r\colon(Y,D) \to (V,0)\) be a good resolution, and put \(S_Y=\Theta_Y(-\log D)\) and \(L_Y=\mathcal{O}_Y(-(K_Y+D))\). Suppose that
\[
 \divr_Y(f)=Z+H,\quad
 Z=\sum_i z_iD_i,\quad z_i>0,
\]
where \(D=\sum_iD_i\) is the reduced exceptional divisor and \(H\) is smooth, reduced, transverse to \(D\), and disjoint from its nodes. On a sufficiently small representative, there is an exact sequence
\begin{equation}\label{eq:df}
 0 \to L_Y(Z) \to S_Y \xrightarrow{df} \mathcal{O}_Y(-Z) \to 0.
\end{equation}

\item For the principalization \(\widetilde{p}\colon(\widetilde{X},\widetilde{E}) \to (V,0)\) fixed above, if the image of \(f\) in \(\mathfrak{m}_A/\mathfrak{m}_A^2\) is sufficiently general, the hypotheses of \textup{(i)} hold with \((Y,D,Z)=(\widetilde{X},\widetilde{E},M)\). Moreover,
\[
 H^0(\widetilde{X},\mathcal{O}_{\widetilde{X}}(-M))=\mathfrak{m}_A,
 \quad
 H^0(\widetilde{X},L_{\widetilde{X}}(M)) \simeq \omega_A^\vee.
\]
The second identification agrees with restriction away from \(\widetilde{E}\).
\end{enumerate}
\end{lemma}

\begin{proof}
(i) Logarithmic vector fields preserve the ideal of every component of \(D\), so differentiation gives \(df\colon S_Y \to \mathcal{O}_Y(-Z)\). We check surjectivity along \(D\). At a smooth point of \(D\) outside \(H\), write \(D=\{x=0\}\) and \(f=x^mu\), where \(u\) is a unit. In the basis \(x\partial_x,\partial_y\), the normalized map is
\[
 x^{-m}df=(mu+xu_x,u_y).
\]
The first coefficient restricts to the unit \(mu\) on \(x=0\). At a point of \(D\cap H\), choose coordinates with \(f=x^my\). Then
\[
 x^{-m}df=(my,1).
\]
At a node, write \(D=\{xy=0\}\) and \(f=x^my^nu\), where \(u\) is a unit. In the basis \(x\partial_x,y\partial_y\), the map is
\[
 x^{-m}y^{-n}df=(mu+xu_x,nu+yu_y).
\]
Both coefficients are nonzero at the node. Away from \(D\), surjectivity follows because \(f\) has no critical point on \(V\setminus\{0\}\), after shrinking the representative. The kernel is therefore invertible, and taking determinants gives
\[
 \ker(df) \simeq \det(S_Y)\otimes\mathcal{O}_Y(Z)=L_Y(Z).
\]
This proves \eqref{eq:df}.

(ii) Write \(f=g+h\), where \(g\) is a general linear combination of fixed minimal generators of \(\mathfrak{m}_A\) and \(h \in \mathfrak{m}_A^2\). Their pullbacks generate \(\mathcal{O}_{\widetilde{X}}(-M)\). Thus
\[
 \divr_{\widetilde{X}}(g)=M+H_g,
\]
where, near \(\widetilde{E}\), the divisor \(H_g\) is smooth, reduced, transverse to \(\widetilde{E}\), and disjoint from its nodes. Let \(s\) be a local equation of \(M\). Since \(h \in \mathcal{O}_{\widetilde{X}}(-2M)\), we may write
\[
 g=su,\quad h=s^2v,\quad f=s(u+sv).
\]
All coefficients of \(M\) are positive, so
\[
 \left.\frac{f}{s}\right|_{\widetilde{E}}
 =\left.\frac{g}{s}\right|_{\widetilde{E}}.
\]
These restrictions have the same simple zeros on each exceptional component and do not vanish at a node. At each zero, the tangent derivative along that component is nonzero. Hence the divisor of \(f/s\) is smooth and transverse there. After shrinking the representative, we obtain \(\divr_{\widetilde{X}}(f)=M+H\) with the required properties.

Normality gives \(H^0(\widetilde{X},\mathcal{O}_{\widetilde{X}})=A\). Since \(M \ge \widetilde{E}\),
\[
 \mathfrak{m}_A
 \subseteq H^0(\widetilde{X},\mathcal{O}_{\widetilde{X}}(-M))
 \subseteq \{a \in A:a(0)=0\}=\mathfrak{m}_A.
\]
The first inclusion follows from the definition of \(M\), proving the first identification.

Put \(\mathscr{K}_f=\ker(df\colon\Der_{\mathbb{C}} A \to A)\). As in the proof of Lemma~\ref{lem:function-conservation}, the depth lemma shows that \(\mathscr{K}_f\) is reflexive: \(\Der_{\mathbb{C}} A\) has depth two and its image \(J_f\) is torsion-free. On \(U=V\setminus\{0\}\), contraction gives the isomorphism
\[
 \begin{aligned}
 \omega_U^\vee \simeq \bigwedge\nolimits^2\Theta_U
 & \to \ker(df\colon\Theta_U \to \mathcal{O}_U),\\
 v\wedge w& \mapsto df(v)w-df(w)v.
 \end{aligned}
\]
Reflexive extension identifies \(\mathscr{K}_f\) with \(\omega_A^\vee\). Taking sections in \eqref{eq:df} and using Lemma~\ref{lem:resolution-comparison}, we obtain
\[
 H^0(\widetilde{X},L_{\widetilde{X}}(M))
 \simeq \mathscr{K}_f \simeq \omega_A^\vee.
\]
The identifications agree with restriction away from \(\widetilde{E}\).
\end{proof}

The next lemma expresses the Euler characteristic of the Milnor fiber in terms of the exceptional cycle. The local pieces over the smooth parts of the exceptional curves are finite coverings; the pieces joining them are annuli and contribute zero.

\begin{lemma}\label{lem:milnor-cycle}
Under the hypotheses of Lemma~\ref{lem:logarithmic-differential}\textup{(i)}, put \(D_i^\circ=D_i\setminus ((D-D_i)\cup H)\). Then
\[
 \chi(F_f)=\sum_i z_i\chi(D_i^\circ)
 =Z^2-(K_Y+D)\cdot Z.
\]
\end{lemma}

\begin{proof}
Let \(g_i\) be the genus of \(D_i\), and set
\[
 t_i
 =
 (D-D_i)\cdot D_i.
\]
All intersections are transverse, and \(H\) avoids the nodes. Therefore
\[
 \chi(D_i^\circ)
 =
 2-2g_i-t_i-H\cdot D_i.
\]

Choose disjoint small discs around the deleted points on each \(D_i\). Over the resulting compact core of \(D_i^\circ\), the local equation \(f=x^{z_i}u\), with \(u\) a unit, identifies the nearby fiber with an unramified covering of degree \(z_i\). Its Euler characteristic is thus \(z_i\chi(D_i^\circ)\).

At a node \(D_i\cap D_j\), the local equation can be written, after absorbing a unit into a coordinate, as
\[
 x^{z_i}y^{z_j}=\varepsilon.
\]
Its intersection with a sufficiently small bidisc consists of \(\gcd(z_i,z_j)\) annuli. At a point of \(D_i\cap H\), the corresponding equation is
\[
 x^{z_i}y=\varepsilon,
\]
which gives one annulus. The remaining boundary pieces are collars, and all these pieces have Euler characteristic zero. They are attached along boundary circles, whose Euler characteristic is also zero. Additivity therefore gives
\[
 \chi(F_f)
 =
 \sum_i z_i\chi(D_i^\circ)
 =
 \sum_i z_i
 (2-2g_i-t_i-H\cdot D_i).
\]

Adjunction on each exceptional component gives
\begin{equation*}
\begin{split}
 (K_Y+D)\cdot D_i
 &= (K_Y+D_i)\cdot D_i +(D-D_i)\cdot D_i\\
 &=2g_i-2+t_i.
\end{split}
\end{equation*}
Moreover, \(Z+H\) is the divisor of a function, so
\[
 (Z+H)\cdot D_i=0
\]
for every \(i\). Multiplying by \(z_i\) and summing yields
\[
 Z^2+H\cdot Z=0.
\]
Consequently,
\begin{align}\label{eq:milnor}
 \chi(F_f)
 &=
 \sum_i z_i
 (2-2g_i-t_i-H\cdot D_i)\notag\\
 &=
 -(K_Y+D)\cdot Z-H\cdot Z\notag\\
 &=
 Z^2-(K_Y+D)\cdot Z.
\end{align}
\end{proof}

\subsection{Proof of Proposition~\ref{prop:index}}\label{sec:proof-of-prop-index}

\begin{proof}
Recall that \(A=\mathcal{O}_{V,0}\) and \(\Delta(A)=\Delta(V,0)\). By Lemma~\ref{lem:function-conservation}, replacing \(f\) by \(f+t\ell\) for general sufficiently small \(t\) does not change \(j_f-\mu_f\). We may therefore assume that its image in \(\mathfrak{m}_A/\mathfrak{m}_A^2\) is sufficiently general. Apply Lemma~\ref{lem:logarithmic-differential}\textup{(ii)} to the chosen principalization \(\widetilde{X} \to X\). We compute \(\Delta_{\widetilde{X}}+a_{\widetilde{X}}\) from the cohomology sequence of \eqref{eq:df} and two restriction sequences.

Write \(p_g=h^1(\mathcal{O}_{\widetilde{X}})\) for the geometric genus of \((V,0)\). Using the vanishing of \(H^2\), the cohomology sequence of \eqref{eq:df} gives
\[
 0 \to \mathfrak{m}_A/J_f
 \to H^1(\widetilde{X},L_{\widetilde{X}}(M))
 \to H^1(\widetilde{X},S_{\widetilde{X}}) \to H^1(\widetilde{X},\mathcal{O}_{\widetilde{X}}(-M))
 \to 0.
\]
Since \(\dim_{\mathbb{C}}(\mathfrak{m}_A/J_f)=j_f-1\), it follows that
\[
 h^1(S_{\widetilde{X}})
 =
 h^1(L_{\widetilde{X}}(M))+h^1(\mathcal{O}_{\widetilde{X}}(-M))-(j_f-1).
\]

Next, consider the restriction sequence
\[
 0 \to L_{\widetilde{X}}
 \to L_{\widetilde{X}}(M)
 \to L_{\widetilde{X}}(M)|_M
 \to 0.
\]
The identification of \(H^0(L_{\widetilde{X}}(M))\) in Lemma~\ref{lem:logarithmic-differential} gives the cohomology sequence
\begin{multline*}
 0
 \to 
 \omega_A^\vee/H^0(\widetilde{X},L_{\widetilde{X}})
 \to H^0(M,L_{\widetilde{X}}(M)|_M)
 \to H^1(\widetilde{X},L_{\widetilde{X}})\\
 \to H^1(\widetilde{X},L_{\widetilde{X}}(M))
 \to H^1(M,L_{\widetilde{X}}(M)|_M)
 \to 0.
\end{multline*}
Taking the alternating sum of dimensions yields
\[
 h^1(L_{\widetilde{X}})
 -h^1(L_{\widetilde{X}}(M))
 =
 \chi(L_{\widetilde{X}}(M)|_M)-a_{\widetilde{X}}.
\]

Finally, the sequence
\[
 0 \to \mathcal{O}_{\widetilde{X}}(-M)
 \to \mathcal{O}_{\widetilde{X}}
 \to \mathcal{O}_M
 \to 0
\]
induces
\[
 0
 \to \mathbb{C}
 \to H^0(M,\mathcal{O}_M)
 \to H^1(\widetilde{X},\mathcal{O}_{\widetilde{X}}(-M))
 \to H^1(\widetilde{X},\mathcal{O}_{\widetilde{X}})
 \to H^1(M,\mathcal{O}_M)
 \to 0,
\]
where the first term is \(A/\mathfrak{m}_A\). Consequently,
\[
 p_g-h^1(\mathcal{O}_{\widetilde{X}}(-M)) = 1-\chi(\mathcal{O}_M).
\]

Riemann--Roch applies to the effective Cartier cycle \(M\), without any reducedness assumption. For an invertible sheaf \(\mathcal{L}\) on \(\widetilde{X}\), it gives
\[
 \begin{aligned}
 \chi(\mathcal{O}_M)
 &=-\frac12 M\cdot(M+K_{\widetilde{X}}),\\
 \chi(\mathcal{L}|_M)-\chi(\mathcal{O}_M)
 &=c_1(\mathcal{L})\cdot M.
 \end{aligned}
\]
Applying the second formula to
\[
 \mathcal{L}=L_{\widetilde{X}}(M)
 =\mathcal{O}_{\widetilde{X}}
   (M-K_{\widetilde{X}}-\widetilde{E}),
\]
and combining the three cohomology identities, we obtain
\begin{align}\label{eq:cycleindex}
 \Delta_{\widetilde{X}}+a_{\widetilde{X}}
 &=
 j_f-1+(p_g-h^1(\mathcal{O}_{\widetilde{X}}(-M)))
 +h^1(L_{\widetilde{X}})
 -h^1(L_{\widetilde{X}}(M))
 +a_{\widetilde{X}}\notag\\
 &=
 j_f+\chi(L_{\widetilde{X}}(M)|_M)
 -\chi(\mathcal{O}_M)\notag\\
 &=
 j_f+M^2-(K_{\widetilde{X}}+\widetilde{E})\cdot M.
\end{align}

Apply Lemma~\ref{lem:milnor-cycle} with \((Y,D,Z)=(\widetilde{X},\widetilde{E},M)\). Combining \eqref{eq:correction}, \eqref{eq:cycleindex}, and \eqref{eq:milnor}, we conclude that
\[
 \Delta(V,0)=\Delta(A)
 =
 \Delta_{\widetilde{X}}+a_{\widetilde{X}}
 =
 j_f+\chi(F_f)
 =
 1+j_f-\mu_f.
 \qedhere
\]
\end{proof}

\section{Differential forms and the residue representation}
\label{sec:forms}

This section constructs a representation from the reflexive de Rham complex and residue maps. We then compare character multiplicities with the Wahl defect of the quotient, using the function-index formula of Section~\ref{sec:index}. Smooth and rational double point covers are treated at the end of the section.

We work in the complex analytic category. We choose sufficiently small Stein representatives of the singularities and take their resolutions. When sections are identified with modules over a local analytic algebra, they are understood as direct limits over smaller representatives. Coherent cohomology in positive degrees is supported at the singular point and is unchanged after sufficiently small shrinking. All representatives can be chosen invariant under a given finite group action.

For a resolution of an isolated normal surface singularity over a Stein representative, the theorem on the dimension of the fibers and the Leray spectral sequence give
\begin{equation}\label{eq:coherent-vanishing}
 H^q(Y,\mathcal{E})=0\quad(q \geq 2)
\end{equation}
for every coherent sheaf \(\mathcal{E}\). The groups \(H^1(Y,\mathcal{E})\) are finite dimensional. We also use Grauert--Riemenschneider vanishing~\cite{GrauertRiemenschneider1970}:
\begin{equation}\label{eq:GR}
 H^1(Y,\omega_Y)=0.
\end{equation}

\subsection{Canonical characters and logarithmic forms}

Let \((V,0)\) be a normal \(\mathbb{Q}\)-Gorenstein surface singularity, and put \(A=\mathcal{O}_{V,0}\). Let \(q\colon(\widetilde{V},0) \to (V,0)\) be its canonical cover, and put \(B=\mathcal{O}_{\widetilde{V},0}\). We use this cover, the covering group \(G=\mu_r\), and the canonical character of Section~\ref{qg:sec:relative-cover}. For \(r=1\), the conventions are \(B=A\), \(G=\{1\}\), and \(\kappa=1\). The cover is Gorenstein and is \'{e}tale over the punctured base. We recall the character conventions needed below.

Let \(\kappa\) be the character of \(\omega_B/\mathfrak{m}_B\omega_B\). A generator of \(\omega_B\) can be chosen to be an eigenvector: apply the \(\kappa\)-isotypic projector to any lift of a nonzero element of this quotient. The resulting section \(\eta\) still generates \(\omega_B\), by Nakayama's lemma, and satisfies
\[
 g\eta=\kappa(g)\eta\quad(g \in G).
\]
If \(r>1\), then \(\kappa \neq \mathbf{1}\). Indeed, an invariant generator would descend over the punctured quotient to a nowhere vanishing canonical form. Reflexive extension would then give \(\omega_A \simeq A\), contrary to \(r>1\).

For a character \(\chi\) and a \(G\)-representation \(M\), put
\[
 M_\chi=\{v \in M:g v=\chi(g)v\text{ for all }g \in G\}.
\]
Since the relevant characters are one dimensional, \(\mult_\chi(M)=\dim M_\chi\). We use the same multiplicity notation for virtual representations. Taking invariants and taking isotypic components are exact functors over \(\mathbb{C}\).

For the logarithmic comparison and residue construction below, let \((\widetilde{V},0)\) be a Gorenstein normal complex surface singularity which is not a rational double point, put \(B=\mathcal{O}_{\widetilde{V},0}\), and let a finite group \(G\) act on \((\widetilde{V},0)\). Freeness away from the singular point is not required until Section~\ref{sec:descent}.

The action lifts to the minimal good resolution \((Y,F)\). Write
\[
 F=\sum_iF_i,\quad U=Y\setminus F,\quad
 S=\Theta_Y(-\log F),\quad L=\bigwedge^2S.
\]

\begin{definition}
The \emph{equivariant Wahl defect} of \((Y,F)\) is
\[
 \mathcal{D}_Y :=[H^1(Y,\mathcal{O}_Y)]-[H^1(Y,S)]+[H^1(Y,L)] \in R(G),
\]
where \(S=\Theta_Y(-\log F)\), \(L=\bigwedge^2 S\), and \(R(G)\) is the complex representation ring of \(G\). Here \([M]\) denotes the class of a finite-dimensional complex \(G\)-representation \(M\), and \(\mathbf{1}\) denotes the class of the one-dimensional trivial representation.
\end{definition}

Choose the eigenvector \(\eta\) as above. Its divisor is supported on \(F\), so write
\[
 \divr_Y(\eta)=-Z.
\]
By Wahl~\cite[Lemma~3.6]{Wahl1985}, \(Z \geq F\). Hence
\begin{equation}\label{eq:D}
 D=Z-F \geq 0.
\end{equation}
All divisors in \eqref{eq:D} are \(G\)-invariant.

Divisorial sheaves have their natural linearizations as subsheaves of meromorphic functions. Differential forms have the linearizations induced by the action on \(Y\). With these conventions,
\begin{equation}\label{eq:twists}
 \begin{aligned}
 \omega_Y& \simeq \mathcal{O}_Y(-Z)\otimes\mathbb{C}_\kappa,\\
 L\otimes\mathbb{C}_\kappa& \simeq \mathcal{O}_Y(D),\\
 S\otimes\mathbb{C}_\kappa& \simeq \Omega_Y^1(\log F)(D),\\
 \omega_Y(Z)& \simeq \mathcal{O}_Y\otimes\mathbb{C}_\kappa.
 \end{aligned}
\end{equation}
The third isomorphism is contraction with \(\eta\). For example, at a smooth point of \(F\), write \(F=(x=0)\) and \(\eta=u x^{-z}dx\wedge dy\), where \(u\) is a unit. Contraction of the logarithmic vector fields \(x\partial_x,\partial_y\) gives generators of \(x^{-(z-1)}\Omega_Y^1(\log F)\). The computation at a node uses the basis \(x\partial_x,y\partial_y\) and gives the same formula.

We shall also use the lifting of derivations to the minimal good resolution, in the form
\begin{equation}\label{eq:der-lifting}
 H^0(Y,S)=\Der_{\mathbb{C}}B=H^0(U,\Theta_U).
\end{equation}
See Wahl~\cite[Section~3, (3.5)]{Wahl1985}. Consequently, contraction with \(\eta\) identifies
\begin{equation}\label{eq:one-form-sections}
 H^0(Y,\Omega_Y^1(\log F)(D))
 =H^0(U,\Omega_U^1).
\end{equation}
For every effective exceptional divisor \(T\), normality gives
\begin{equation}\label{eq:Hartogs}
 H^0(Y,\mathcal{O}_Y(T))=B.
\end{equation}
Indeed, a section restricts to a holomorphic function on the punctured normal surface and therefore extends across the singular point. Its pullback is regular on \(Y\).

\subsection{The logarithmic de Rham comparison}
\label{sec:comparison}

We consider the complex
\begin{equation}\label{eq:polar-complex}
 \mathscr{C}_D^\bullet:
 \mathcal{O}_Y(D) \xrightarrow{d}\Omega_Y^1(\log F)(D)
 \xrightarrow{d}\omega_Y(F+D),
\end{equation}
in degrees \(0,1,2\). Exterior differentiation preserves the allowed pole orders in the logarithmic basis, so this is a complex. Its differentials are \(\mathbb{C}\)-linear differential operators. Hypercohomology is taken in the category of analytic sheaves of complex vector spaces.

\begin{lemma}\label{lem:log-comparison}
The inclusion
\[
 \Omega_Y^\bullet(\log F) \hookrightarrow \mathscr{C}_D^\bullet
\]
is a \(G\)-equivariant quasi-isomorphism. Consequently,
\[
 \mathbb{H}^k(Y,\mathscr{C}_D^\bullet) \simeq H^k(U,\mathbb{C}).
\]
\end{lemma}

\begin{proof}
We first prove that the quotient complex is exact. At a node, choose coordinates with
\[
 F=(xy=0),\quad D=a(x=0)+b(y=0),\quad a,b \geq 0.
\]
Write forms in the basis \(d\log x,d\log y\). The quotient is represented by convergent Laurent series whose monomials \(x^iy^j\) satisfy
\[
 i \geq -a,\quad j \geq -b,\quad i<0\text{ or }j<0.
\]
On the exterior algebra multiplying such a monomial, the differential is exterior multiplication by \(i d\log x+j d\log y\). If \(i<0\), contraction with \(x\partial_x\), divided by \(i\), is a homotopy between the identity and zero. For the remaining monomials, \(i \geq 0\) and \(j<0\), use contraction with \(y\partial_y\), divided by \(j\). In both cases,
\[
 dh+hd=\mathrm{id}.
\]
The negative powers are bounded, and the divisors \(i,j\) used here are nonzero integers. The operations therefore preserve convergence after shrinking the coordinate polydisc.

At a smooth point of \(F\), use the same homotopy in the normal coordinate. The tangential differential anticommutes with the contraction, so the identity remains valid. Away from \(F\), the quotient is zero. This proves exactness of the quotient complex.

For completeness, the ordinary logarithmic complex computes \(Rj_*\mathbb{C}_U\), where \(j\colon U \hookrightarrow Y\). Locally at a point where \(k\) branches of \(F\) meet, the complement is a product of \(k\) punctured discs and \(2-k\) discs. Its de Rham cohomology is the exterior algebra generated by the classes of \(d\log x_1,\ldots,d\log x_k\). The logarithmic complex has the same local cohomology: nonconstant monomials are contracted as above, and the holomorphic Poincar\'{e} lemma applies in the remaining coordinates. This proves the local comparison, and hence the claimed hypercohomology identification. The inclusions commute with the \(G\)-action. Equivariance of the local homotopies is not needed for an equivariant quasi-isomorphism.
\end{proof}

By \eqref{eq:one-form-sections}, \eqref{eq:Hartogs}, and \eqref{eq:twists}, the complex of global sections is
\begin{equation}\label{eq:reflexive}
 V^\bullet:=H^0(Y,\mathscr{C}_D^\bullet)
 \simeq [B \xrightarrow{d}\Omega_B^{[1]} \xrightarrow{d}\omega_B].
\end{equation}
The first and second cohomology spaces of this complex are
\[
 \mathcal{P}=H^1(V^\bullet)=\frac{\ker(d\colon\Omega_B^{[1]} \to \omega_B)}{dB},
 \quad \mathcal{Q}=H^2(V^\bullet)=\frac{\omega_B}{d\Omega_B^{[1]}}.
\]
We shall prove their finite dimensionality in the next proposition.

\begin{proposition}\label{prop:comparison}
The spaces \(\mathcal{P}\) and \(\mathcal{Q}\) are finite-dimensional \(G\)-representations, and
\begin{equation}\label{eq:comparison}
 [\mathbb{C}_\kappa]\mathcal{D}_Y=\mathbf{1}+[\mathcal{Q}]-[\mathcal{P}].
\end{equation}
\end{proposition}

\begin{proof}
We compare the two rows of coherent cohomology with the cohomology of the link, whose contributions in degrees one and two cancel. The hypercohomology spectral sequence is
\[
 E_1^{p,q}=H^q(Y,\mathscr{C}_D^p)
 \Rightarrow H^{p+q}(U,\mathbb{C}).
\]
By \eqref{eq:coherent-vanishing}, only the rows \(q=0,1\) can be nonzero. Let
\[
 W^\bullet=[H^1(Y,\mathscr{C}_D^0) \xrightarrow{d}
 H^1(Y,\mathscr{C}_D^1) \xrightarrow{d}
 H^1(Y,\mathscr{C}_D^2)],
\]
and put \(K=H^0(W^\bullet)\), \(M=H^1(W^\bullet)\), and \(N=H^2(W^\bullet)\). These three spaces are finite dimensional. The \(E_2\)-page has the following two rows:
\[
\begin{array}{c|ccc}
q=1&K&M&N\\
q=0&\mathbb{C}&\mathcal{P}&\mathcal{Q}\\ \hline
 &p=0&p=1&p=2.
\end{array}
\]
The only potentially nonzero differential after \(E_1\) is
\[
 d_2^{0,1}\colon K \to \mathcal{Q}.
\]
In particular, \(N\) survives to \(E_\infty\). It is the only term of total degree three, and therefore
\[
 N \simeq H^3(U,\mathbb{C}).
\]
The punctured surface retracts onto its connected, oriented link. Holomorphic automorphisms preserve the orientation. Hence
\begin{equation}\label{eq:orientation}
 N \simeq \mathbb{C}_{\mathbf{1}}.
\end{equation}

The filtrations on total cohomology in degrees one and two give the exact sequence
\begin{equation}\label{eq:edge}
\begin{split}
0 \to \mathcal{P} \to H^1(U,\mathbb{C}) \to K
 \xrightarrow{d_2} \mathcal{Q} \to H^2(U,\mathbb{C}) \to M \to 0.
\end{split}
\end{equation}
For example, the quotient of \(H^1(U,\mathbb{C})\) by \(\mathcal{P}\) is \(\ker d_2\), while the kernel of \(H^2(U,\mathbb{C}) \to M\) is \(\coker d_2\). This also proves that \(\mathcal{P}\) and \(\mathcal{Q}\) are finite dimensional.

Poincar\'{e} duality on the link is equivariant and gives
\[
 H^2(U,\mathbb{C}) \simeq H^1(U,\mathbb{C})^*.
\]
The representation \(H^1(U,\mathbb{C})\) is the complexification of a real representation. Averaging a positive definite real inner product over \(G\) and complexifying gives an equivariant isomorphism with its dual. Thus
\begin{equation*}%\label{eq:top-cancel}
 [H^1(U,\mathbb{C})]=[H^2(U,\mathbb{C})]\quad\text{in }R(G).
\end{equation*}
Taking the alternating sum in \eqref{eq:edge} now yields
\[
 [K]-[M]=[\mathcal{Q}]-[\mathcal{P}].
\]
Together with \eqref{eq:orientation}, this gives
\[
 \chi_G(W^\bullet)=\mathbf{1}+[\mathcal{Q}]-[\mathcal{P}].
\]
On the other hand, \eqref{eq:twists} identifies the three terms of \(W^\bullet\) with \(\mathbb{C}_\kappa\) tensored with
\[
 H^1(Y,L),\quad H^1(Y,S),\quad H^1(Y,\mathcal{O}_Y),
\]
in that order. Their alternating sum is \([\mathbb{C}_\kappa]\mathcal{D}_Y\), proving \eqref{eq:comparison}.
\end{proof}

\begin{remark}
Equation~\eqref{eq:comparison} is an identity in \(R(G)\). Its proof does not imply that \([\mathcal{Q}]-[\mathcal{P}]\) is effective. The residue argument in the next subsection establishes this additional assertion without choosing an injection \(\mathcal{P} \to \mathcal{Q}\).
\end{remark}

\subsection{Closed one-forms and residues}
\label{sec:residues}

Let \(\nu\colon\widetilde{F}=\coprod_iF_i \to F\) be the normalization. Put
\[
 W_F=H^0(F,\omega_F).
\]
A section of \(\omega_F\) is a collection of meromorphic differentials on the components, with at most simple poles at the preimages of the nodes, whose two residues at each node have sum zero. This description gives a natural injection \(\bigoplus_iH^0(F_i,\omega_{F_i}) \hookrightarrow W_F\).

\begin{lemma}\label{lem:closed-one-forms}
Restriction induces a \(G\)-equivariant injection
\[
 \iota\colon\mathcal{P} \hookrightarrow \bigoplus_iH^0(F_i,\omega_{F_i}) \hookrightarrow W_F.
\]
\end{lemma}

\begin{proof}
A class in \(\mathcal{P}\) is represented by a closed holomorphic one-form \(\alpha\) on the punctured surface. By the extension theorem for closed one-forms, \(\alpha\) extends to a holomorphic one-form \(\widetilde{\alpha}\) on \(Y\); see Jörder~\cite[Proposition~1.2]{JorderExtension}. The extension is unique and remains closed, since its exterior derivative vanishes on the dense open set \(U\). Restricting to the components of \(F\) therefore defines an element of \(\bigoplus_iH^0(F_i,\omega_{F_i})\).

If \(\alpha=df\) for \(f \in B\), the pullback of \(f\) is constant on \(F\), and \(df\) restricts to zero on every \(F_i\). Thus restriction is well defined on \(\mathcal{P}\).

Suppose that the restriction of \(\widetilde{\alpha}\) to every component is zero. Every one-cycle on \(F\) is represented by a finite sum of piecewise smooth loops whose arcs lie on its components. The integral of \(\widetilde{\alpha}\) along each arc is zero. Thus all periods on \(F\) vanish, including periods corresponding to cycles in the dual graph. A sufficiently small resolution neighborhood retracts onto \(F\), so all periods on \(Y\) vanish as well. Integrating \(\widetilde{\alpha}\) from a fixed point gives a single-valued holomorphic function \(f\) on \(Y\) with \(df=\widetilde{\alpha}\). Since \(H^0(Y,\mathcal{O}_Y)=B\), this function comes from \(B\). The original class is therefore zero in \(\mathcal{P}\).

Uniqueness of extension and naturality of restriction show that \(\iota\) is equivariant.
\end{proof}

We henceforth identify \(\mathcal{P}\) with its image under \(\iota\). The adjunction sequence
\begin{equation}\label{eq:adjunction}
 0 \to \omega_Y \to \omega_Y(F) \xrightarrow{\Res} \omega_F \to 0
\end{equation}
is equivariant. At a node with \(F=(xy=0)\), the section
\[
 h(x,y)\frac{dx\wedge dy}{xy}
\]
has residues \(h(0,y)dy/y\) and \(-h(x,0)dx/x\) on the two branches. The opposite signs are part of the residue map and ensure equivariance when an automorphism exchanges the branches.

The following lemma applies to any good resolution, without a Gorenstein assumption.

\begin{lemma}\label{lem:exact-residue}
Let \((Y,F)\) be a good resolution of a normal complex surface singularity, and put \(U=Y\setminus F\). If \(\beta \in H^0(Y,\omega_Y(F))\) satisfies \(\beta=d\alpha\) on \(U\) for some \(\alpha \in H^0(U,\Omega_U^1)\), then \(\Res(\beta)=0\).
\end{lemma}

\begin{proof}
Fix a component \(F_i\), and let \(\gamma\) be a smooth closed loop in \(F_i\) avoiding the nodes. The normal circle bundle of small radius over \(\gamma\) is a closed oriented torus \(T_{\gamma,\epsilon} \subset U\). Orient it so that integration along the normal circle gives the usual factor \(2\pi i\). The residue integral formula gives
\[
 2\pi i\int_\gamma\Res_{F_i}(\beta)
 =\lim_{\epsilon \to 0}\int_{T_{\gamma,\epsilon}}\beta.
\]
For each sufficiently small \(\epsilon\), Stokes' theorem on this closed torus gives
\[
 \int_{T_{\gamma,\epsilon}}\beta
 =\int_{T_{\gamma,\epsilon}}d\alpha=0.
\]
Thus every period of \(\Res_{F_i}(\beta)\) away from the nodes is zero.

The differential \(\Res_{F_i}(\beta)\) has at most simple poles at the nodes. Applying the preceding calculation to small loops around the nodes shows that the residues of all these poles vanish. The poles are therefore removable. We obtain a holomorphic differential on the compact smooth curve \(F_i\) with all periods zero. It has a single-valued holomorphic primitive, which is constant by compactness. Hence \(\Res_{F_i}(\beta)=0\). This holds for every component, and proves the assertion.
\end{proof}

Using the natural map to \(\mathcal{Q}=H^0(U,\omega_U)/dH^0(U,\Omega_U^1)\), define
\[
 I_{\mathrm{reg}}=\im(H^0(Y,\omega_Y) \to \mathcal{Q}),\quad
 I_{\log}=\im(H^0(Y,\omega_Y(F)) \to \mathcal{Q}).
\]
In particular,
\begin{equation}\label{eq:regular-quotient}
 I_{\mathrm{reg}}=
 \frac{H^0(Y,\omega_Y)}
 {H^0(Y,\omega_Y)\cap dH^0(U,\Omega_U^1)}.
\end{equation}
This formula allows primitives with poles along \(F\).

\begin{proposition}\label{prop:residue-sequence}
Residue induces an exact sequence of \(G\)-representations
\begin{equation}\label{eq:residue-sequence}
 0 \to I_{\mathrm{reg}} \to I_{\log} \xrightarrow{\Res} W_F \to 0.
\end{equation}
\end{proposition}

\begin{proof}
Two logarithmic two-forms represent the same class in \(\mathcal{Q}\) precisely when their difference is \(d\alpha\) for a holomorphic one-form \(\alpha\) on \(U\). Lemma~\ref{lem:exact-residue} shows that they have the same residue. Thus residue is well defined on \(I_{\log}\).

By \eqref{eq:adjunction}, a logarithmic two-form has zero residue if and only if it is regular on \(Y\). Hence the kernel of the induced map on \(I_{\log}\) is exactly \(I_{\mathrm{reg}}\). Finally, \eqref{eq:GR} and the cohomology sequence of \eqref{eq:adjunction} show that
\[
 H^0(Y,\omega_Y(F)) \to H^0(F,\omega_F)
\]
is surjective. The induced map from \(I_{\log}\) is therefore surjective. All maps commute with \(G\), and all spaces are finite dimensional by Proposition~\ref{prop:comparison} and compactness of \(F\).
\end{proof}

\begin{theorem}\label{thm:effective}
With the notation above, put
\[
 \mathcal{R}_{\mathrm{res}} := I_{\mathrm{reg}}\oplus(W_F/\mathcal{P})\oplus(\mathcal{Q}/I_{\log}).
\]
Then \(\mathcal{R}_{\mathrm{res}}\) is a finite-dimensional \(G\)-representation and
\begin{equation}\label{eq:effective}
 [\mathbb{C}_\kappa]\mathcal{D}_Y=\mathbf{1}+[\mathcal{R}_{\mathrm{res}}].
\end{equation}
\end{theorem}

\begin{proof}
Lemma~\ref{lem:closed-one-forms} defines the quotient \(W_F/\mathcal{P}\). Proposition~\ref{prop:residue-sequence} gives
\[
 [\mathcal{Q}]=[I_{\mathrm{reg}}]+[W_F]+[\mathcal{Q}/I_{\log}].
\]
Subtracting \([\mathcal{P}]\) and applying Proposition~\ref{prop:comparison} proves \eqref{eq:effective}.
\end{proof}

\begin{remark}\label{rem:complex}
The representation in Theorem~\ref{thm:effective} is the cohomology of a naturally defined finite-dimensional complex. Put
\[
 \mathscr{K}^{-1}=\mathcal{P},\quad
 \mathscr{K}^0=I_{\mathrm{reg}}\oplus W_F\oplus(\mathcal{Q}/I_{\log}),
\]
with differential \(p \mapsto (0,\iota(p),0)\). Lemma~\ref{lem:closed-one-forms} gives
\[
 H^{-1}(\mathscr{K}^\bullet)=0,\quad
 H^0(\mathscr{K}^\bullet)=\mathcal{R}_{\mathrm{res}}.
\]
No splitting of \eqref{eq:residue-sequence} is used in this construction.
\end{remark}

\subsubsection{The maximal ideal and the residue representation}
\label{subsec:maximal-ideal}

We also record the representation obtained from the maximal ideal. Choose a \(G\)-equivariant principalization \(p\colon Y' \to (\Specan B,0)\) such that its reduced exceptional divisor \(F'\) has simple normal crossings, and write
\[
 \mathfrak{m}_B\mathcal{O}_{Y'}=\mathcal{O}_{Y'}(-D_{\mathfrak{m}}),\quad
 D_{\mathfrak{m}}=\sum_i m_iF'_i,\quad m_i>0.
\]
Here and below a local analytic representative is understood. The cohomology spaces \(\mathcal{P},\mathcal{Q}\) of the reflexive de Rham complex and the image \(I_{\mathrm{reg}}\) of regular two-forms are unchanged by replacing \(Y\) by \(Y'\). For \(I_{\mathrm{reg}}\), this follows from the invariance of regular two-forms under a proper modification between smooth surfaces: pullback extends regularly, while descent holds off a finite set and then by Hartogs' theorem. Closed one-forms likewise pull back regularly.

Put \(W_{F'}=H^0(F',\omega_{F'})\). If \(\alpha\) is a closed reflexive one-form, let
\[
 w_{\mathfrak{m}}([\alpha])=(m_i\alpha|_{F'_i})_i \in W_{F'}.
\]
Each component is holomorphic, so the node residues are zero. Exact forms restrict to zero on the exceptional curves. This map is therefore well defined on \(\mathcal{P}\) and is equivariant. It is injective: the coefficients \(m_i\) are positive, and vanishing of all restrictions implies exactness by the period argument of Lemma~\ref{lem:closed-one-forms}.

Proposition~\ref{prop:residue-sequence}, applied on \(Y'\), gives a canonical injection \(W_{F'} \to \mathcal{Q}/I_{\mathrm{reg}}\). Let \(j_{\mathfrak{m}}\colon\mathcal{P} \to \mathcal{Q}/I_{\mathrm{reg}}\) be the composite of \(w_{\mathfrak{m}}\) with the residue inclusion, and define
\begin{equation}\label{eq:Em}
 E_{\mathfrak{m}} :=\{u \in \mathcal{Q}:u+I_{\mathrm{reg}} \in j_{\mathfrak{m}}(\mathcal{P})\}.
\end{equation}
Define
\begin{equation}\label{eq:Hm}
 \mathcal{H}_{\mathfrak{m}} :=I_{\mathrm{reg}}\oplus(\mathcal{Q}/E_{\mathfrak{m}}).
\end{equation}
This is a finite-dimensional representation. The exact sequence
\[
 0 \to I_{\mathrm{reg}} \to E_{\mathfrak{m}} \to \mathcal{P} \to 0
\]
immediately gives
\begin{equation}\label{eq:Hm-class}
 [\mathcal{H}_{\mathfrak{m}}]=[\mathcal{Q}]-[\mathcal{P}]=[\mathcal{R}_{\mathrm{res}}].
\end{equation}
In particular, no splitting or injection \(\mathcal{P} \to \mathcal{Q}\) is required.

The definition agrees with the construction by logarithmic differentiation of a local generator \(t\) of \(\mathcal{O}_{Y'}(-D_{\mathfrak{m}})\). Indeed, \(d\log t\wedge\alpha\) has residue \((m_i\alpha|_{F'_i})_i\). Replacing \(t\) by \(ut\), for a unit \(u\), changes this two-form by the regular form \(d\log u\wedge\alpha\). Thus its class modulo regular two-forms is independent of \(t\).

For completeness, the subspace \(j_{\mathfrak{m}}(\mathcal{P})\) is independent of the principalization. Under a point blowup, \(D_{\mathfrak{m}}\) pulls back, its coefficients on strict transforms are unchanged, and the restriction of a pulled-back regular one-form to the new exceptional projective line is zero. Pullback of the logarithmic two-form above has precisely these residues. The identifications of regular two-forms, \(\mathcal{Q}\), and \(I_{\mathrm{reg}}\) therefore identify the two images. A common equivariant refinement compares any two choices. Consequently \(E_{\mathfrak{m}}\) and \eqref{eq:Hm} are independent of the choice.

Combining \eqref{eq:Hm-class} with Theorem~\ref{thm:effective} gives
\begin{equation}\label{eq:effective-Hm}
 [\mathbb{C}_\kappa]\mathcal{D}_Y=\mathbf{1}+[\mathcal{H}_{\mathfrak{m}}].
\end{equation}

\subsection{Passage to the quotient}
\label{sec:descent}

Assume now that \(G\) acts freely away from the singular point of \(B\), let \((V,0)=(\widetilde{V},0)/G\), and put \(A=\mathcal{O}_{V,0}=B^G\). We assume that \(A\) is singular and is not a rational double point. This is automatic when \(A\) is non-Gorenstein.

We use the function-index formula of Proposition~\ref{prop:index}. If \((W,0)\) is a normal complex surface singularity which is not a rational double point, \(C=\mathcal{O}_{W,0}\), and \(f \in \mathfrak{m}_C\) defines a reduced curve and has an isolated critical point, then
\begin{equation}\label{eq:function-index}
 \Delta(C)=1+j_C(f)-\mu_C(f),
\end{equation}
where
\[
 j_C(f)=\dim_\mathbb{C} C/df(\Der_\mathbb{C} C),\quad
 \mu_C(f)=\dim_\mathbb{C} H^1(F_f,\mathbb{C}).
\]
Here \(F_f\) is the Milnor fiber of the function on the surface singularity, so it is a real two-dimensional manifold with boundary. Formula \eqref{eq:function-index} was proved in Section~\ref{sec:index}. We now give its equivariant comparison for passage to the quotient.

\begin{lemma}\label{lem:der-descent}
Restriction gives an isomorphism
\[
 (\Der_\mathbb{C} B)^G \simeq \Der_\mathbb{C} A.
\]
\end{lemma}

\begin{proof}
An invariant derivation of \(B\) preserves \(A\). Conversely, a derivation of \(A\) lifts uniquely over the finite \'{e}tale map of punctured surfaces. For each \(b \in B\), its derivative under this lifted vector field is holomorphic on the punctured normal surface and extends across the singular point by Hartogs' theorem. These extensions define a derivation of \(B\), since the Leibniz identity holds on the punctured surface and hence everywhere. Uniqueness of the lift makes it invariant. The two constructions are inverse to one another.
\end{proof}

\subsubsection{An invariant function}

\begin{lemma}\label{lem:invariant-function}
Put \(n=|G|\). There is a function \(f \in \mathfrak{m}_A\) such that
\begin{equation*}%\label{eq:function-divisor}
 \divr_Y(f)=M+H,\quad M=\sum_i m_iF_i,\quad
 m_i>0,\quad n \mid m_i,
\end{equation*}
where \(H\) is reduced, smooth, transverse to \(F\), and disjoint from its nodes. On both \(A\) and \(B\), the function defines a reduced curve and has an isolated critical point. Its Milnor fibers are connected.
\end{lemma}

\begin{proof}
Choose an effective invariant exceptional cycle \(C\) with \(C\cdot F_i<0\) for every component, and such that every coefficient of \(C\) is divisible by \(n\). Such a cycle is obtained from a strictly anti-nef integral cycle by summing its translates and multiplying by \(n\). The invertible sheaf \(\mathcal{O}_Y(-C)\) is relatively ample over the contraction.

The stabilizer of every point acts trivially on the fiber of \(\mathcal{O}_Y(-C)\): this sheaf is the \(n\)-th power of an invariant divisorial sheaf, and stabilizer characters have orders dividing \(n\). It therefore descends to an invertible sheaf on \(Y/G\). A sufficiently high power is relatively generated and separates first jets along the quotient exceptional curves. Pulling back a general section gives an invariant function \(f \in \mathfrak{m}_A\) with exceptional part \(M\), a positive multiple of \(C\). To see transversality, choose the section to have simple zeros on the quotient curves away from points with larger stabilizers. The generic pointwise stabilizer fixes the tangential coordinate, so these simple tangential zeros remain simple after pullback. A general section also avoids the finitely many excluded points.

After shrinking, the function has no critical point away from the singular point on either surface. This can also be checked from the local expressions in Lemma~\ref{lem:logarithmic-differential} and the fact that the quotient map is \'{e}tale away from the singular point. The curve sections are generically reduced Cartier sections of Cohen--Macaulay surfaces, so they are reduced. Their Milnor fibers are connected: in the Milnor fibration of the link complement, a meridian around a reduced branch maps with degree one to the circle. The induced homomorphism to \(\mathbb{Z}\) is therefore surjective, and the fiber has one connected component.
\end{proof}

\begin{lemma}\label{lem:df}
For \(f\) as in Lemma~\ref{lem:invariant-function}, there is an equivariant exact sequence
\begin{equation}\label{eq:df-sequence}
 0 \to L(M) \to S \xrightarrow{df} \mathcal{O}_Y(-M) \to 0.
\end{equation}
\end{lemma}

\begin{proof}
The local calculation in Lemma~\ref{lem:logarithmic-differential} gives the exact sequence after shrinking. Taking determinants identifies the kernel with \(\det S\otimes\mathcal{O}_Y(M)=L(M)\). This identification and the map \(df\) are equivariant because \(f\) is invariant.
\end{proof}

Put \(J=df(\Der_\mathbb{C} B)\), and define
\[
 \mathcal{T}_M=\chi_G(\mathcal{O}_M)-\chi_G(L(M)|_M),
\]
where \(\chi_G\) is the alternating sum of coherent cohomology representations on the compact cycle \(M\).

\begin{lemma}\label{lem:critical}
In \(R(G)\), one has
\begin{equation}\label{eq:critical}
 [B/J]=\mathcal{D}_Y+\mathcal{T}_M.
\end{equation}
\end{lemma}

\begin{proof}
The cohomology sequence of \(\mathcal{O}_Y(-M) \subset \mathcal{O}_Y\) gives
\[
 [H^1(Y,\mathcal{O}_Y)]-[H^1(Y,\mathcal{O}_Y(-M))]
 =[B/H^0(Y,\mathcal{O}_Y(-M))]-\chi_G(\mathcal{O}_M).
\]
The image of \(H^0(Y,S) \to H^0(Y,\mathcal{O}_Y(-M))\) is the critical ideal \(J=df(\Der_\mathbb{C} B)\), by \eqref{eq:der-lifting}. Thus \eqref{eq:df-sequence} gives
\begin{equation*}
\begin{split}
 [H^1(Y,S)]
 &=[H^1(Y,L(M))]+[H^1(Y,\mathcal{O}_Y(-M))]\\
 &\quad-[B/J]+[B/H^0(Y,\mathcal{O}_Y(-M))].
\end{split}
\end{equation*}
Finally, \eqref{eq:twists} and \eqref{eq:Hartogs} give
\[
 H^0(Y,L)=H^0(Y,L(M))=B\otimes\mathbb{C}_{\kappa^{-1}}.
\]
The sequence of \(L \subset L(M)\) therefore implies
\[
 [H^1(Y,L)]-[H^1(Y,L(M))]=\chi_G(L(M)|_M).
\]
Combining these three identities proves \eqref{eq:critical}.
\end{proof}

\subsubsection{The character of the cycle contribution}

We first compute the traces on the component thickenings and on the node terms. For a \(g\)-equivariant coherent sheaf \(\mathcal{E}\) with compact one-dimensional support, write
\[
 \chi_g(\mathcal{E})=\tr(g \mid H^0(\mathcal{E}))
                  -\tr(g \mid H^1(\mathcal{E})).
\]

\begin{lemma}\label{lem:local-trace}
Let a finite group \(G\) of order \(n\) act on a smooth surface \(Y\), preserving a compact simple normal crossing divisor \(F=\sum_iF_i\), and freely on \(Y\setminus F\). Let \(M=\sum_i m_iF_i\) be a \(G\)-invariant divisor with \(m_i>0\) and \(n \mid m_i\). Write \(L=\bigwedge^2\Theta_Y(-\log F)\), and give \(L(M)=L\otimes\mathcal{O}_Y(M)\) its natural linearization. Fix \(g \in G\), \(g \neq 1\).
\begin{enumerate}
\item If \(g\) preserves \(F_i\), then
\[
 \chi_g(\mathcal{O}_{m_iF_i})=\chi_g(L(M)|_{m_iF_i}).
\]
\item If \(g\) fixes a node \(p \in F_i\cap F_j\), put
\[
 N_p=\mathcal{O}_{Y,p}/(x^{m_i},y^{m_j}),
\]
where \(F_i=(x=0)\) and \(F_j=(y=0)\). Give \(N_p\) the action induced by local functions, with an additional minus sign when \(g\) exchanges the branches. Then
\[
 \tr(g \mid N_p)=\tr(g \mid L(M)\otimes N_p).
\]
\end{enumerate}
\end{lemma}

\begin{proof}
Since \(M\) is \(n\) times an invariant integral divisor, \(\mathcal{O}_Y(M)\) has trivial stabilizer action on every fiber. We use eigenvalues for the induced action on local functions and their differentials. In particular, the conormal eigenvalue along \(C=(x=0)\) is the eigenvalue on the class of \(x\) in \(\mathcal{O}_Y(-C)|_C\).

\smallskip
\noindent\emph{Case 1: \(g\) preserves \(C=F_i\) and acts nontrivially on it.} Put \(m=m_i\). The filtration of \(\mathcal{O}_{mC}\) by powers of the ideal of \(C\) has quotients \(\mathcal{O}_C(-kC)\), \(0 \leq k<m\). At a fixed point \(p\) of \(g|_C\), let \(\mu \neq 1\) be the eigenvalue on the cotangent line of \(C\), \(\lambda\) the conormal eigenvalue, and \(\tau\) the character on the fiber of \(L(M)\). The holomorphic Lefschetz formula on \(C\)~\cite{AtiyahBott1967} gives the difference of local contributions
\[
 \frac{1-\tau}{1-\mu}\sum_{k=0}^{m-1}\lambda^k.
\]
At a smooth point of \(F\), one has \(\lambda \neq 1\). Otherwise finite-order linearization would give a fixed transverse curve, contrary to freeness on \(Y\setminus F\). Since the order of \(\lambda\) divides \(n\) and \(n \mid m\), the sum is zero. At a node, both branches are preserved. The logarithmic vector fields \(x\partial_x,y\partial_y\) have trivial action on their fibers, so \(\tau=1\). Thus all local contributions vanish. If \(g|_C\) has no fixed point, the Lefschetz formula gives zero directly.

\smallskip
\noindent\emph{Case 2: \(g\) fixes \(C=F_i\) pointwise.} Its conormal character \(\lambda\) is constant and nontrivial. The action on the fiber of the logarithmic determinant, and hence of \(L(M)\), is trivial at a general point and therefore on all of \(C\). On the \(k\)-th quotient, Riemann--Roch gives
\[
 \chi_g(\mathcal{O}_C(-kC))-\chi_g(L(M)|_C\otimes\mathcal{O}_C(-kC))
 =-\lambda^k\deg(L(M)|_C).
\]
The conormal degrees cancel in this difference. Summing over \(0 \leq k<m_i\) gives zero because \(n \mid m_i\). This proves the first assertion.

\smallskip
\noindent\emph{Case 3: \(g\) fixes a node and preserves both branches.} The fiber character of the logarithmic determinant is one, as in Case 1, so the same holds for \(L(M)\). Averaging a local frame of \(L(M)\) over the cyclic group generated by \(g\) gives an invariant frame. Thus \(N_p\) and \(L(M)\otimes N_p\) are isomorphic as representations, and their traces agree.

\smallskip
\noindent\emph{Case 4: \(g\) fixes a node and exchanges the branches.} Invariance of \(M\) gives \(m_i=m_j=m\). On the cotangent space, write
\[
 x \mapsto \alpha y,\quad y \mapsto \beta x.
\]
The associated graded of \(\mathcal{O}_{Y,p}/(x^m,y^m)\) has basis \(x^ay^b\), \(0 \leq a,b<m\). Only the terms with \(a=b\) contribute to the trace of the action induced by local functions, which is
\[
 \sum_{a=0}^{m-1}(\alpha\beta)^a.
\]
Here \(\alpha\beta \neq 1\), since otherwise linearization would give a fixed curve transverse to both branches. The order of \(\alpha\beta\) divides \(n\), and \(n \mid m\), so this sum is zero. The additional sign in the action on \(N_p\) does not change the vanishing. On the logarithmic tangent fiber the two basis vectors are exchanged, so the fiber character of \(L(M)\) is \(-1\). A local eigenvector frame of \(L(M)\) identifies its tensor action with this scalar times the action on \(N_p\). Its trace is therefore also zero. This proves the second assertion.
\end{proof}

\begin{lemma}\label{lem:trace}
For every \(g \in G\) with \(g \neq 1\),
\[
 \tr(g \mid \mathcal{T}_M)=0.
\]
\end{lemma}

\begin{proof}
The normalization sequence of the nonreduced cycle is
\[
 0 \to \mathcal{O}_M \to \bigoplus_i\mathcal{O}_{m_iF_i}
 \to \bigoplus_{p \in F_i\cap F_j}N_p \to 0.
\]
The last map is the difference of the branch restrictions. Consequently, an interchange of branches contributes the minus sign used in Lemma~\ref{lem:local-trace}. Tensor this sequence by \(L(M)\) and subtract equivariant Euler characteristics. Summands permuted without being fixed by \(g\) have zero trace. For the fixed summands, the two assertions of Lemma~\ref{lem:local-trace} give zero difference. Hence \(\tr(g \mid \mathcal{T}_M)=0\).
\end{proof}

Let \(F_f\) be the Milnor fiber of \(f\) on \(B\). Lemma~\ref{lem:milnor-cycle} gives
\begin{equation*}%\label{eq:milnor-cover}
 \chi(F_f)=M^2-(K_Y+F)\cdot M.
\end{equation*}
On the other hand, Riemann--Roch for an invertible sheaf on the compact cycle \(M\) gives
\begin{equation*}
\begin{split}
 \dim\mathcal{T}_M
 &=\chi(\mathcal{O}_M)-\chi(L(M)|_M)\\
 &=-(M-K_Y-F)\cdot M\\
 &=-\chi(F_f).
\end{split}
\end{equation*}

Choose an invariant Milnor representative. The action on \(F_f\) is free. A finite equivariant cell decomposition shows that the Euler representation of \(F_f\) has zero character at every nonidentity element of \(G\). Since \(F_f\) is connected and has cohomology only in degrees zero and one, Lemma~\ref{lem:trace} and the preceding dimension calculation imply
\begin{equation}\label{eq:milnorrep}
 \mathcal{T}_M=[H^1(F_f,\mathbb{C})]-\mathbf{1},\quad
 \mathcal{D}_Y=[B/J]-[H^1(F_f,\mathbb{C})]+\mathbf{1}.
\end{equation}
Equality follows because the characters agree at every group element.

\begin{proposition}\label{prop:transfer}
Under the hypotheses of this subsection,
\[
 \Delta(A)=\mult_{\mathbf{1}}(\mathcal{D}_Y).
\]
\end{proposition}

\begin{proof}
By Lemma~\ref{lem:der-descent} and exactness of invariants,
\[
 (B/df(\Der_\mathbb{C} B))^G
 \simeq A/df(\Der_\mathbb{C} A).
\]
The quotient \(F_f/G\) is the Milnor fiber of \(f\) on \(A\), and the finite covering gives
\[
 H^1(F_f/G,\mathbb{C}) \simeq H^1(F_f,\mathbb{C})^G.
\]
Taking trivial multiplicities in \eqref{eq:milnorrep} therefore yields
\[
 \mult_{\mathbf{1}}(\mathcal{D}_Y)=1+j_A(f)-\mu_A(f).
\]
Formula~\eqref{eq:function-index} identifies this with \(\Delta(A)\).
\end{proof}

\subsection{Smooth and rational double point covers}
\label{sec:quotient-covers}

The pole condition \(Z \geq F\) is essential in Proposition~\ref{prop:comparison}. We first determine the corresponding complex when the cover is smooth or a rational double point. We then compute the quotient defect by a finite-group character calculation.

\subsubsection{The logarithmic complex with vanishing along the boundary}

\begin{lemma}\label{lem:negative-log}
Let \(F\) be a simple normal crossing divisor on a smooth surface \(Y\), let \(T\) be an effective divisor with support \(F\) and strictly positive coefficients, and let \(j\colon Y\setminus F \hookrightarrow Y\). Then
\[
 \Omega_Y^\bullet(\log F)(-T) \simeq j_!\mathbb{C}_{Y\setminus F}.
\]
\end{lemma}

\begin{proof}
Away from \(F\), the assertion is the holomorphic Poincar\'{e} lemma. At a smooth point of \(F=(x=0)\), write coefficients in powers \(x^a\) with \(a \geq t>0\), using \(d\log x\) as the normal differential. On each such term, contraction with \(x\partial_x\), divided by \(a\), is a contracting homotopy. The tangential differential anticommutes with this contraction. At a node, the coefficients are sums of \(x^ay^b\) with \(a,b>0\), and the same contraction applies. Division by positive integers preserves convergence on smaller polydiscs. Thus the complex has zero stalk cohomology along \(F\). The natural map from \(j_!\mathbb{C}_{Y\setminus F}\) is consequently a quasi-isomorphism. The construction is equivariant whenever the divisor and the group action are equivariant.
\end{proof}

\begin{proposition}\label{prop:ADE-zero}
Suppose that \(B\) is a rational double point and \((Y,F)\) is its minimal resolution. Then \(\mathcal{P}=\mathcal{Q}=0\) and \(\mathcal{D}_Y=0\) in \(R(G)\). If \(B\) is smooth, the same assertions hold with \(Y\) the blowup of the origin and \(F\) its exceptional curve.
\end{proposition}

\begin{proof}
\emph{The rational double point case.} Suppose that \(B\) is a rational double point. The resolution is crepant, so \(Z=0\) and \(D=-F\). The complex corresponding to \eqref{eq:polar-complex} is now
\[
 \mathscr{C}_{-F}^\bullet=
 [\mathcal{O}_Y(-F) \to \Omega_Y^1(\log F)(-F) \to \omega_Y].
\]
By Lemma~\ref{lem:negative-log}, it computes \(j_!\mathbb{C}_U\). A sufficiently small resolution neighborhood retracts onto \(F\). The exact sequence \(0 \to j_!\mathbb{C}_U \to \mathbb{C}_Y \to \mathbb{C}_F \to 0\) therefore gives
\[
 \mathbb{H}^k(Y,\mathscr{C}_{-F}^\bullet)=0\quad\text{for every }k.
\]

Write \(B=\mathbb{C}\{x,y\}^H\) with \(H \subset \SL_2(\mathbb{C})\) finite. Reflexive forms on \(B\) identify with \(H\)-invariant regular forms on \(\mathbb{C}^2\): pull back over the punctured surface and extend the coefficients across the origin. The radial homotopy for the holomorphic de Rham complex commutes with \(H\). It proves \(\mathcal{P}=\mathcal{Q}=0\).

The global-section complex of \(\mathscr{C}_{-F}^\bullet\) is
\begin{equation}\label{eq:ADE-global}
 [\mathfrak{m}_B \to \Omega_B^{[1]} \to \omega_B].
\end{equation}
For its first term, a function vanishes along \(F\) precisely when its value at the singular point is zero. The middle term follows from lifting of derivations and contraction with the nowhere vanishing canonical form; the last term follows from crepancy. The radial homotopy just described contracts \eqref{eq:ADE-global}, since constants have been removed from degree zero.

Only the rows \(q=0,1\) occur in the coherent hypercohomology spectral sequence. The first has zero cohomology, and total hypercohomology is zero. Thus the second row is acyclic. Its alternating class is \([\mathbb{C}_\kappa]\mathcal{D}_Y\), by the contraction identifications \eqref{eq:twists}, which remain valid for \(D=-F\). Hence \(\mathcal{D}_Y=0\).

\smallskip
\noindent\emph{The smooth case.} Suppose that \(B=\mathbb{C}\{x,y\}\) and \(Y\) is the blowup of the origin. The pullback of \(dx\wedge dy\) has divisor \(F\). Thus \(Z=-F\), \(D=-2F\), and the relevant complex is \(\Omega_Y^\bullet(\log F)(-2F)\). Its hypercohomology again vanishes by Lemma~\ref{lem:negative-log} and the retraction onto \(F\). Its global-section complex is
\begin{equation}\label{eq:smooth-global}
 [\mathfrak{m}_B^2 \to \mathfrak{m}_B\Omega_B^1 \to \Omega_B^2].
\end{equation}
Indeed, \(H^0(\mathcal{O}_Y(-2F))=\mathfrak{m}_B^2\). Logarithmic vector fields descend to the vector fields vanishing at the origin, and all such vector fields lift to the blowup. Contraction with \(dx\wedge dy\) gives the middle term. Finally, \(\omega_Y(-F)\) is the pullback of \(\omega_B\).

On a homogeneous form of total radial weight \(w>0\), contraction with \(x\partial_x+y\partial_y\), divided by \(w\), gives the usual homotopy. On \eqref{eq:smooth-global}, it sends one-forms with coefficients in \(\mathfrak{m}_B\) to functions in \(\mathfrak{m}_B^2\), and two-forms to one-forms with coefficients in \(\mathfrak{m}_B\). The complex is therefore acyclic. The same spectral sequence argument gives \(\mathcal{D}_Y=0\). The ordinary Poincar\'{e} lemma gives \(\mathcal{P}=\mathcal{Q}=0\).
\end{proof}

\begin{remark}\label{rem:constant}
The quotient of the full reflexive de Rham complex by \eqref{eq:ADE-global} is \(\mathbb{C}[0]\). In the smooth case the quotient is
\[
 [B/\mathfrak{m}_B^2 \xrightarrow{d}
       \Omega_B^1/\mathfrak{m}_B\Omega_B^1],
\]
which is also quasi-isomorphic to \(\mathbb{C}[0]\): differentiation identifies the linear functions with the constant one-forms. Thus both cases remove the same trivial representation at the level of complexes. The term \(\mathbf{1}\) in \eqref{eq:effective-Hm} cannot be retained when the polar divisor is negative.
\end{remark}

\subsubsection{The determinant character of a quotient surface singularity}

Let \(\Gamma \subset \GL_2(\mathbb{C})\) be finite and act freely on \(\mathbb{C}^2\setminus\{0\}\). Put
\[
 R_0=\mathbb{C}\{x,y\},\quad n=|\Gamma| \geq 2,\quad
 V=\mathbb{C} x\oplus\mathbb{C} y,\quad \chi=\det V.
\]
The action on \(V\) is the action on linear functions. Accordingly, \(dx\wedge dy\) has character \(\chi\). This convention fixes all determinant twists below.

\begin{lemma}\label{lem:invariant-plane-function}
There exists a reduced \(\Gamma\)-invariant homogeneous polynomial \(f(x,y)\) of degree \(n\).
\end{lemma}

\begin{proof}
The degree-\(n\) invariant linear system on \(\mathbb{P}^1\) has no base point. At a prescribed point \([v]\), choose a linear form \(\ell\) that does not vanish at any point in its finite orbit. Then the invariant product \(\prod_{g \in \Gamma}g\ell\) does not vanish at \([v]\). The resulting base point free system defines a nonconstant morphism from \(\mathbb{P}^1\), because its line bundle has positive degree. In characteristic zero this morphism is generically unramified onto its image. A general hyperplane avoids the finitely many exceptional values and meets the image transversely. Its pullback has distinct zeros. The corresponding invariant homogeneous polynomial is reduced.
\end{proof}

For such an \(f\), put \(J_f=R_0/(f_x,f_y)\), and let \(F_f\) be its Milnor fiber in a \(\Gamma\)-invariant Milnor ball.

\begin{proposition}\label{prop:plane-representations}
There are isomorphisms of \(\Gamma\)-representations
\begin{equation}\label{eq:plane-representations}
 \begin{aligned}
 J_f& \simeq \mathbb{C}_{\chi^{-1}}\oplus\mathbb{C}[\Gamma]^{\oplus(n-2)},\\
 H^1(F_f,\mathbb{C})& \simeq \mathbb{C}\oplus\mathbb{C}[\Gamma]^{\oplus(n-2)}.
 \end{aligned}
\end{equation}
Consequently,
\begin{equation}\label{eq:plane-index}
 \mathbf{1}+[J_f]-[H^1(F_f,\mathbb{C})]=[\mathbb{C}_{\chi^{-1}}]
 \quad\text{in }R(\Gamma).
\end{equation}
The isomorphisms in \eqref{eq:plane-representations} are not asserted to be canonical.
\end{proposition}

\begin{proof}
\emph{The Jacobian representation.} Since \(f\) is reduced and homogeneous in two variables, its gradient vanishes only at the origin. For if a nonzero point were critical, Euler's identity would put it on \(f=0\), whose distinct lines meet only at the origin and are smooth elsewhere. Thus \(f_x,f_y\) form a regular sequence. The graded algebra \(J_f\) has dimension \((n-1)^2\), which is also the Milnor number of \(f\); see \cite{Milnor} for the Milnor fiber and Jacobian algebra comparison.

The two partial derivatives span a copy of \(V^*\): the map \(V^* \to \mathbb{C}[x,y]_{n-1}\) sending a constant vector field to its derivative of \(f\) is equivariant and injective. The equivariant Koszul resolution therefore gives
\begin{equation}\label{eq:Koszul-character}
 \sum_{a \geq 0}\tr(g \mid (J_f)_a)t^a
 =\frac{\det(1-t^{n-1}g \mid V^*)}{\det(1-tg \mid V)}.
\end{equation}
If \(g \ne 1\), neither eigenvalue of \(g\) on \(V\) is \(1\), since a nonzero fixed vector would contradict freeness away from the origin. Putting \(t=1\) in \eqref{eq:Koszul-character} gives
\[
 \tr(g \mid J_f)
 =\frac{\det(1-g \mid V^*)}{\det(1-g \mid V)}
 =\chi(g)^{-1}.
\]
At the identity, \(\dim J_f=(n-1)^2=1+n(n-2)\). These are exactly the characters of the first representation in \eqref{eq:plane-representations}. Equality of characters implies an isomorphism because finite-group representations over \(\mathbb{C}\) are semisimple.

\smallskip
\noindent\emph{The Milnor fiber representation.} By Milnor~\cite{Milnor}, the Milnor fiber is connected and has the homotopy type of a bouquet of \((n-1)^2\) circles. The group \(\Gamma\) acts freely on it, since the origin does not belong to a nonzero fiber. Lift a finite cell decomposition of \(F_f/\Gamma\) to \(F_f\). Each chain group is a sum of regular representations, and hence
\[
 \mathbf{1}-[H^1(F_f,\mathbb{C})]
 =\chi(F_f/\Gamma)[\mathbb{C}[\Gamma]].
\]
The Euler characteristic of the quotient is
\[
 \chi(F_f/\Gamma)
 =\frac{1-(n-1)^2}{n}=-(n-2).
\]
This proves the second isomorphism and then \eqref{eq:plane-index}.
\end{proof}

\begin{proposition}\label{prop:quotient-zero}
If \(A=R_0^\Gamma\) is non-Gorenstein, then \(\Delta(A)=0\).
\end{proposition}

\begin{proof}
For a finite quotient that is unramified away from the origin, the canonical character is the determinant character. In particular, \(A\) is Gorenstein if and only if \(\chi=1\). One direction follows by descending \(dx\wedge dy\). For the other, a generator of a trivial canonical module pulls back on the punctured plane to \(u dx\wedge dy\), where \(u\) and \(u^{-1}\) extend holomorphically across the origin by Hartogs' theorem. Evaluating invariance at the origin gives \(\chi=1\). Here \(A\) is non-Gorenstein, so \(\chi \ne 1\).

The derivation descent argument of Lemma~\ref{lem:der-descent} applies also to the smooth algebra \(R_0\) and gives
\[
 (\Der_\mathbb{C} R_0)^\Gamma \simeq \Der_\mathbb{C} A.
\]
As \(f\) is invariant, \(df\) is equivariant. Exactness of invariants gives
\[
 J_f^\Gamma \simeq A/df(\Der_\mathbb{C} A).
\]
The Milnor fiber on \(A\) is \(F_f/\Gamma\); transfer identifies its cohomology with the invariant cohomology of \(F_f\). The descended curve \(f=0\) is reduced, since \(A/(f)=(R_0/(f))^\Gamma\) is reduced. Thus it satisfies the hypotheses of the function-index formula.

Taking invariants in \eqref{eq:plane-representations} gives
\[
 j_A(f)=n-2,\quad \mu_A(f)=n-1.
\]
The singularity \(A\) is non-Gorenstein, so it is not a rational double point, and \eqref{eq:function-index} applies. Therefore
\[
 \Delta(A)=1+j_A(f)-\mu_A(f)=0.
\]
Equivalently, \eqref{eq:plane-index} identifies this integer with \(\dim(\mathbb{C}_{\chi^{-1}})^\Gamma\), which vanishes.
\end{proof}

\subsubsection{Passage to the canonical cover}

Let \((V,0)\) be a non-Gorenstein normal \(\mathbb{Q}\)-Gorenstein surface singularity, and put \(A=\mathcal{O}_{V,0}\). Suppose that its canonical cover \(q\colon(\widetilde{V},0) \to (V,0)\) is smooth or a rational double point, and put \(B=\mathcal{O}_{\widetilde{V},0}\). Write \(B=\mathbb{C}\{x,y\}^H\), where \(H\) is trivial or a finite subgroup of \(\SL_2(\mathbb{C})\). The punctured plane is the universal cover of the punctured surface defined by \(B\). Lifting the action of \(G\) gives a finite extension
\[
 1 \to H \to \Gamma \to G \to 1.
\]
Every lift and its inverse extend across the origin by Hartogs' theorem. Averaging coordinates linearizes this finite analytic action. Explicitly, if \(\rho\) is the derivative representation, the map \(\Phi(z)=|\Gamma|^{-1}\sum_{g \in \Gamma}\rho(g)^{-1}g(z)\) has derivative the identity and satisfies \(\Phi(gz)=\rho(g)\Phi(z)\). The resulting linear action is free outside the origin. Hence
\begin{equation}\label{eq:quotient-cover}
 A \simeq \mathbb{C}\{x,y\}^{\Gamma}.
\end{equation}
Proposition~\ref{prop:quotient-zero} then gives \(\Delta(A)=0\).

With the convention of Proposition~\ref{prop:plane-representations}, \(H=\ker\chi\) and the induced character on \(G\) is \(\kappa\). Indeed, the quotient by the kernel of the determinant is the index-one cover; equivalently, determinant semi-invariants are the reflexive canonical powers on the quotient. This is also obtained by lifting a canonical generator and evaluating its unit coefficient at the origin. The canonical character has order equal to the index. Taking \(H\)-invariants in \eqref{eq:plane-representations} gives
\begin{equation*}%\label{eq:ADE-cover-representations}
 \begin{aligned}
 B/df(\Der_\mathbb{C} B)& \simeq 
 \mathbb{C}_{\kappa^{-1}}\oplus\mathbb{C}[G]^{\oplus(n-2)},\\
 H^1(F_{f,B},\mathbb{C})& \simeq 
 \mathbb{C}\oplus\mathbb{C}[G]^{\oplus(n-2)}.
 \end{aligned}
\end{equation*}
Here \(n=|\Gamma|\), not \(|G|\). Thus the precise identity on these covers is
\begin{equation}\label{eq:exceptional-index}
 \mathbf{1}+[B/df(\Der_\mathbb{C} B)]-[H^1(F_{f,B},\mathbb{C})]
 =[\mathbb{C}_{\kappa^{-1}}].
\end{equation}
It must not be identified with \(\mathcal{D}_Y\), which is zero by Proposition~\ref{prop:ADE-zero}. Its invariant multiplicity is nevertheless zero when \(A\) is non-Gorenstein.

In the ADE case the extra term can also be read directly from the critical-cycle argument. For a sufficiently large exceptional cycle \(M\) as in Section~\ref{sec:descent}, crepancy gives
\[
 H^0(Y,L)=\mathfrak{m}_B\otimes\mathbb{C}_{\kappa^{-1}},\quad
 H^0(Y,L(M))=B\otimes\mathbb{C}_{\kappa^{-1}}.
\]
The last cohomology calculation in Lemma~\ref{lem:critical} acquires the term \(-[\mathbb{C}_{\kappa^{-1}}]\). The same cycle trace therefore gives
\[
 \mathcal{D}_Y=\mathbf{1}+[B/df(\Der_\mathbb{C} B)]-[H^1(F_{f,B},\mathbb{C})]
             -[\mathbb{C}_{\kappa^{-1}}].
\]
This is exactly the correction in \eqref{eq:exceptional-index}. The function-index formula is applied to the non-Gorenstein quotient \(A\) throughout, never to a smooth surface or a rational double point.

\section{Positivity of Wahl defects for \texorpdfstring{\(\mathbb{Q}\)}{Q}-Gorenstein singularities}
\label{sec:positivity-and-equality}

We now combine the representation construction with the two directions of the equality criterion. The result treats the Gorenstein and non-Gorenstein cases in one statement.

\begin{theorem}\label{thm:main}
Let \((V,0)\) be a normal \(\mathbb{Q}\)-Gorenstein complex surface singularity. Put
\[
 \varepsilon(V,0)=
 \begin{cases}
 1&\text{if \((V,0)\) is Gorenstein},\\
 0&\text{if \((V,0)\) is non-Gorenstein}.
 \end{cases}
\]
If \((V,0)\) is not a rational double point, then
\[
 \Delta(V,0) \geq \varepsilon(V,0),
\]
with equality if and only if \((V,0)\) is quasi-homogeneous. For a rational double point, \(\Delta(V,0)=0\).
\end{theorem}

We prove Theorem~\ref{thm:main} in Section~\ref{sec:conclusion}. The proof realizes the Wahl defect as the multiplicity of the canonical character in a finite-dimensional representation of the canonical covering group. This construction treats the Gorenstein and non-Gorenstein cases within a single argument.

\subsection{Nonnegativity of the Wahl defect}
\label{sec:nonnegativity}

Let \((V,0)\) be a normal \(\mathbb{Q}\)-Gorenstein surface singularity, put \(A=\mathcal{O}_{V,0}\), and let \(q\colon(\widetilde{V},0) \to (V,0)\) be its canonical cover. Put \(B=\mathcal{O}_{\widetilde{V},0}\), and use the canonical character \(\kappa\) as above. Let \(p\colon Y \to \widetilde{V}\) be the minimal good resolution of \((\widetilde{V},0)\) when \(\widetilde{V}\) is singular, and the blowup at the origin when \(\widetilde{V}\) is smooth. Write \(\omega_B=(\omega_{\widetilde{V}})_0\) for the canonical module of \(B\). The natural inclusion \(p_*\omega_Y \subseteq \omega_{\widetilde{V}}\) induces a \(G\)-equivariant inclusion on stalks. Thus
\[
 R_B :=(p_*\omega_Y)_0 \subseteq \omega_B
\]
is a \(G\)-stable \(B\)-submodule. Since \(B\) is Gorenstein, \(\omega_B/\mathfrak{m}_B\omega_B\) is a one-dimensional \(G\)-representation with character \(\kappa\). Its quotient \(\omega_B/(R_B+\mathfrak{m}_B\omega_B)\) is therefore either zero or a one-dimensional representation with the same character. Let \(\mathbb{C}_{\kappa^{-1}}\) denote the one-dimensional \(G\)-representation on which \(g \in G\) acts by multiplication by \(\kappa(g)^{-1}\), and put
\[
 T_B :=(\omega_B/(R_B+\mathfrak{m}_B\omega_B))\otimes_{\mathbb{C}}\mathbb{C}_{\kappa^{-1}},
\]
with the diagonal \(G\)-action. Thus \(T_B\) is either zero or a one-dimensional trivial \(G\)-representation. We compute the equivariant Wahl defect \(\mathcal{D}_Y\) on \(Y\).

The submodule \(R_B\) consists of reflexive two-forms that extend to regular two-forms on \(Y\). This extension property is unchanged by further resolution, so \(R_B\) and \(T_B\) are independent of the resolution. Neither definition requires a choice of generator of \(\omega_B\).

Recall that \(\mathcal{P}=H^1(V^\bullet)\) and \(\mathcal{Q}=H^2(V^\bullet)\), where \(V^\bullet\) is the reflexive de Rham complex \eqref{eq:reflexive}. Let \(I_{\mathrm{reg}}\) be the image of regular two-forms in \(\mathcal{Q}\), as in \eqref{eq:regular-quotient}. Let \(E_{\mathfrak{m}} \subseteq \mathcal{Q}\) be the subspace defined in \eqref{eq:Em}. Recall from \eqref{eq:Hm} that \(\mathcal{H}_{\mathfrak{m}}:=I_{\mathrm{reg}}\oplus(\mathcal{Q}/E_{\mathfrak{m}})\). Put
\begin{equation}\label{eq:uniform-representation}
 \mathcal{E}_B :=T_B\oplus\mathcal{H}_{\mathfrak{m}}.
\end{equation}
For smooth and rational double point covers, \(\mathcal{P}=\mathcal{Q}=0\), so \(\mathcal{H}_{\mathfrak{m}}=0\) by the same definition.

\begin{theorem}\label{thm:uniform}
The space \(\mathcal{E}_B\) is a finite-dimensional \(G\)-representation, independent of the principalization, and
\begin{equation}\label{eq:uniform-defect}
 \Delta(A)=\dim(\mathcal{E}_B)_\kappa
          =\dim(\mathcal{E}_B\otimes\mathbb{C}_{\kappa^{-1}})^G.
\end{equation}
If \(B\) is singular and not a rational double point, then \(T_B \simeq \mathbb{C}\) with trivial action, and
\begin{equation}\label{eq:uniform-nonADE}
 \Delta(A)=\delta_{\kappa,1}+\dim(\mathcal{H}_{\mathfrak{m}})_\kappa.
\end{equation}
If \(B\) is smooth or a rational double point, then \(\mathcal{E}_B=0\) and \(\Delta(A)=0\).
\end{theorem}

\begin{proof}
Suppose first that \(B\) is singular and not a rational double point. On its minimal good resolution, the canonical generator has a pole along every exceptional component. A regular two-form \(a\eta\) cannot have \(a\) a unit, since its pullback would then have the same poles. Consequently \(R_B \subset \mathfrak{m}_B\omega_B\), and
\[
 \omega_B/(R_B+\mathfrak{m}_B\omega_B)
 =\omega_B/\mathfrak{m}_B\omega_B \simeq \mathbb{C}_\kappa.
\]
This proves \(T_B \simeq \mathbb{C}\). Equation~\eqref{eq:effective-Hm} gives
\[
 [\mathbb{C}_\kappa]\mathcal{D}_Y=[\mathcal{E}_B].
\]
If \(G\) is trivial, the definition of the defect gives \(\Delta(A)=\dim\mathcal{D}_Y=1+\dim\mathcal{H}_{\mathfrak{m}}\). If \(G\) is nontrivial, Proposition~\ref{prop:transfer} gives \(\Delta(A)=\mult_{\mathbf{1}}(\mathcal{D}_Y)\). Taking character multiplicities proves \eqref{eq:uniform-defect} and \eqref{eq:uniform-nonADE}.

If \(B\) is smooth or a rational double point, then \(R_B=\omega_B\), so \(T_B=0\). Proposition~\ref{prop:ADE-zero} gives \(\mathcal{P}=\mathcal{Q}=0\) and hence \(\mathcal{H}_{\mathfrak{m}}=0\). If \(G\) is trivial and \(A\) is singular, then \(A=B\) is a rational double point, and the same proposition gives \(\Delta(A)=0\). If \(G\) is nontrivial, Section~\ref{sec:quotient-covers} gives \(\Delta(A)=0\) from the determinant character calculation and the function-index formula. Thus \eqref{eq:uniform-defect} holds in these cases as well.

All summands in \eqref{eq:uniform-representation} are finite-dimensional \(G\)-representations. Independence follows from the invariance of \(R_B\) and the principalization comparison in Section~\ref{subsec:maximal-ideal}.
\end{proof}

\subsection{Quasi-homogeneity and equality}
\label{sec:qh-equality}

Let \((V,0)\) be a quasi-homogeneous normal \(\mathbb{Q}\)-Gorenstein surface singularity. In this section we prove \(\Delta(V,0)=\varepsilon(V,0)\) when \((V,0)\) is not a rational double point. We first lift its good \(\mathbb{C}^*\)-action to the canonical cover so that it commutes with the covering group \(G\), by Lemma~\ref{lem:graded-canonical-cover}. Lemma~\ref{lem:analytic-weight-homotopy} reduces the reflexive de Rham complex of the cover to its weight-zero part. Lemma~\ref{lem:zero-weight-one-forms} uses the extension theorem for closed one-forms to show that weight-zero reflexive one-forms annihilate the Euler vector field. Contraction with this field then gives a \(G\)-equivariant isomorphism \(\mathcal{Q}\simeq\mathcal{P}\), by Proposition~\ref{prop:Euler-PQ}. Since the residue representation \(\mathcal{H}_{\mathfrak{m}}\) has class \([\mathcal{Q}]-[\mathcal{P}]\), it vanishes, as proved in Corollary~\ref{cor:qh-residue-zero}. Corollary~\ref{cor:qh-equality} then applies the representation formula \eqref{eq:uniform-defect} to obtain \(\Delta(V,0)=\varepsilon(V,0)\) when \((V,0)\) is not a rational double point.

\subsubsection{The grading on the canonical cover}

We recall a consequence of Singh~\cite[Proposition~4.2]{Singh2003}.

\begin{lemma}\label{lem:graded-canonical-cover}
Let \((V,0)\) be a quasi-homogeneous normal complex \(\mathbb{Q}\)-Gorenstein surface singularity of index \(r\), and let \(q\colon(\widetilde{V},0)\to(V,0)\) be its canonical cover, with covering group \(G\). Then \((\widetilde{V},0)\) admits a good \(\mathbb{C}^*\)-action commuting with \(G\). The induced action on \((V,0)\) is the given action composed with \(t\mapsto t^r\).
\end{lemma}

\begin{proof}
Put \(A:=\mathcal{O}_{V,0}\) and \(B:=\mathcal{O}_{\widetilde{V},0}\). Choose a finitely generated positively graded normal \(\mathbb{C}\)-algebra \(R_0\), with \((R_0)_0=\mathbb{C}\), such that
\[
 A\simeq\mathcal{O}_{(\Spec R_0)^{\mathrm{an}},0}.
\]
The freeness of \(\omega_A^{[r]}\), faithful flatness of the local map from \(R_0\) to \(A\), and graded Nakayama's lemma give a homogeneous generator \(\theta\) of \(\omega_{R_0}^{[r]}\). Put \(d:=\deg\theta\).

Define the canonical cover
\[
 S:=\bigoplus_{i=0}^{r-1}\omega_{R_0}^{[i]},
\]
with multiplication determined by \(\theta\mapsto 1\). For a homogeneous element \(s\in(\omega_{R_0}^{[i]})_n\), put
\[
 \wt(s):=rn-id.
\]
By Singh~\cite[Proposition~4.2]{Singh2003}, these weights are nonnegative and the weight-zero part of \(S\) is \(\mathbb{C}\). For \(t\in\mathbb{C}^*\) and \(g\in G=\mu_r\), the formulas
\[
 t\cdot s:=t^{rn-id}s,\quad g\cdot s:=g^is
\]
define commuting actions on \(S\). The induced action on \(R_0\) multiplies an element of degree \(n\) by \(t^{rn}\), so it is the original action composed with \(t\mapsto t^r\).

The algebraic action on \(\Spec S\) defines a holomorphic action on \((\Spec S)^{\mathrm{an}}\). Its local ring at the vertex is
\[
 \mathcal{O}_{(\Spec S)^{\mathrm{an}},0}
 \simeq S\otimes_{R_0}A
 \simeq\bigoplus_{i=0}^{r-1}\omega_A^{[i]},
\]
with multiplication determined by the image of \(\theta\). The generator used to define \(B\) has the form \(u\theta\), where \(u\in A^\times\). Choose \(v\in A^\times\) with \(v^r=u\). Multiplication by \(v^i\) on the \(i\)th summand gives a \(G\)-equivariant \(A\)-algebra isomorphism from this local ring to \(B\). Through this isomorphism, the good \(\mathbb{C}^*\)-action induces the required action on \((\widetilde{V},0)\).
\end{proof}

\subsubsection{The weight-zero de Rham complex}

Let \((\widetilde{V},0)\) be a normal complex surface singularity with a good \(\mathbb{C}^*\)-action. We also allow \(\widetilde{V}\) to be smooth at \(0\). Put \(B:=\mathcal{O}_{\widetilde{V},0}\) and \(U:=\widetilde{V}\setminus\{0\}\), and let \(\xi\) denote the Euler vector field. We allow a finite group \(G\) commuting with the action. For this subsubsection \(B\) need not be Gorenstein. Recall that
\[
 V^\bullet=[B \xrightarrow{d}\Omega_B^{[1]} \xrightarrow{d}\omega_B],
 \quad \mathcal{P}=H^1(V^\bullet),\quad \mathcal{Q}=H^2(V^\bullet).
\]

\begin{lemma}\label{lem:analytic-weight-homotopy}
Let \(V_0^\bullet\) denote the weight-zero subcomplex of \(V^\bullet\) for the \(\mathbb{C}^*\)-action. Then the inclusion \(V_0^\bullet\to V^\bullet\) is a \(G\)-equivariant chain homotopy equivalence.
\end{lemma}

\begin{proof}
Let \(\rho_\theta\) denote pullback by \(e^{i\theta}\). Choose an \(S^1\times G\)-invariant analytic representative. After shrinking, the pullbacks are defined for \(\theta\) in a complex neighborhood of \([0,2\pi]\). Define
\[
 \Pi_0\alpha:=\frac{1}{2\pi}\int_0^{2\pi}\rho_\theta\alpha d\theta,\quad
 T\alpha:=\frac{i}{2\pi}\int_0^{2\pi}(\theta-\pi)\rho_\theta\alpha d\theta.
\]
These integrals define analytic sections by Jörder~\cite[Proposition~4.10]{JorderPoincare2014}. They preserve holomorphicity on \(U\), so the resulting differential forms are reflexive. The resulting functions extend across \(0\) by normality. Both operators commute with \(d\); see Jörder~\cite[Remark~4.18]{JorderPoincare2014}. The operator \(\Pi_0\) is the projection onto \(V_0^\bullet\), since \(S^1\)-invariance is equivalent to weight zero for the \(\mathbb{C}^*\)-action.

The identity \(\partial_\theta\rho_\theta=i\mathcal{L}_\xi\rho_\theta\) holds. The periodic averaging formula of Avenda\~{n}o Camacho and Vorobiev~\cite[Proposition~2.1]{AvendanoCamachoVorobiev2011} therefore gives
\[
 \mathcal{L}_\xi T=\mathrm{id}-\Pi_0.
\]
Put \(h:=\iota_\xi T\). Since \(T\) commutes with \(d\), Cartan's formula gives
\begin{equation}\label{eq:weight-homotopy}
 dh+hd=(d\iota_\xi+\iota_\xi d)T=\mathrm{id}-\Pi_0.
\end{equation}

The action of \(G\) commutes with the \(\mathbb{C}^*\)-action. Hence \(\rho_\theta\), \(\Pi_0\), \(T\), and \(\iota_\xi\) commute with \(G\). Since \(\Pi_0\) restricts to the identity on \(V_0^\bullet\), equation~\eqref{eq:weight-homotopy} shows that \(\Pi_0\colon V^\bullet\to V_0^\bullet\) is a \(G\)-equivariant homotopy inverse of the inclusion.
\end{proof}

We record a consequence of the extension theorem for closed one-forms.

\begin{lemma}\label{lem:zero-weight-one-forms}
Every \(\alpha\in(\Omega_B^{[1]})_0\) satisfies
\[
 d\alpha=0,\quad \iota_\xi\alpha=0.
\]
\end{lemma}

\begin{proof}
The function \(\iota_\xi\alpha\) has weight zero, so it is constant because \(B_0=\mathbb{C}\). Since \(\mathcal{L}_\xi\alpha=0\), Cartan's formula gives
\[
 \iota_\xi(d\alpha)
 =\mathcal{L}_\xi\alpha-d(\iota_\xi\alpha)=0.
\]
Choose an embedding of the positively graded algebraic model using homogeneous generators \(x_j\) of positive weights \(w_j\). Away from the origin, some \(x_j\) is nonzero, and the corresponding component \(w_jx_j\) of \(\xi\) is nonzero. Thus \(\xi\) is nowhere zero on \(U\). Since \(U\) is a smooth surface, contraction by \(\xi\) is injective on two-forms. Hence \(d\alpha=0\).

Suppose first that \((\widetilde{V},0)\) is singular. Using the positively graded algebraic model, choose a \(\mathbb{C}^*\)-equivariant resolution \(p\colon(Y,F)\to(\widetilde{V},0)\) with an algebraic action on \(F\). By van Straten--Steenbrink~\cite[Corollary~(1.4)]{VanStratenSteenbrink1985}, the closed one-form \(\alpha\) extends to a holomorphic one-form \(\widetilde{\alpha}\) on a neighborhood of \(F\); see also J\"{o}rder~\cite[Proposition~1.2]{JorderExtension}.

The proper algebraic fiber \(F\) has a \(\mathbb{C}^*\)-fixed point \(y\). Indeed, properness extends the algebraic orbit map \(t\mapsto t\cdot z\), for any \(z\in F\), to \(t=0\), and its limit is fixed. Let \(\widetilde{\xi}\) denote the lifted Euler vector field. Then \(\widetilde{\xi}(y)=0\). The identity
\[
 \widetilde{\alpha}(\widetilde{\xi})
 =p^*(\iota_\xi\alpha)
\]
holds on \(Y\setminus F\) and hence on a neighborhood of \(F\). Evaluating at \(y\) shows that the constant \(\iota_\xi\alpha\) is zero.

If \(\widetilde{V}\) is smooth at \(0\), then \(\alpha\) is holomorphic at \(0\). Evaluating there gives the same conclusion, since \(\xi(0)=0\).
\end{proof}

The equality \(\dim\mathcal{P}=\dim\mathcal{Q}\) in the quasi-homogeneous case follows from Du--Yau~\cite[Proposition~3.3]{DuYau2010} and Wahl~\cite[Corollary~1.11]{Wahl1985}. We give an explicit \(G\)-equivariant isomorphism using the Euler vector field.

\begin{proposition}\label{prop:Euler-PQ}
There is a \(G\)-equivariant isomorphism
\begin{equation}\label{eq:Euler-PQ}
 \mathcal{Q}\xrightarrow{\sim}\mathcal{P},
 \quad [\beta]\longmapsto[\iota_\xi\Pi_0\beta],
\end{equation}
where \(\Pi_0\) is the projection onto the weight-zero part for the \(\mathbb{C}^*\)-action.
\end{proposition}

\begin{proof}
Since \(\xi\) is nowhere zero on the smooth surface \(U\), contraction induces a sheaf isomorphism
\[
 \iota_\xi\colon\Omega_U^2\xrightarrow{\sim}
 \ker(\iota_\xi\colon\Omega_U^1\to\mathcal{O}_U).
\]
Indeed, in local coordinates with \(\xi=\partial/\partial z_1\), contraction sends \(dz_1\wedge dz_2\) to \(dz_2\). Its inverse preserves the weights and commutes with \(G\), since it is unique and \(\xi\) is invariant under both actions. Lemma~\ref{lem:zero-weight-one-forms} therefore gives a \(G\)-equivariant isomorphism
\[
 \iota_\xi\colon(\omega_B)_0\xrightarrow{\sim}(\Omega_B^{[1]})_0.
\]

Both differentials in \(V_0^\bullet\) vanish: functions of weight zero are constant, and one-forms of weight zero are closed by Lemma~\ref{lem:zero-weight-one-forms}. By Lemma~\ref{lem:analytic-weight-homotopy}, the projection \(\Pi_0\) therefore induces canonical \(G\)-equivariant isomorphisms
\[
 \mathcal{Q}\simeq(\omega_B)_0,\quad
 \mathcal{P}\simeq(\Omega_B^{[1]})_0.
\]
Combining these isomorphisms with contraction gives \eqref{eq:Euler-PQ}. In particular, for every \(\alpha\in\Omega_B^{[1]}\),
\[
 \Pi_0(d\alpha)=d(\Pi_0\alpha)=0,
\]
so the displayed map is independent of the two-form representative.
\end{proof}

\subsubsection{Vanishing of the residue representation}

\begin{corollary}\label{cor:qh-residue-zero}
Let \((\widetilde{V},0)\) be a quasi-homogeneous Gorenstein normal complex surface singularity. We also allow \(\widetilde{V}\) to be smooth at \(0\). Put \(B:=\mathcal{O}_{\widetilde{V},0}\). Then the residue representation \(\mathcal{H}_{\mathfrak{m}}\) vanishes. If \(\widetilde{V}\) is singular at \(0\) and \((\widetilde{V},0)\) is not a rational double point, then
\begin{equation}\label{eq:qh-residue-terms}
 I_{\mathrm{reg}}=0,\quad
 \mathcal{P}\xrightarrow[\iota]{\sim}W_F,\quad
 \mathcal{Q}=I_{\log}\xrightarrow[\Res]{\sim}W_F.
\end{equation}
Here \(\iota\) is the restriction map of Lemma~\ref{lem:closed-one-forms}, and \(\Res\) is the residue map in \eqref{eq:residue-sequence}.
\end{corollary}

\begin{proof}
If \(\widetilde{V}\) is smooth at \(0\) or \((\widetilde{V},0)\) is a rational double point, Proposition~\ref{prop:ADE-zero} gives \(\mathcal{P}=\mathcal{Q}=0\). The definition \eqref{eq:Hm} therefore gives \(\mathcal{H}_{\mathfrak{m}}=0\).

Suppose now that \(\widetilde{V}\) is singular at \(0\) and \((\widetilde{V},0)\) is not a rational double point. Proposition~\ref{prop:Euler-PQ}, applied with the trivial group, gives \(\dim\mathcal{Q}=\dim\mathcal{P}\). This numerical equality also follows from Du--Yau~\cite[Proposition~3.3]{DuYau2010} and Wahl~\cite[Corollary~1.11]{Wahl1985}. Taking dimensions in \eqref{eq:Hm-class} gives
\[
 \dim\mathcal{H}_{\mathfrak{m}}
 =\dim\mathcal{R}_{\mathrm{res}}
 =\dim\mathcal{Q}-\dim\mathcal{P}=0.
\]
Hence \(\mathcal{H}_{\mathfrak{m}}=0\).

Each summand in the direct sum defining \(\mathcal{R}_{\mathrm{res}}\) in Theorem~\ref{thm:effective} must also vanish. Thus
\[
 I_{\mathrm{reg}}=0,\quad
 W_F/\mathcal{P}=0,\quad
 \mathcal{Q}/I_{\log}=0.
\]
The restriction map of Lemma~\ref{lem:closed-one-forms} is therefore an isomorphism. Since \(I_{\mathrm{reg}}=0\) and \(I_{\log}=\mathcal{Q}\), the residue exact sequence \eqref{eq:residue-sequence} gives the second isomorphism in \eqref{eq:qh-residue-terms}. Both maps are \(G\)-equivariant by their construction.
\end{proof}

\begin{corollary}\label{cor:qh-equality}
Let \((V,0)\) be a quasi-homogeneous normal \(\mathbb{Q}\)-Gorenstein surface singularity, and put \(A=\mathcal{O}_{V,0}\). If \(A\) is Gorenstein and is not a rational double point, then \(\Delta(A)=1\). If \(A\) is non-Gorenstein, then \(\Delta(A)=0\). A rational double point has \(\Delta(A)=0\).
\end{corollary}

\begin{proof}
Let \(q\colon(\widetilde{V},0)\to(V,0)\) be the canonical cover, with covering group \(G\), and put \(A:=\mathcal{O}_{V,0}\) and \(B:=\mathcal{O}_{\widetilde{V},0}\). By Lemma~\ref{lem:graded-canonical-cover}, \((\widetilde{V},0)\) admits a good \(\mathbb{C}^*\)-action commuting with \(G\). Corollary~\ref{cor:qh-residue-zero} gives \(\mathcal{H}_{\mathfrak{m}}=0\).

Suppose first that \((V,0)\) is Gorenstein and is not a rational double point. The canonical cover is trivial, so \(B=A\) and \(\kappa=1\). Equation~\eqref{eq:uniform-nonADE} therefore gives
\[
 \Delta(A)=\delta_{\kappa,1}=1.
\]

Suppose next that \((V,0)\) is non-Gorenstein. Then \(\kappa\ne 1\). If \(\widetilde{V}\) is singular at \(0\) and \((\widetilde{V},0)\) is not a rational double point, equation~\eqref{eq:uniform-nonADE} gives
\[
 \Delta(A)=\delta_{\kappa,1}=0.
\]
If the canonical cover is smooth or a rational double point, equation~\eqref{eq:exceptional-index} and the function-index formula \eqref{eq:function-index} give
\[
 \Delta(A)=\dim((\mathbb{C}_{\kappa^{-1}})^G)=0,
\]
since \(\kappa\) is nontrivial.

Finally, if \((V,0)\) is a rational double point, its canonical cover is trivial, and Proposition~\ref{prop:ADE-zero} gives \(\Delta(A)=0\).
\end{proof}

\subsection{Strict inequality and the equality criterion}
\label{sec:strict-equality}

Let \((V,0)\) be a normal \(\mathbb{Q}\)-Gorenstein surface singularity which is not quasi-homogeneous. We prove \(\Delta(V,0)>\varepsilon(V,0)\) by finding a nonzero \(\kappa\)-isotypic component of the residue representation. We use the analytic criterion of Scheja--Wiebe~\cite[Satz~(3.1)]{SchejaWiebe1980}. A normal complex analytic algebra of dimension two is quasi-homogeneous if it admits a derivation acting nonnilpotently on its cotangent space; see also Wahl~\cite[\S3.16]{Wahl1985}.

\subsubsection{Divergence and the cotangent space}

Let \((\widetilde{V},0)\) be a Gorenstein normal complex surface singularity, put \(B=\mathcal{O}_{\widetilde{V},0}\), with maximal ideal \(\mathfrak{m}_B\), and let \(\eta\) generate \(\omega_B\). For a derivation \(\delta\), define \(\divr_\eta\delta\) by
\[
 d(\iota_\delta\eta)=(\divr_\eta\delta)\eta.
\]
This is a complex-linear operation; its image need not be an ideal.

\begin{proposition}\label{prop:unit-divergence}
If \(\divr_\eta\delta\notin\mathfrak{m}_B\) for some \(\delta\in\Der_{\mathbb{C}}(B)\), then \((\widetilde{V},0)\) is quasi-homogeneous.
\end{proposition}

\begin{proof}
Every derivation of \(B\) preserves \(\mathfrak{m}_B\), since its local flow preserves the isolated singular point. Suppose that \(\delta\) induces a nilpotent endomorphism of \(\mathfrak{m}_B/\mathfrak{m}_B^2\). Put \(a:=\divr_\eta\delta\). We show that \(a(0)=0\).

Define
\[
 J_{\mathrm{can}}:=\{f\in B:f\eta\in\im(\Omega_B^2\to\omega_B)\}.
\]
This ideal is nonzero and is the unit ideal away from the singular point. Hence it is either \(B\) or \(\mathfrak{m}_B\)-primary. Lie differentiation preserves the image of K\"{a}hler two-forms, and
\[
 \mathcal{L}_\delta(f\eta)=(\delta(f)+af)\eta.
\]
Since \(aJ_{\mathrm{can}}\subseteq J_{\mathrm{can}}\), it follows that \(\delta(J_{\mathrm{can}})\subseteq J_{\mathrm{can}}\).

The action of \(\delta\) on coefficients and Lie differentiation of forms induce endomorphisms \(T_1\) and \(T_2\) of \(J_{\mathrm{can}}/\mathfrak{m}_BJ_{\mathrm{can}}\), respectively. For the latter, we use multiplication by \(\eta\) to identify the quotient with \((J_{\mathrm{can}}\eta)/\mathfrak{m}_B(J_{\mathrm{can}}\eta)\). Thus
\[
 T_1([f]):=[\delta(f)],\quad
 T_2([f]):=[\delta(f)+af],\quad
 T_2-T_1=a(0)\id.
\]

The induced derivation on each \(\mathfrak{m}_B^j/\mathfrak{m}_B^{j+1}\) is nilpotent, since this space is an equivariant quotient of \(\Sym_{\mathbb{C}}^j(\mathfrak{m}_B/\mathfrak{m}_B^2)\) with its derivation action. The finite filtration by powers of \(\mathfrak{m}_B\) therefore shows that \(\delta\) acts nilpotently on \(B/\mathfrak{m}_B^N\) for every \(N\). Choose \(N\) such that \(\mathfrak{m}_B^N\subseteq\mathfrak{m}_BJ_{\mathrm{can}}\). The endomorphism \(T_1\) is induced on a \(\delta\)-stable subquotient of \(B/\mathfrak{m}_B^N\), so it is nilpotent.

The natural surjection
\[
 \Omega_B^2\otimes_B\mathbb{C}
 \to (J_{\mathrm{can}}\eta)/\mathfrak{m}_B(J_{\mathrm{can}}\eta)
\]
commutes with Lie differentiation. Under the identification
\[
 \Omega_B^2\otimes_B\mathbb{C}
 \simeq\bigwedge_{\mathbb{C}}^2(\mathfrak{m}_B/\mathfrak{m}_B^2),
\]
the Lie derivative is the derivation action induced by the nilpotent endomorphism on the cotangent space. It is therefore nilpotent, and so is its quotient \(T_2\).

Nakayama's lemma gives \(J_{\mathrm{can}}/\mathfrak{m}_BJ_{\mathrm{can}}\ne 0\). Taking traces yields
\[
 0=\tr(T_2)-\tr(T_1)
 =a(0)\dim_{\mathbb{C}}(J_{\mathrm{can}}/\mathfrak{m}_BJ_{\mathrm{can}}).
\]
Hence \(a(0)=0\).

Consequently, if \(\divr_\eta\delta\) is a unit, then the induced endomorphism of \(\mathfrak{m}_B/\mathfrak{m}_B^2\) is nonnilpotent. The analytic criterion of Scheja--Wiebe~\cite[Satz~(3.1)]{SchejaWiebe1980} now gives quasi-homogeneity; see also Wahl~\cite[\S3.16]{Wahl1985}.
\end{proof}

\subsubsection{A nonzero canonical-character quotient}

\begin{proposition}\label{prop:strict-residue}
Let \((\widetilde{V},0)\) be a Gorenstein normal complex surface singularity that is not quasi-homogeneous, put \(B:=\mathcal{O}_{\widetilde{V},0}\), and let a finite group \(G\) act on \((\widetilde{V},0)\). Choose a generator \(\eta\) of \(\omega_B\) satisfying \(g\cdot\eta=\kappa(g)\eta\) for every \(g\in G\). Then
\[
 \dim(\mathcal{R}_{\mathrm{res}})_\kappa\geq 1.
\]
\end{proposition}

\begin{proof}
The singularity \((\widetilde{V},0)\) is not a rational double point. Let \(p\colon(Y,F)\to(\widetilde{V},0)\) be its minimal good resolution. Put \(Z:=-\divr_Y(p^*\eta)\) and \(D:=Z-F\). By Wahl~\cite[Lemma~3.6]{Wahl1985}, \(Z\geq F\), so \(D\geq 0\).

\emph{The case \(D\ne 0\).} Recall that \(\mathcal{Q}=\omega_B/d\Omega_B^{[1]}\) and that \(I_{\log}\) is the image of \((p_*\omega_Y(F))_0\) in \(\mathcal{Q}\). Contraction with \(\eta\) gives a \(G\)-equivariant isomorphism
\[
 \Der_{\mathbb{C}}(B)\otimes_{\mathbb{C}}\mathbb{C}_\kappa
 \xrightarrow{\sim}\Omega_B^{[1]}.
\]
Moreover,
\[
 (p_*\omega_Y(F))_0=(p_*\mathcal{O}_Y(-D))_0\eta.
\]
Thus there is an isomorphism of \(G\)-representations
\begin{equation*}%\label{eq:strict-quotient}
 \mathcal{Q}/I_{\log}\simeq
 \frac{B}{\divr_\eta(\Der_{\mathbb{C}}(B))+(p_*\mathcal{O}_Y(-D))_0}
 \otimes_{\mathbb{C}}\mathbb{C}_\kappa.
\end{equation*}
Here the denominator is a complex subspace of \(B\). Since \(g\cdot\eta=\kappa(g)\eta\), Lie differentiation gives
\[
 \divr_\eta(g\cdot\delta)
 =g\cdot\divr_\eta\delta.
\]
Hence the denominator is \(G\)-stable.

Proposition~\ref{prop:unit-divergence} gives \(\divr_\eta(\Der_{\mathbb{C}}(B))\subseteq\mathfrak{m}_B\). Since \(D\) is nonzero and effective, a unit cannot vanish along \(D\), so \((p_*\mathcal{O}_Y(-D))_0\subseteq\mathfrak{m}_B\) as well. Evaluation at the origin therefore induces a \(G\)-equivariant surjection
\[
 \mathcal{Q}/I_{\log}\twoheadrightarrow\mathbb{C}_\kappa.
\]
Since complex representations of finite groups are semisimple and \(\mathcal{Q}/I_{\log}\) is a summand of \(\mathcal{R}_{\mathrm{res}}\), we obtain
\[
 \dim(\mathcal{R}_{\mathrm{res}})_\kappa
 \geq\dim(\mathcal{Q}/I_{\log})_\kappa
 \geq 1.
\]

\emph{The case \(D=0\).} Then \(Z=F\), so \(p^*\eta\) is a nowhere vanishing section of \(\omega_Y(F)\) with character \(\kappa\). Adjunction gives
\[
 \omega_F\simeq\mathcal{O}_F\otimes_{\mathbb{C}}\mathbb{C}_\kappa.
\]
Since \(F\) is connected, proper, and reduced, \(H^0(F,\mathcal{O}_F)=\mathbb{C}\). Consequently,
\[
 W_F\simeq\mathbb{C}_\kappa.
\]

Let \(n_i\) denote the number of intersections of \(F_i\) with the other components, counted with multiplicity. Restricting \(\omega_Y(F)\) to \(F_i\) gives
\[
 2g(F_i)-2+n_i=0.
\]
Connectedness therefore implies that \(F\) is either a single smooth elliptic curve or a cycle of smooth rational curves. In the first case, \((\widetilde{V},0)\) is simple elliptic and hence quasi-homogeneous; see Wahl~\cite[\S3.8]{Wahl1985}. This contradicts the hypothesis.

In the rational-cycle case, every \(H^0(F_i,\omega_{F_i})\) is zero. The injection of Lemma~\ref{lem:closed-one-forms} therefore gives \(\mathcal{P}=0\). Thus the summand \(W_F/\mathcal{P}\) of \(\mathcal{R}_{\mathrm{res}}\) is isomorphic to \(\mathbb{C}_\kappa\), proving the assertion.
\end{proof}

\subsubsection{Finite quotients and the converse}

We recall a consequence of M\"{u}ller's results on reductive automorphism groups.

\begin{lemma}\label{lem:qh-quotient-descent}
Let \((\widetilde{V},0)\) be a quasi-homogeneous Gorenstein normal complex surface singularity that is not a rational double point, and let a finite group \(H\) act on \((\widetilde{V},0)\). Then the quotient \((V,0):=(\widetilde{V},0)/H\) admits a good \(\mathbb{C}^*\)-action.
\end{lemma}

\begin{proof}
The singularity \((\widetilde{V},0)\) is not a cyclic quotient singularity, since a Gorenstein cyclic quotient singularity is a rational double point. By M\"{u}ller~\cite[\S2, Theorem~1]{Mueller}, after conjugating \(H\) by an analytic automorphism, we may assume that \(H\) is contained in a maximal reductive algebraic subgroup of \(\Aut(\widetilde{V},0)\). By M\"{u}ller~\cite[\S2, Theorem~2 and Remark~(iii)]{Mueller}, the identity component of this subgroup is the central \(\mathbb{C}^*\) defining a good action. Hence this action commutes with \(H\).

Put \(B:=\mathcal{O}_{\widetilde{V},0}\) and \(A:=\mathcal{O}_{V,0}=B^H\). As in the proof of M\"{u}ller's Theorem~2, choose homogeneous coordinates in which the maximal reductive group acts linearly and the defining equations are weighted homogeneous. Let \(R\) be the resulting positively graded coordinate ring, with \(R_0=\mathbb{C}\). The action of \(H\) preserves each graded piece, so
\[
 R^H=\bigoplus_{n\geq 0}R_n^H.
\]
This invariant algebra is finitely generated, has degree-zero part \(\mathbb{C}\), and has only positive degrees otherwise. The finite quotient gives
\[
 \mathcal{O}_{(\Spec(R^H))^{\mathrm{an}},0}
 \simeq B^H=A.
\]
Thus the positive grading defines a good \(\mathbb{C}^*\)-action on \((V,0)\).
\end{proof}

\begin{corollary}\label{cor:strict-base}
Let \((V,0)\) be a normal complex surface singularity which is not quasi-homogeneous. If \((V,0)\) is Gorenstein, then \(\Delta(V,0) \geq 2\). If \((V,0)\) is non-Gorenstein and \(\mathbb{Q}\)-Gorenstein, then \(\Delta(V,0) \geq 1\).
\end{corollary}

\begin{proof}
Put \(A:=\mathcal{O}_{V,0}\). Suppose first that \((V,0)\) is Gorenstein. Since it is not quasi-homogeneous, it is not a rational double point. Its canonical cover is trivial, so \(B=A\), \(G=\{1\}\), and \(\kappa=1\). Equations~\eqref{eq:uniform-nonADE} and~\eqref{eq:Hm-class}, together with Proposition~\ref{prop:strict-residue}, give
\[
 \Delta(V,0)
 =1+\dim\mathcal{H}_{\mathfrak{m}}
 =1+\dim\mathcal{R}_{\mathrm{res}}
 \geq 2.
\]

Suppose now that \((V,0)\) is non-Gorenstein and \(\mathbb{Q}\)-Gorenstein. Let \(q\colon(\widetilde{V},0) \to (V,0)\) be its canonical cover, with covering group \(G\) and canonical character \(\kappa\), and put \(B:=\mathcal{O}_{\widetilde{V},0}\). Then \(\kappa\neq 1\). If \(\widetilde{V}\) is smooth at the origin or \((\widetilde{V},0)\) is a rational double point, the quotient presentation \eqref{eq:quotient-cover} shows that scalar multiplication on \(\mathbb{C}^2\) induces a good \(\mathbb{C}^*\)-action on \((V,0)\). This contradicts the hypothesis. Thus \((\widetilde{V},0)\) is singular and is not a rational double point. If it were quasi-homogeneous, Lemma~\ref{lem:qh-quotient-descent} would imply that \((V,0)\) is quasi-homogeneous. Hence \((\widetilde{V},0)\) is not quasi-homogeneous.

Since \(\kappa\neq 1\), equations~\eqref{eq:uniform-nonADE} and~\eqref{eq:Hm-class}, together with Proposition~\ref{prop:strict-residue}, give
\[
 \Delta(V,0)
 =\dim(\mathcal{H}_{\mathfrak{m}})_\kappa
 =\dim(\mathcal{R}_{\mathrm{res}})_\kappa
 \geq 1. \qedhere
\]
\end{proof}

\subsection{Proof of the main theorem}
\label{sec:conclusion}

We now prove Theorem~\ref{thm:main}.

\begin{proof}
Put \(A=\mathcal{O}_{V,0}\), and let \(q\colon(\widetilde{V},0) \to (V,0)\) be the canonical cover, with local ring \(B=\mathcal{O}_{\widetilde{V},0}\), covering group \(G\) and canonical character \(\kappa\). We use \(\Delta(A)=\Delta(V,0)\). If \(A\) is Gorenstein, then \(B=A\), \(G=\{1\}\), and \(\kappa=1\). Suppose first that \(B\) is singular and is not a rational double point. By Theorem~\ref{thm:uniform},
\[
 \Delta(V,0)=\Delta(A)=\delta_{\kappa,1}
                 +\dim(\mathcal{H}_{\mathfrak{m}})_\kappa
 \geq \varepsilon(V,0),
\]
since \(\delta_{\kappa,1}=\varepsilon(V,0)\). If \(A\) is quasi-homogeneous, its good action lifts to \(B\) by Lemma~\ref{lem:graded-canonical-cover}, and Corollary~\ref{cor:qh-residue-zero} gives \(\mathcal{H}_{\mathfrak{m}}=0\). Thus equality holds. Conversely, Corollary~\ref{cor:strict-base} gives \(\Delta(A) \geq \varepsilon(V,0)+1\) whenever \(A\) is not quasi-homogeneous.

If \(B\) is smooth or a rational double point and \(G\) is nontrivial, Section~\ref{sec:quotient-covers} presents \(A\) as a finite linear quotient of the plane and gives \(\Delta(A)=0\). The canonical character is nontrivial, so \(\varepsilon(V,0)=0\), and scalar multiplication on the plane makes \(A\) quasi-homogeneous. Finally, if \(G\) is trivial, the singularity \(A=B\) is a rational double point, and Proposition~\ref{prop:ADE-zero} gives \(\Delta(A)=0\).
\end{proof}

\begin{remark}
For a singular canonical cover other than a rational double point, the difference between the lower bounds \(1\) and \(0\) is the contribution of the trivial summand \(T_B\) to the \(\kappa\)-isotypic component. The role of a trivial summand in distinguishing the two bounds already appears in Wahl's hypersurface quotient calculation~\cite[proof of Theorem~5.2]{Wahl2015}. Here it is realized by the explicit quotient of two-forms defining \(T_B\), together with the residue representation, for general canonical covers without a smoothing hypothesis. Quasi-homogeneity makes the entire residue representation vanish by Euler contraction. In the other direction, failure of quasi-homogeneity supplies a nonzero canonical-character quotient by Proposition~\ref{prop:strict-residue}.
\end{remark}

\section{Wahl's conjectures}
\label{sec:corollaries}

We apply Theorem~\ref{thm:main} to prove Wahl's \(\mathbb{Q}\)-Gorenstein and Rational Conjectures.

\begin{theorem}[Wahl's \(\mathbb{Q}\)-Gorenstein conjecture]\label{thm:QG-conjecture}
For a strong \(\mathbb{Q}\)-Gorenstein smoothing \(\pi\) of a non-Gorenstein normal surface singularity \((V,0)\),
\[
 1+\mu_\pi \geq \tau_\pi,
\]
with equality if and only if \((V,0)\) is quasi-homogeneous.
\end{theorem}

\begin{proof}
Since \(\pi\) is a strong \(\mathbb{Q}\)-Gorenstein smoothing, the central singularity \((V,0)\) is \(\mathbb{Q}\)-Gorenstein. It is non-Gorenstein by assumption, so Theorem~\ref{thm:main} gives \(\Delta(V,0) \ge 0\), with equality if and only if \((V,0)\) is quasi-homogeneous. Since \(\alpha_\pi=0\) by Wahl~\cite[Lemma~5.1]{Wahl2015}, Equation~\eqref{eq:introduction-wahl-smoothing} reads \(\Delta(V,0)=1+\mu_\pi-\tau_\pi\). Both assertions now follow from Theorem~\ref{thm:main}.
\end{proof}

The weaker assumption that \(K_{\mathcal{V}}\) is \(\mathbb{Q}\)-Cartier does not ensure this specialization property and allows \(\alpha_\pi>0\). Indeed, Example~\ref{ex:weak-qg} gives a weakly \(\mathbb{Q}\)-Gorenstein smoothing of a weighted homogeneous rational surface singularity with
\[
 \Delta(V,0)=0,\quad \alpha_\pi=2,\quad
 \mu_\pi=3,\quad \tau_\pi=6.
\]
Thus \(1+\mu_\pi-\tau_\pi=-2\), and Wahl's inequality fails. The index-one-cover hypothesis therefore cannot be replaced by the requirement that \(K_{\mathcal{V}}\) be \(\mathbb{Q}\)-Cartier.

\begin{theorem}[Wahl's Rational Conjecture]
\label{thm:rational-conjecture}
For a rational surface singularity which is not a rational double point,
\[
 h^1(X,L_X) \ge h^1(X,S_X)
\]
with equality if and only if the singularity is quasi-homogeneous.
\end{theorem}

\begin{proof}
A rational Gorenstein surface singularity is a rational double point, by Artin~\cite[pp.~134--136]{Artin1966}. Thus the hypotheses imply that \((V,0)\) is non-Gorenstein. It is \(\mathbb{Q}\)-Gorenstein, as recalled in Section~\ref{qg:sec:relative-cover}, so Theorem~\ref{thm:main}, together with \(h^1(X,\mathcal{O}_X)=0\), proves the inequality and its equality criterion.
\end{proof}

\section{A modified form of Wahl's Main Conjecture}
\label{genus:sec:bound}

The comparison of differential forms also gives a lower bound for arbitrary normal surface singularities. Let \((X,E) \to (V,0)\) be the minimal good resolution, and put \(p_g=p_g(V,0)=h^1(X,\mathcal{O}_X)\).

\begin{theorem}\label{genus:thm:main}
For every normal complex surface singularity \((V,0)\),
\[
 \Delta(V,0)+4p_g(V,0) \ge 0
\]
with equality if and only if \((V,0)\) is rational and quasi-homogeneous.
\end{theorem}

We prove Theorem~\ref{genus:thm:main} in Section~\ref{sec:proof-genus-thm-main}. The argument uses a numerically trivial line bundle \(N\) for which
\[
 h^1(N)-h^1(S_X\otimes N)+h^1(L_X\otimes N) \ge 0.
\]
Here we call a line bundle \emph{numerically trivial} if its degree on every exceptional component is zero. The bound then follows by estimating the change in cohomology under tensoring by \(N\). The construction of \(N\) uses the numerical anticanonical cycle and a finite cyclic cover.

\subsection{Numerically trivial line bundles}
We first estimate the change in first cohomology under tensoring by a numerically trivial line bundle in terms of the geometric genus. For line bundles, we also obtain a strict estimate when \(p_g>0\), which will be used to determine the equality case.

For line bundles on resolutions of surface singularities with rational homology sphere links, the corresponding estimate follows from Nagy--N\'{e}methi~\cite[\S5.1, Theorem~5.3.1, and Proposition~5.7.1(c)]{NagyNemethi2019}. We give a direct proof for vector bundles on arbitrary good resolutions.

\begin{lemma}\label{genus:lem:variation}
Let \(\pi\colon(X,E) \to (V,0)\) be a good resolution of a normal complex surface singularity, let \(F\) be a vector bundle of rank \(s\) on \(X\), and let \(N\) be a numerically trivial line bundle on \(X\). Then
\[
 |h^1(F)-h^1(F\otimes N)| \le s p_g,
\]
where \(p_g=h^1(X,\mathcal{O}_X)\).
\end{lemma}

\begin{proof}
Choose an effective exceptional cycle \(H\) such that \(\mathcal{O}_X(-H)\) is relatively ample. By relative Serre vanishing, a sufficiently large multiple \(D\) of \(H\) satisfies
\[
 H^1(X,F(-D))=H^1(X,F\otimes N(-D))=0.
\]
Together with the coherent \(H^2\)-vanishing in \eqref{eq:coherent-vanishing}, the restriction sequences give isomorphisms
\[
 H^1(X,F)\simeq H^1(D,F|_D),\quad
 H^1(X,F\otimes N)\simeq H^1(D,(F\otimes N)|_D).
\]
The same vanishing makes \(H^1(X,\mathcal{O}_X) \to H^1(D,\mathcal{O}_D)\) surjective, so \(h^1(D,\mathcal{O}_D)\le p_g\).

Put \(\mathcal{G}:=\omega_D\otimes(F|_D)^\vee\), and let \(\mathcal{A}\subseteq\mathcal{G}\) be the image of the evaluation map
\[
 H^0(D,\mathcal{G})\otimes_{\mathbb{C}}\mathcal{O}_D \to \mathcal{G}.
\]
Then \(H^0(D,\mathcal{A})=H^0(D,\mathcal{G})\). Since \(\mathcal{O}_D(-H)\) is ample, \(D\) is projective, and GAGA allows us to work with algebraic coherent sheaves on \(D\). At the generic point of a component of multiplicity \(m\), the local ring of \(D\) is a discrete valuation ring modulo the \(m\)th power of a uniformizer. A submodule of a free module of rank \(s\) over this principal Artinian local ring is generated by at most \(s\) elements. Thus \(s\) general linear combinations of the global sections generate \(\mathcal{A}\) at that generic point. The failure of generation is a proper algebraic condition on their complex coefficients. Since there are only finitely many components, the sections can be chosen to generate at all generic points simultaneously.

We obtain a map \(\mathcal{O}_D^{\oplus s} \to \mathcal{A}\) with finite-length cokernel. Since coherent sheaves on \(D\) have vanishing second cohomology, the cohomology sequences for its kernel and image give
\[
 h^1(D,\mathcal{A}) \le s h^1(D,\mathcal{O}_D).
\]

Tensoring a coherent sheaf on \(D\) by \(N^{-1}|_D\) preserves its Euler characteristic. Indeed, a filtration by sheaves on the reduced components and finite-length sheaves reduces this assertion to Riemann--Roch on each smooth component, where the change is the rank times \(\deg(N^{-1}|_C)=0\); see the Stacks Project~\cite[Tag~0AYV]{Stacks}. Since \(\mathcal{A}\otimes N^{-1}|_D\) is a subsheaf of \(\mathcal{G}\otimes N^{-1}|_D\), it follows that
\begin{align*}
 h^0(D,\mathcal{G})-h^0(D,\mathcal{G}\otimes N^{-1}|_D)
 &\le h^0(D,\mathcal{A})-h^0(D,\mathcal{A}\otimes N^{-1}|_D)\\
 &=h^1(D,\mathcal{A})-h^1(D,\mathcal{A}\otimes N^{-1}|_D)\\
 &\le s h^1(D,\mathcal{O}_D) \le s p_g.
\end{align*}
The cycle \(D\) is a Cartier divisor in the smooth surface \(X\), hence a Gorenstein curve. Serre duality and the restriction isomorphisms identify the left side with \(h^1(F)-h^1(F\otimes N)\). Applying the same argument to \(F\otimes N\) and \(N^{-1}\) proves the opposite inequality.
\end{proof}

\begin{lemma}\label{genus:lem:strict}
Let \((X,E) \to (V,0)\) be a good resolution of a normal complex surface singularity, and let \(N\) be a numerically trivial line bundle on \(X\). If \(p_g=0\), then \(h^1(N)=0\). If \(p_g>0\), then
\[
 h^1(N) \le 2p_g-1.
\]
\end{lemma}

\begin{proof}
If \(p_g=0\), Lemma~\ref{genus:lem:variation}, applied to \(\mathcal{O}_X\) and \(N\), gives \(h^1(N)=0\). Suppose that \(p_g>0\), and put \(n:=h^1(N)\). There is nothing to prove if \(n=0\). Assume that \(n>0\). As in the restriction argument of Lemma~\ref{genus:lem:variation}, choose an effective exceptional cycle \(D\) with \(h^1(D,N|_D)=n\), having minimal sum of coefficients among all such cycles. In particular, \(D\neq 0\).

Put \(\mathcal{G}:=\omega_D\otimes N^{-1}|_D\). We adapt the minimality argument in Konno~\cite[Lemma~3.1(2)]{Konno2008} to this twisted dualizing sheaf. For every reduced irreducible component \(C\le D\), adjunction gives
\[
 0 \to \omega_{D-C}\otimes N^{-1}|_{D-C}
 \to \mathcal{G}
 \to \omega_C(D-C)\otimes N^{-1}|_C
 \to 0.
\]
The first term is understood to be zero when \(D-C=0\). If the induced restriction map on global sections were zero, this sequence and Serre duality would give
\[
 h^1(D-C,N|_{D-C})
 =h^0(D-C,\omega_{D-C}\otimes N^{-1}|_{D-C})
 =h^0(D,\mathcal{G})
 =n,
\]
contrary to the minimality of \(D\).

The kernels of these restriction maps are therefore proper linear subspaces of \(H^0(D,\mathcal{G})\). Choose a section outside their finite union. Its restriction to each reduced component is nonzero, so its coefficient in a local trivialization is a unit at every generic point of \(D\). Since \(D\) is a Cartier divisor in the smooth surface \(X\), it is Cohen--Macaulay and has no embedded associated points. The section is consequently a nonzerodivisor and gives an exact sequence
\[
 0 \to \mathcal{O}_D \to \mathcal{G} \to \mathcal{T} \to 0,
\]
where \(\mathcal{T}\) has finite length. Hence
\[
 h^1(D,\mathcal{G})\le h^1(D,\mathcal{O}_D).
\]

Numerical triviality preserves the Euler characteristic under tensoring by \(N^{-1}|_D\), as in the preceding proof. Thus Serre duality gives
\[
 \chi(\mathcal{G})=\chi(\omega_D)=-\chi(\mathcal{O}_D).
\]
It follows that
\begin{equation*}
\begin{split}
 n
 &=h^0(D,\mathcal{G})=\chi(\mathcal{G})+h^1(D,\mathcal{G})\\
 &\le -\chi(\mathcal{O}_D)+h^1(D,\mathcal{O}_D)=2h^1(D,\mathcal{O}_D)-h^0(D,\mathcal{O}_D)\\
 &\le 2p_g-1.
\end{split}
\end{equation*}
Here \(h^1(D,\mathcal{O}_D)\le p_g\) follows from the restriction sequence and \eqref{eq:coherent-vanishing}, while \(h^0(D,\mathcal{O}_D)\ge 1\) because \(D\neq 0\).
\end{proof}

\subsection{Differential forms on numerically Gorenstein resolutions}

We use the complexes of logarithmic forms from Section~\ref{sec:forms}, allowing an arbitrary integral exceptional divisor in place of the effective divisor occurring there. The following comparison applies to all good resolutions.

The following lemma extends the comparison of logarithmic complexes in Wahl~\cite[\S1.5--1.6 and Proposition~1.7]{Wahl1985} to arbitrary integral exceptional divisors and equivariant cohomology classes.

\begin{lemma}\label{genus:lem:euler}
Let \((Y,F)\) be a good resolution of a normal complex surface singularity or of a smooth surface point, with \(F\neq\varnothing\). Let a finite group \(G\) act on \(Y\), preserving \(F\), and let \(D\) be a \(G\)-invariant integral exceptional divisor. Define
\[
 C_D^\bullet:=
 [\mathcal{O}_Y(D) \xrightarrow{d} \Omega_Y^1(\log F)(D)
 \xrightarrow{d} \omega_Y(F+D)].
\]
The cohomology groups of its global-section complex,
\[
 B^i:=H^i(\Gamma(Y,C_D^\bullet)),
\]
are finite dimensional, and
\begin{equation}\label{genus:eq:euler}
 \sum_{i=0}^2(-1)^i[H^1(Y,C_D^i)]
 =[B^0]-[B^1]+[B^2]
 \quad\text{in }R(G).
\end{equation}
Moreover, \(B^0\) is the trivial representation when \(D\ge 0\), and is zero otherwise.
\end{lemma}

\begin{proof}
Exterior differentiation preserves the indicated pole orders in a logarithmic basis, so \(C_D^\bullet\) is a complex of sheaves of complex vector spaces. Choose a \(G\)-invariant effective exceptional cycle \(Z\ge F\) such that \(-Z\le D\). By Lemma~\ref{lem:negative-log}, \(C_{-Z}^\bullet\) resolves \(j_!\mathbb{C}_{Y\setminus F}\), where \(j\colon Y\setminus F\hookrightarrow Y\). On a sufficiently small resolution neighborhood, \(Y\) retracts onto \(F\). Thus the exact sequence
\[
 0 \to j_!\mathbb{C}_{Y\setminus F}
 \to \mathbb{C}_Y
 \to \mathbb{C}_F
 \to 0
\]
gives \(\mathbb{H}^k(Y,C_{-Z}^\bullet)=0\) for every \(k\).

Put
\[
 \mathcal{T}:=\mathcal{O}_Y(D)/\mathcal{O}_Y(-Z).
\]
The quotient complex \(C_D^\bullet/C_{-Z}^\bullet\) has terms \(\mathcal{T}\otimes\bigwedge^i\Omega_Y^1(\log F)\), all supported on the proper exceptional curve. For a \(G\)-equivariant coherent sheaf \(\mathcal{M}\) with such support, write
\[
 \chi_G(\mathcal{M}):=[H^0(Y,\mathcal{M})]-[H^1(Y,\mathcal{M})].
\]
We show that
\[
 \sum_{i=0}^2(-1)^i
 \chi_G(\mathcal{T}\otimes\bigwedge^i\Omega_Y^1(\log F))=0.
\]

Let \(G_i\) be the stabilizer of a component \(F_i\), and let \(P_i\) be the reduced divisor of its intersections with the other components. The residue sequence
\[
 0 \to \Omega_{F_i}^1(\log P_i)
 \to \Omega_Y^1(\log F)|_{F_i}
 \to \mathcal{O}_{F_i}
 \to 0
\]
is \(G_i\)-equivariant, with the natural linearization on the quotient. Taking determinants gives
\[
 \bigwedge^2(\Omega_Y^1(\log F)|_{F_i})
 \simeq \Omega_{F_i}^1(\log P_i)
\]
equivariantly. Tensoring the residue sequence by any \(G_i\)-equivariant coherent sheaf on \(F_i\) therefore makes the alternating cohomology classes of the three exterior powers cancel. Induction from \(G_i\) to \(G\) gives the same cancellation on the orbit of \(F_i\).

The corresponding cancellation also holds for sheaves of finite length. At a smooth point of \(F\), the residue map gives a trivial quotient of the logarithmic cotangent fiber. At a node with local equation \(xy=0\), the point stabilizer permutes the classes of \(d\log x\) and \(d\log y\) in that fiber. Indeed, changes of the local equations by units contribute regular differentials, whose classes in the logarithmic cotangent fiber are zero. The sum of the two classes is therefore invariant. In both cases the fiber contains a trivial summand, so its alternating exterior-power class is zero. A filtration by powers of the maximal ideal gives the same cancellation for every equivariant sheaf of finite length.

A power of the ideal of \(F\) annihilates \(\mathcal{T}\). The resulting equivariant filtration reduces the calculation to coherent sheaves on the reduced curve \(F\). The normalization map is equivariant, and its comparison with a coherent sheaf has kernel and cokernel of finite length supported at the nodes. The component and point calculations above therefore prove the asserted cancellation. Comparing the quotient complex with the acyclic complex \(C_{-Z}^\bullet\), we obtain finite-dimensional hypercohomology groups and
\[
 \sum_k(-1)^k[\mathbb{H}^k(Y,C_D^\bullet)]=0.
\]

By \eqref{eq:coherent-vanishing}, the hypercohomology spectral sequence
\[
 E_1^{p,q}=H^q(Y,C_D^p)
 \ \Rightarrow\ \mathbb{H}^{p+q}(Y,C_D^\bullet)
\]
has only the rows \(q=0,1\). The second row consists of finite-dimensional spaces, and \(E_2^{p,0}=B^p\). The only possible higher differential is \(d_2\colon E_2^{0,1}\to E_2^{2,0}\), whose source is finite dimensional. Since the abutment is finite dimensional, every \(B^p\) is finite dimensional as well. Comparing Euler characteristics gives
\[
 0
 =\sum_{p=0}^2(-1)^p[B^p]
 -\sum_{p=0}^2(-1)^p[H^1(Y,C_D^p)],
\]
which proves \eqref{genus:eq:euler}.

Finally, a meromorphic function with zero differential is constant. A nonzero constant is a section of \(\mathcal{O}_Y(D)\) exactly when \(D\ge 0\), and \(G\) acts trivially on constants.
\end{proof}

For a good resolution \((Y,F)\), let \(Z_K\) denote the rational exceptional cycle determined by
\[
 (K_Y+Z_K)\cdot F_i=0
\]
for every component \(F_i\). The intersection matrix is negative definite, so this cycle exists and is unique. The singularity is numerically Gorenstein if \(Z_K\) is integral.

We recall the following standard property of the numerical anticanonical cycle; see Popescu-Pampu~\cite[Corollary~2.8 and its proof]{PopescuPampu2011}.

\begin{lemma}\label{genus:lem:cycle}
Let \((Y,F) \to (V,0)\) be the minimal good resolution of a numerically Gorenstein normal complex surface singularity. If \((V,0)\) is a rational double point, then \(Z_K=0\). Otherwise, \(Z_K\ge F\).
\end{lemma}

\begin{proof}
If \((V,0)\) is a rational double point, its minimal good resolution is crepant, so \(Z_K=0\). Suppose that \((V,0)\) is not a rational double point. Let \(p\colon Y\to Y_0\) be the morphism to the minimal resolution, and let \(F_0\) be its reduced exceptional divisor. By the projection formula, \(Z_{K,0}:=p_*Z_K\) is the numerical anticanonical cycle on \(Y_0\), and it is integral. Popescu-Pampu--Seade~\cite[Proposition~2.4]{PopescuPampuSeade2009} gives
\[
 Z_{K,0}\ge F_0.
\]

Make the reduced exceptional divisor simple normal crossings by successively blowing up points where this condition fails. Consider one such blowup \(b\colon T'\to T\), with center \(x\) and new exceptional curve \(C_{\mathrm{new}}\). Suppose that the numerical anticanonical cycle before the blowup is \(Z=\sum_i z_iC_i\), with every \(z_i\ge 1\). Since \(K_{T'}=b^*K_T+C_{\mathrm{new}}\), the new numerical anticanonical cycle is
\[
 Z'=b^*Z-C_{\mathrm{new}}.
\]
Its coefficient along \(C_{\mathrm{new}}\) is
\[
 \sum_i z_i\mult_x(C_i)-1
 \ge \sum_i\mult_x(C_i)-1
 \ge 1.
\]
The last inequality holds because the reduced exceptional divisor is not simple normal crossings at \(x\), so its multiplicity there is at least two. The coefficients on the strict transforms remain unchanged. Thus the numerical anticanonical cycle remains at least the reduced exceptional divisor throughout this process.

The resulting good resolution dominates \(Y\). By the projection formula, its numerical anticanonical cycle pushes forward to \(Z_K\), with unchanged coefficients on the components that remain. Hence \(Z_K\ge F\).
\end{proof}

We first compare closed one-forms with forms on the exceptional curve. The divisor in the following lemma need not be effective.

\begin{lemma}\label{genus:lem:closed-restriction}
Let \(\pi\colon(Y,F)\to(V,0)\) be a \(G\)-equivariant good resolution of a normal complex surface singularity, where \(G\) is finite, and let \(D\) be a \(G\)-invariant integral exceptional divisor. Let \((Y_0,F_0)\to(V,0)\) be the minimal good resolution. For the complex \(C_D^\bullet\) of Lemma~\ref{genus:lem:euler}, restriction of closed one-forms induces a \(G\)-equivariant injection
\[
 H^1(\Gamma(Y,C_D^\bullet))
 \hookrightarrow H^0(F_0,\omega_{F_0}).
\]
For a resolution of a smooth surface point, the group on the left is zero.
\end{lemma}

\begin{proof}
Put \(B^1:=H^1(\Gamma(Y,C_D^\bullet))\), and let \(p\colon Y\to Y_0\) be the natural morphism. Uniqueness of the minimal good resolution gives the induced \(G\)-action on \(Y_0\) and makes \(p\) equivariant. A class in \(B^1\) is represented by a closed meromorphic one-form \(\alpha\) on \(Y\), holomorphic off \(F\). By van Straten--Steenbrink~\cite[Corollary~(1.4)]{VanStratenSteenbrink1985}, its restriction to the punctured surface extends to a holomorphic one-form \(\alpha_0\) on \(Y_0\); see also J\"{o}rder~\cite[Proposition~1.2]{JorderExtension}. The extension is unique and remains closed, since its exterior derivative vanishes off \(F_0\). Moreover, \(p^*\alpha_0=\alpha\), since the two forms agree off \(F\).

Every section of \(\mathcal{O}_Y(D)\) restricts to a holomorphic function on the punctured normal surface and extends across the singular point. Its pullback to \(Y_0\) is constant on \(F_0\), so its differential restricts to zero on every component of \(F_0\). Thus restriction is well defined on \(B^1\). The restrictions of \(\alpha_0\) to the components are regular one-forms. Their residues at the preimages of every node are zero, so they satisfy the residue compatibility condition defining \(\omega_{F_0}\). We obtain a map
\[
 B^1\to H^0(F_0,\omega_{F_0}).
\]
Uniqueness of extension and naturality of restriction make this map equivariant.

Suppose that the image of \([\alpha]\) is zero. Then \(\alpha_0\) restricts to zero on every component of \(F_0\). As in the period argument of Lemma~\ref{lem:closed-one-forms}, loops in \(F_0\) can be represented by piecewise smooth paths lying in its components. All their integrals are zero, including those arising from cycles in the dual graph. Since a sufficiently small resolution neighborhood retracts onto \(F_0\), all periods of \(\alpha_0\) vanish. It therefore has a global holomorphic primitive on \(Y_0\). Normality gives \((\pi_0)_*\mathcal{O}_{Y_0}=\mathcal{O}_V\), where \(\pi_0\colon Y_0\to V\), so this primitive descends to a holomorphic function \(f\) on \(V\).

Put \(h:=\pi^*(f-f(0))\). Then \(\alpha=dh\). If \(h=0\), the assertion follows immediately. Otherwise, \(h\) has positive order along every exceptional component. Let \(C\) be a component on which \(D\) has coefficient \(-a<0\). Near a general point of \(C\), choose coordinates with \(C=(x=0)\), and write \(h=x^m u\), where \(m>0\) and \(u\) is a unit. The coefficient of \(dx/x\) in \(dh\) is
\[
 x^m(mu+x\partial_xu),
\]
which has order exactly \(m\) along \(C\). Since \(dh\in H^0(Y,\Omega_Y^1(\log F)(D))\), this coefficient must vanish to order at least \(a\). Hence \(m\ge a\). This holds for every component with a negative coefficient in \(D\), so \(h\in H^0(Y,\mathcal{O}_Y(D))\). Consequently \([\alpha]=0\) in \(B^1\).

If \((V,0)\) is smooth, Hartogs extension gives a holomorphic closed one-form on a sufficiently small ball in \(V\). The holomorphic Poincar\'{e} lemma gives a primitive there. The same vanishing-order calculation on \(Y\) shows that its pullback, after subtracting its value at the origin, belongs to \(H^0(Y,\mathcal{O}_Y(D))\). Thus \(B^1=0\).
\end{proof}

\begin{proposition}\label{genus:prop:polar}
Let \(\pi\colon(Y,F)\to(V,0)\) be a \(G\)-equivariant good resolution of a numerically Gorenstein normal complex surface singularity, where \(G\) is finite. Then the class
\[
 [H^1(Y,\mathcal{O}_Y(Z_K-F))]
 -[H^1(Y,\Omega_Y^1(\log F)(Z_K-F))]
 +[H^1(Y,\omega_Y(Z_K))]
\]
is effective in \(R(G)\). The same assertion holds for a resolution of a smooth surface point.
\end{proposition}

\begin{proof}
If \(F=\varnothing\), all three cohomology groups vanish by Stein vanishing. Suppose that \(F\neq\varnothing\). Put \(D:=Z_K-F\), and write
\[
 B^i:=H^i(\Gamma(Y,C_D^\bullet))
\]
for the cohomology groups of the global-section complex in Lemma~\ref{genus:lem:euler}. Let \(\pi_0\colon(Y_0,F_0)\to(V,0)\) be the minimal good resolution, and let \(p\colon Y\to Y_0\) be the natural morphism. If \((V,0)\) is smooth, take \(Y_0=V\) and \(F_0=\varnothing\). The morphism \(p\) is equivariant. Put \(W:=H^0(F_0,\omega_{F_0})\), with \(W=0\) when \(F_0=\varnothing\). By Lemma~\ref{genus:lem:closed-restriction}, \(B^1\) injects equivariantly into \(W\).

If \((V,0)\) is smooth or a rational double point, then \(W=0\), so \(B^1=0\). Equation~\eqref{genus:eq:euler} expresses the class in the statement as \([B^0]+[B^2]\), which is effective.

Suppose now that \((V,0)\) is neither smooth nor a rational double point. Let \(Z_{K,0}\) be the numerical anticanonical cycle on \(Y_0\). By Lemma~\ref{genus:lem:cycle}, \(Z_{K,0}\ge F_0\). The transformation formula for numerical anticanonical cycles gives
\[
 Z_K=p^*Z_{K,0}-K_{Y/Y_0}.
\]
Together with the canonical bundle formula, this yields an isomorphism
\[
 \omega_Y(Z_K)\simeq p^*\omega_{Y_0}(Z_{K,0})
\]
of \(G\)-equivariant line bundles. Since \(Z_{K,0}\ge F_0\), pullback followed by projection to cohomology defines a map \(H^0(Y_0,\omega_{Y_0}(F_0))\to B^2\). Define
\[
 \mathcal{J}:=\im\left(
 H^0(Y_0,\omega_{Y_0}(F_0))\to B^2
 \right).
\]

Adjunction and Grauert--Riemenschneider vanishing~\cite{GrauertRiemenschneider1970} give the equivariant exact sequence
\[
 0\to H^0(Y_0,\omega_{Y_0})
 \to H^0(Y_0,\omega_{Y_0}(F_0))
 \xrightarrow{\Res} W
 \to 0.
\]
If a logarithmic two-form has zero image in \(B^2\), its pullback is the differential of a one-form in \(H^0(Y,\Omega_Y^1(\log F)(D))\). On the punctured surface this primitive is holomorphic. Lemma~\ref{lem:exact-residue}, applied on \(Y_0\), therefore shows that the logarithmic two-form has zero residue. Consequently residue descends to a well-defined equivariant surjection
\[
 \Res\colon\mathcal{J}\twoheadrightarrow W.
\]

Identify \(B^1\) with its image in \(W\). The inclusions \(\mathcal{J}\subseteq B^2\) and \(B^1\subseteq W\), together with the residue surjection, give
\[
 [B^0]-[B^1]+[B^2]
 =[B^0]+[B^2/\mathcal{J}]
 +[\ker(\Res)]+[W/B^1].
\]
Every term on the right is the class of a finite-dimensional \(G\)-representation. Hence the class is effective, and \eqref{genus:eq:euler} proves the assertion.
\end{proof}

\subsection{Finite cyclic covers and the genus bound}

We construct a numerically trivial line bundle \(N\) and prove the required cohomological inequality by applying the preceding residue comparison on a resolution of a finite cyclic cover of \((V,0)\). The cohomology estimates above then give the modified inequality and its equality criterion.

\begin{lemma}\label{genus:lem:cyclic-log-cover}
Let \(\pi\colon(X,E) \to (V,0)\) be a good resolution of a normal complex surface singularity, or a resolution of a smooth surface point. Put \(U:=X\setminus E\), and let \(c\colon U' \to U\) be a connected finite unramified cyclic cover with group \(G\).

Then \(c\) extends to a finite cyclic cover \(\tau\colon(V',0') \to (V,0)\), unramified away from \(0\). There are a \(G\)-equivariant good resolution \(\pi'\colon(Y,F) \to (V',0')\) and a map \(q\colon Y \to X\) satisfying \(\pi\circ q=\tau\circ\pi'\). The map \(q\) factors as \(q=\rho\circ\sigma\), where \(\rho\colon\overline{X} \to X\) is the finite normalization in \(U'\), and \(\sigma\colon Y \to \overline{X}\) is a \(G\)-equivariant toroidal resolution.
\end{lemma}

\begin{proof}
If \(E=\varnothing\), then \(X\) is a sufficiently small smooth neighborhood of the origin. The connected cover is trivial, and we take all maps to be identities. Assume that \(E\ne\varnothing\). Choose a sufficiently small representative of \((V,0)\). Its punctured neighborhoods retract onto the link, so the connected cover remains connected after further shrinking.

The complement of \(E\) in a small coordinate neighborhood has fundamental group \(\mathbb{Z}\) at a smooth point of \(E\), and \(\mathbb{Z}^2\) at a node. Each connected local component of the cover therefore extends by normalization of
\[
 w^m=x^a y^b,
 \quad 0\le a,b<m,
 \quad \gcd(m,a,b)=1,
\]
where \(m\) is the order of the local monodromy image. At a smooth point of \(E\), take \(b=0\). These normalizations glue by uniqueness of normalization to give \(\rho\colon\overline{X} \to X\). They are smooth over smooth points of \(E\), and have cyclic quotient singularities over its nodes; cf.~Esnault--Viehweg~\cite[Lemma~3.24]{EsnaultViehweg1992}. Their minimal toric resolutions give a \(G\)-equivariant toroidal resolution \(\sigma\colon Y \to \overline{X}\). Put
\[
 q:=\rho\circ\sigma,\quad F:=(q^{-1}E)_{\mathrm{red}}.
\]

The Stein factorization of \(\pi\circ q\) gives \(Y \xrightarrow{\pi'} V' \xrightarrow{\tau} V\), with \(V'\) normal and \(\tau\) finite. Over \(V\setminus\{0\}\), it is the cover \(U' \to U\). If \(\tau^{-1}(0)\) contained more than one point, finiteness would separate the inverse image of a sufficiently small neighborhood into disjoint neighborhoods of these points. Removing the fiber would disconnect the restricted cover, a contradiction. Thus \(\tau^{-1}(0)=\{0'\}\). The \(G\)-action descends to \(V'\), and \(\tau\) is a cyclic cover. Since \(Y\) is smooth, \(F\) has simple normal crossings, and \(\pi'\) is an isomorphism away from \(F\), the map \(\pi'\colon(Y,F) \to (V',0')\) is a good resolution.
\end{proof}

\begin{lemma}\label{genus:lem:log-pullback}
For the map \(q\colon Y \to X\) constructed in Lemma~\ref{genus:lem:cyclic-log-cover}, pullback induces a \(G\)-equivariant isomorphism
\begin{equation}\label{genus:eq:logpullback}
 q^*\Omega_X^1(\log E) \simeq \Omega_Y^1(\log F).
\end{equation}
\end{lemma}

\begin{proof}
Over a node of \(E\), toric coordinates give
\[
 x=u^a v^b,\quad y=u^c v^d,\quad ad-bc\ne 0,
\]
and hence
\[
 \begin{pmatrix}dx/x\\ dy/y\end{pmatrix}
 =
 \begin{pmatrix}a&b\\ c&d\end{pmatrix}
 \begin{pmatrix}du/u\\ dv/v\end{pmatrix}.
\]
The matrix is invertible over \(\mathbb{C}\). Where only one toric boundary coordinate vanishes, the other is a unit, and its logarithmic differential is regular. Over a smooth point of \(E\), the local map is \(x=u^m,\ y=v\). Pullback sends \(dx/x\) to \(m du/u\), and \(dy\) to \(dv\). Away from \(F\), the map is unramified. These calculations prove the isomorphism; cf.~Ambro~\cite[Theorem~1.1(b)]{Ambro2016}. It is \(G\)-equivariant because it is induced by pullback.
\end{proof}

\begin{lemma}\label{genus:lem:cyclic-descent}
For the map \(q\colon Y \to X\) constructed in Lemma~\ref{genus:lem:cyclic-log-cover} and every vector bundle \(\mathcal{A}\) on \(X\), pullback induces isomorphisms
\begin{equation}\label{genus:eq:descent}
 H^i(X,\mathcal{A}) \simeq H^i(Y,q^*\mathcal{A})^G
 \quad(i=0,1).
\end{equation}
\end{lemma}

\begin{proof}
Write \(q=\rho\circ\sigma\) as in Lemma~\ref{genus:lem:cyclic-log-cover}. Cyclic quotient singularities are rational; see Esnault--Viehweg~\cite[proof of Claim~5.13]{EsnaultViehweg1992}. Thus
\[
 \sigma_*\mathcal{O}_Y=\mathcal{O}_{\overline{X}},
 \quad R^1\sigma_*\mathcal{O}_Y=0.
\]
Since \(\rho\) is finite and \(X=\overline{X}/G\), it follows that
\[
 R^1q_*\mathcal{O}_Y=0,\quad
 (q_*\mathcal{O}_Y)^G=\mathcal{O}_X.
\]
The projection formula and Leray give
\[
 H^i(Y,q^*\mathcal{A})
 \simeq H^i(X,\mathcal{A}\otimes q_*\mathcal{O}_Y)
 \quad(i=0,1).
\]
The Reynolds projection identifies the invariant subsheaf on the right with \(\mathcal{A}\) and commutes with cohomology. Taking invariants therefore proves \eqref{genus:eq:descent}.
\end{proof}

\begin{lemma}\label{genus:lem:horizontal-section}
Let \(q\colon Y \to X\) be the map constructed in Lemma~\ref{genus:lem:cyclic-log-cover}. Let \(P\) be a line bundle on \(X\) with a flat logarithmic connection \(\nabla\) along \(E\). Suppose that the pullback connection has trivial monodromy on \(Y\setminus F\).

Then there are a \(G\)-invariant integral divisor \(D\) supported on \(F\) and a character \(\chi\) of \(G\) such that
\[
 (q^*P,q^*\nabla) \simeq (\mathcal{O}_Y(D),d)\otimes_{\mathbb{C}}\mathbb{C}_\chi
\]
as \(G\)-equivariant line bundles with logarithmic connection. Here \(d\) is exterior differentiation, and \(\mathbb{C}_\chi\) is the one-dimensional representation with character \(\chi\).
\end{lemma}

\begin{proof}
By Lemma~\ref{genus:lem:log-pullback}, the pullback connection is logarithmic along \(F\). Its trivial monodromy gives a nonzero global horizontal section \(s\) on the connected space \(Y\setminus F\).

Near a point of \(F\), choose a local frame \(e\) of \(q^*P\) and coordinates in which \(F\) is defined by \(u=0\) or \(uv=0\). Flatness gives a connection form
\[
 \ell_1\frac{du}{u}+\ell_2\frac{dv}{v}+dh,
\]
with constant residues and \(h\) holomorphic. The second logarithmic term is omitted when \(F\) has one branch. Trivial monodromy implies that the residues are integral. The horizontal equation then gives
\[
 s=c u^{-\ell_1}v^{-\ell_2}\exp(-h)e
\]
for a nonzero constant \(c\), with the \(v\)-factor omitted in the one-branch case. Thus \(s\) extends meromorphically to \(Y\), and \(D:=\divr(s)\) is integral and supported on \(F\).

The horizontal sections on \(Y\setminus F\) form a one-dimensional vector space preserved by \(G\). Define \(\chi\) by \(g\cdot s=\chi(g)s\). Then \(D\) is \(G\)-invariant, and multiplication by \(s\) gives the equivariant identification
\[
 \mathcal{O}_Y(D)\otimes_{\mathbb{C}}\mathbb{C}_\chi
 \simeq q^*P.
\]
Since \(q^*\nabla(s)=0\), this identification carries exterior differentiation to \(q^*\nabla\).
\end{proof}

\begin{proposition}\label{genus:prop:twist}
Let \(\pi\colon(X,E) \to (V,0)\) be a good resolution of a normal complex surface singularity. Then there is a numerically trivial line bundle \(N\) on \(X\) such that
\[
 h^1(X,N)-h^1(X,S_X\otimes N)+h^1(X,L_X\otimes N) \geq 0.
\]
\end{proposition}

\begin{proof}
Choose a positive integer \(r\) such that \(rZ_K\) is integral. The line bundle \(\omega_X^{\otimes r}(rZ_K)\) has degree zero on every exceptional component. Since \(X\) retracts onto \(E\) and \(H^2(E,\mathbb{Z})\) is freely generated by the component degree classes, its first Chern class is zero. The exponential sequence identifies the group of numerically trivial line bundles with a quotient of \(H^1(X,\mathcal{O}_X)\). Multiplication by \(r\) is surjective on this quotient. Hence there is a numerically trivial line bundle \(N\) such that
\[
 N^{\otimes r} \simeq \omega_X^{\otimes r}(rZ_K).
\]

Put \(P:=L_X\otimes N\). Since \(L_X\simeq\omega_X^\vee(-E)\), we have
\[
 P^{\otimes r} \simeq \mathcal{O}_X(r(Z_K-E)).
\]
Transport the exterior-differentiation connection to \(P^{\otimes r}\) through this isomorphism. Dividing its local connection forms by \(r\) defines a flat logarithmic connection \(\nabla\) on \(P\). Indeed, the connection forms on \(P^{\otimes r}\) change by \(r\) times the logarithmic differentials of the transition functions of \(P\). The \(r\)th tensor power of \((P,\nabla)\) is trivial as a flat line bundle on \(U:=X\setminus E\), so its monodromy is finite.

Let \(U' \to U\) be the connected cover corresponding to the kernel of this monodromy representation. Lemma~\ref{genus:lem:cyclic-log-cover} gives a good resolution \(\pi'\colon(Y,F) \to (V',0')\), a map \(q\colon Y \to X\), and a covering group \(G\). The pullback connection has trivial monodromy on \(Y\setminus F\). By Lemma~\ref{genus:lem:horizontal-section}, there are a \(G\)-invariant integral divisor \(D\) supported on \(F\) and a character \(\chi\) such that
\[
 (q^*P,q^*\nabla)
 \simeq (\mathcal{O}_Y(D),d)\otimes_{\mathbb{C}}\mathbb{C}_\chi.
\]
The natural rank-two identity
\[
 L_X\otimes\Omega_X^1(\log E)\simeq S_X
\]
and the logarithmic pullback isomorphism \eqref{genus:eq:logpullback} give the following \(G\)-equivariant identifications:
\[
 \begin{aligned}
 q^*(L_X\otimes N)
 &\simeq \mathcal{O}_Y(D)\otimes_{\mathbb{C}}\mathbb{C}_\chi,\\
 q^*(S_X\otimes N)
 &\simeq \Omega_Y^1(\log F)(D)\otimes_{\mathbb{C}}\mathbb{C}_\chi,\\
 q^*N
 &\simeq \omega_Y(F+D)\otimes_{\mathbb{C}}\mathbb{C}_\chi.
 \end{aligned}
\]

The line bundle \(q^*N\) is numerically trivial: each exceptional curve on \(Y\) maps either to an exceptional curve on \(X\) or to a point. The last identification therefore gives
\[
 (K_Y+F+D)\cdot C=0
\]
for every exceptional component \(C\) on \(Y\). By uniqueness of the numerical anticanonical cycle,
\[
 Z_{K,Y}=F+D.
\]
This cycle is integral, so \((V',0')\) is numerically Gorenstein, or is a smooth surface point.

Proposition~\ref{genus:prop:polar}, together with the three identifications above, shows that
\[
 [H^1(Y,q^*(L_X\otimes N))]
 -[H^1(Y,q^*(S_X\otimes N))]
 +[H^1(Y,q^*N)]
 \in R(G)
\]
is effective. Here tensoring by \(\mathbb{C}_\chi\) preserves effectivity. Taking \(G\)-invariants and applying Lemma~\ref{genus:lem:cyclic-descent} gives
\[
 h^1(X,L_X\otimes N)-h^1(X,S_X\otimes N)+h^1(X,N)
 \geq 0,
\]
as required.
\end{proof}

\subsection{Proof of Theorem~\ref{genus:thm:main}}\label{sec:proof-genus-thm-main}

We now prove Theorem~\ref{genus:thm:main}.

\begin{proof}
Let \(\pi\colon(X,E) \to (V,0)\) be the minimal good resolution, and put \(p_g:=p_g(V,0)=h^1(\mathcal{O}_X)\). Choose \(N\) as in Proposition~\ref{genus:prop:twist}, and define
\[
 \Delta_N:=h^1(N)-h^1(S_X\otimes N)+h^1(L_X\otimes N).
\]
Then \(\Delta_N\geq 0\).

Applying Lemma~\ref{genus:lem:variation} to \(S_X\) and \(L_X\), of ranks \(2\) and \(1\), respectively, gives
\begin{equation*}
\begin{split}
 \Delta(V,0)+4p_g
 &=\Delta_N+5p_g-h^1(N)\\
 &\quad+h^1(S_X\otimes N)-h^1(S_X)\\
 &\quad+h^1(L_X)-h^1(L_X\otimes N)\\
 &\geq\Delta_N+2p_g-h^1(N).
\end{split}
\end{equation*}
If \(p_g=0\), Lemma~\ref{genus:lem:strict} gives \(h^1(N)=0\), so \(\Delta(V,0)\geq\Delta_N\geq 0\). If \(p_g>0\), the same lemma gives \(h^1(N)\leq 2p_g-1\), and therefore
\[
 \Delta(V,0)+4p_g\geq\Delta_N+1\geq 1.
\]
Thus the inequality holds in all cases, and equality forces \(p_g=0\).

The condition \(p_g=0\) is equivalent to rationality. For a rational singularity other than a rational double point, Theorem~\ref{thm:rational-conjecture} gives \(\Delta(V,0)=0\) if and only if \((V,0)\) is quasi-homogeneous. A rational double point is rational and quasi-homogeneous, and Theorem~\ref{thm:main} gives \(\Delta(V,0)=0\).
\end{proof}

\section{Riemann--Roch and specialization of canonical powers}\label{sec:algebra}

We study the specialization maps for reflexive canonical powers considered by Lee--Nakayama~\cite[Lemma~7.3]{LeeNakayama2018}. For a smoothing of a normal surface singularity, these maps are injective and have cokernels of finite length. Their lengths \(\delta_i(\pi)\), recalled below, measure the failure of reflexive canonical powers to commute with restriction to the central fiber. For rational surface singularities, Theorem~\ref{thm:rr} computes these lengths by classical Riemann--Roch. We combine this calculation with the finite specialization criterion of Lee--Nakayama to obtain numerical criteria for the weak and strong \(\mathbb{Q}\)-Gorenstein conditions. We conclude with an example showing that the weak condition does not imply Wahl's inequality.

Let \(\pi\colon(\mathcal{V},0) \to (T,0)\) be a one-parameter smoothing of a normal complex surface singularity \((V,0)\), over a smooth base with parameter \(t\). Put
\[
 R=\mathcal{O}_{\mathcal{V},0}, A=\mathcal{O}_{V,0}=R/(t),
 M_i=\omega_R^{[i]}, N_i=\omega_A^{[i]}
  (i \in \mathbb{Z}).
\]
Flatness and the Cohen--Macaulay property of \(A\) show that \(R\) is Cohen--Macaulay of dimension three. Since \(R\) is regular away from the closed point, Serre's criterion also gives normality. Write \(\mu_\pi=b_2(M_{\pi})\) and let \(\tau_\pi\) be the dimension of the smoothing component containing \(\pi\). The specialization construction and duality below do not require rationality.

\subsection{Specialization and duality}
\label{sec:specialization}

We first consider restriction to the central fiber. Reflexivity gives \(\dep_R M_i \ge 2\), and hence \(\dep_A(M_i/tM_i) \ge 1\). Adjunction away from the closed point induces \(M_i/tM_i \to N_i\). Its kernel has finite length and therefore vanishes. We obtain
\begin{equation}\label{eq:specialization}
 0 \to M_i/tM_i \to N_i \to C_i \to 0,
\end{equation}
where \(C_i\) has finite length. Write
\[
 \delta_i(\pi)=\ell_A C_i
\]
for the specialization defect of the \(i\)-th reflexive canonical power.

We record the following standard depth calculation; see The Stacks Project Authors~\cite[Tags~090R, 0AVZ, and 0AWS]{Stacks}.

\begin{lemma}\label{prop:depth}
For every \(i \in \mathbb{Z}\),
\[
 C_i \simeq \ker(t:H^2_{\mathfrak{m}_R}(M_i)
 \to H^2_{\mathfrak{m}_R}(M_i)),
 \delta_i=0 \Longleftrightarrow \dep_R M_i=3.
\]
In particular, \(\delta_0=\delta_1=0\).
\end{lemma}

\begin{proof}
Since \(N_i\) has depth two, local cohomology of \eqref{eq:specialization} gives \(C_i \simeq H^1_{\mathfrak{m}_A}(M_i/tM_i)\). The sequence \(0 \to M_i \xrightarrow{t} M_i \to M_i/tM_i \to 0\) identifies this group with the displayed kernel. The regular-element criterion gives \(\dep_R M_i=\dep_A(M_i/tM_i)+1\), so \(C_i=0\) is equivalent to depth three. Finally, \(M_0=R\) and \(M_1=\omega_R\) are maximal Cohen--Macaulay.
\end{proof}

Wahl~\cite[\S2, p.~319]{Wahl2015} defines \(\alpha_\pi\) as the colength of specialization of the dual relative dualizing module. Thus \(\alpha_\pi=\delta_{-1}\) in our notation.

The symmetry below is a consequence of local duality and change of rings; see The Stacks Project Authors~\cite[Tags~0AAK, 0A70, and 0A7Q]{Stacks}. We include the argument to identify the relevant cokernels.

\begin{lemma}\label{prop:duality}
For every \(i \in \mathbb{Z}\),
\begin{equation}\label{eq:duality}
 \delta_i=\delta_{1-i}.
\end{equation}
In particular,
\begin{equation*}%\label{eq:alpha}
 \alpha_\pi=\delta_{-1}=\delta_2.
\end{equation*}
\end{lemma}

\begin{proof}
The module \(\Ext^1_R(M_i,\omega_R)\) has finite length, since \(M_i\) is invertible off the closed point. Reflexivity and canonical duality in codimension one give \(\Hom_R(M_i,\omega_R)=M_{1-i}\). Change of rings gives
\[
 \Ext^{q+1}_R(M_i/tM_i,\omega_R)
 \simeq \Ext^q_A(M_i/tM_i,\omega_A) (q \ge 0).
\]
Dualizing \eqref{eq:specialization} over \(A\), and using the Cohen--Macaulay property of \(N_i\), gives
\[
 \begin{gathered}
 \Hom_A(M_i/tM_i,\omega_A)=N_{1-i},\\
 \Ext^1_A(M_i/tM_i,\omega_A)=\Ext^2_A(C_i,\omega_A)=D_A(C_i),
 \end{gathered}
\]
where \(D_A\) denotes Matlis duality. Applying \(\Hom_R(-,\omega_R)\) to \(0 \to M_i \xrightarrow{t} M_i \to M_i/tM_i \to 0\) therefore gives
\[
 0 \to C_{1-i} \to \Ext^1_R(M_i,\omega_R) \xrightarrow{t} \Ext^1_R(M_i,\omega_R) \to D_A(C_i) \to 0.
\]
Here \(\Ext^2_R(M_i,\omega_R)=0\) by local duality and \(\dep_R M_i \ge 2\). The kernel and cokernel of multiplication by \(t\) on \(\Ext^1_R(M_i,\omega_R)\) have equal lengths, and Matlis duality preserves length. This proves \eqref{eq:duality}.
\end{proof}

If \(A\) is not a rational double point, the smoothing formula of Wahl~\cite[Corollary~2.5]{Wahl2015} reads
\begin{equation}\label{eq:wahl}
 \Delta(A)=1+\mu_\pi-\tau_\pi+\delta_{-1},\quad
 \tau_\pi=1+\mu_\pi+\delta_{-1}-\Delta(A).
\end{equation}

\subsection{Riemann--Roch and periodicity}
\label{sec:rr}

The following calculation applies the classical Riemann--Roch method for smoothings; see Looijenga~\cite[Theorem~3.3]{Looijenga1986}. We compute the local contribution using the full pullback of each reflexive canonical power.

Throughout this subsection, assume that \((V,0)\) is rational, and recall that \(A=\mathcal{O}_{V,0}\). For a good resolution \(p\colon(X,E) \to (V,0)\), write \(E=\sum_{a=1}^n E_a\). Let \(K=\sum_{a=1}^n k_aE_a\) be the unique rational exceptional cycle satisfying
\[
 K\cdot E_a=K_X\cdot E_a=-2-E_a^2
 \quad (a=1,\ldots,n).
\]
Existence and uniqueness follow from the negative definiteness of the intersection matrix. Fix \(i \in \mathbb{Z}\) and choose the resolution so that the full pullback \(\mathscr{L}_i=p^*(\omega_A^{[i]})/\mathrm{torsion}\) is invertible. Write
\[
 \mathscr{L}_i \simeq \mathcal{O}_X(iK_X+Z_i),
 u_i=c_1(\mathscr{L}_i)=iK+Z_i,
\]
where \(Z_i\) is integral and exceptional.

\begin{theorem}\label{thm:rr}
Let \(\pi\) be a one-parameter smoothing of a rational normal complex surface singularity. With \(n,K,Z_i,u_i\) as above, for every \(i\in\mathbb{Z}\) we have
\begin{align}
 \delta_i
 &=\frac12\left[i(i-1)(\mu_\pi-n)+(2i-1)KZ_i+Z_i^2\right]
 \label{eq:rr-expanded}\\
 &=\frac{u_i(u_i-K)}2+\binom{i}{2}(\mu_\pi-n-K^2).
 \label{eq:rr}
\end{align}
Both expressions are independent of the resolution on which the full pullback is invertible.
\end{theorem}

\begin{proof}
By Steenbrink~\cite[Theorem~1]{Steenbrink1995}, the smoothing extends, preserving the parameter, to a flat projective family \(g\colon\mathcal{Y} \to T\). After shrinking \(T\), the central fiber \(Y_0\) has the given singularity as its only singular point, and \(g\) is smooth elsewhere. The sheaf \(\omega_{\mathcal{Y}/T}^{[i]}\) is torsion-free over the smooth curve \(T\), and hence flat over \(T\). Its restriction to a smooth fiber is \(\omega_{Y_t}^{\otimes i}\). Constancy of its fiberwise Euler characteristic and the specialization sequence \eqref{eq:specialization} give
\begin{equation}\label{eq:euler}
 \delta_i
 =\chi(Y_0,\omega_{Y_0}^{[i]})
 -\chi(Y_t,\omega_{Y_t}^{\otimes i}).
\end{equation}

Extend the chosen local resolution to a projective resolution \(p\colon\overline{X} \to Y_0\), and write
\[
 \overline{\mathscr{L}}_i
 :=p^*(\omega_{Y_0}^{[i]})/\mathrm{torsion}.
\]
Its restriction near the exceptional divisor is \(\mathscr{L}_i\). Pullback of sections and reflexive extension across the singular point give \(p_*\overline{\mathscr{L}}_i\simeq\omega_{Y_0}^{[i]}\). Over an affine neighborhood of that point, relative generation gives an exact sequence
\[
 0 \to \mathscr{G}
 \to \mathcal{O}_{\overline{X}}^{\oplus m}
 \to \overline{\mathscr{L}}_i
 \to 0.
\]
Rationality gives \(R^1p_*\mathcal{O}_{\overline{X}}=0\), and the fiber dimension gives \(R^2p_*\mathscr{G}=0\). Thus
\begin{equation*}%\label{eq:full-vanishing}
 p_*\overline{\mathscr{L}}_i
 \simeq\omega_{Y_0}^{[i]},
 \quad R^1p_*\overline{\mathscr{L}}_i=0.
\end{equation*}
Consequently,
\[
 \chi(Y_0,\omega_{Y_0}^{[i]})
 =\chi(\overline{X},\overline{\mathscr{L}}_i).
\]
Rationality and flatness also give
\[
 \chi(\overline{X},\mathcal{O}_{\overline{X}})
 =\chi(Y_0,\mathcal{O}_{Y_0})
 =\chi(Y_t,\mathcal{O}_{Y_t}).
\]

The exceptional divisor is a tree of \(n\) rational curves, so \(e(E)=n+1\). By Greuel--Steenbrink~\cite{GreuelSteenbrink1983}, the Milnor fiber has \(b_1(M_\pi)=0\). It has the homotopy type of a finite complex of dimension at most two, and hence \(e(M_\pi)=1+\mu_\pi\). Comparing a resolution neighborhood with the Milnor fiber gives
\[
 e(\overline{X})-e(Y_t)=n-\mu_\pi.
\]
Noether's formula therefore yields
\[
 K_{\overline{X}}^2-K_{Y_t}^2=\mu_\pi-n.
\]

The canonical identification away from \(E\) gives
\[
 \overline{\mathscr{L}}_i
 \simeq\mathcal{O}_{\overline{X}}(iK_{\overline{X}}+Z_i).
\]
Applying Riemann--Roch to \(\overline{X}\) and \(Y_t\) in \eqref{eq:euler}, we obtain
\[
 \delta_i
 =\frac12\left[
 i(i-1)(K_{\overline{X}}^2-K_{Y_t}^2)
 +(2i-1)K_{\overline{X}}Z_i+Z_i^2
 \right].
\]
Since \(Z_i\) is exceptional, the definition of the numerical canonical cycle gives \(K_{\overline{X}}Z_i=KZ_i\). Substitution proves \eqref{eq:rr-expanded}. Expanding \(u_i=iK+Z_i\) then gives \eqref{eq:rr}.

Let \(b\colon X' \to X\) be a blowup at a point of \(E\), with exceptional curve \(C\). The full line bundle on \(X'\) is \(b^*\mathscr{L}_i\), so
\[
 u_i'=b^*u_i,\quad K'=b^*K+C,\quad n'=n+1.
\]
Since \(C^2=-1\) and \(C\) is orthogonal to pullbacks,
\[
 n'+(K')^2=n+K^2,\quad
 u_i'(u_i'-K')=u_i(u_i-K).
\]
Any two resolutions on which the full pullback is invertible admit a common resolution by point blowups. Thus \eqref{eq:rr}, and consequently \eqref{eq:rr-expanded}, is independent of the chosen resolution.
\end{proof}

For \(i=-1\), this recovers the anticanonical specialization formula of Wahl~\cite[Theorem~1.1(i)]{Wahl2015}; the case \(i=2\) is equivalent by Lemma~\ref{prop:duality}.

The following corollary rewrites Theorem~\ref{thm:rr} using periodic correction terms. These are the local Riemann--Roch corrections for reflexive canonical powers; see Blache~\cite{Blache1995}.

\begin{corollary}\label{cor:periodic}
Let \(\pi\) be a one-parameter smoothing of a rational normal complex surface singularity \((V,0)\) of canonical index \(r>1\), and put \(A:=\mathcal{O}_{V,0}\). There are rational numbers \(\gamma_i(A)\), depending only on \(A\), such that
\begin{equation}\label{eq:periodic}
 \delta_i(\pi)
 =\binom{i}{2}(\mu_\pi-n-K^2)+\gamma_i(A)
 \quad(i\in\mathbb{Z}).
\end{equation}
They satisfy
\[
 \gamma_{i+r}(A)=\gamma_i(A),\quad
 \gamma_i(A)=\gamma_{1-i}(A),\quad
 \gamma_0(A)=\gamma_1(A)=0.
\]
Consequently,
\begin{equation}\label{eq:period-shift}
\begin{aligned}
 \delta_r(\pi)
 &=\binom{r}{2}(\mu_\pi-n-K^2),\quad
 \mu_\pi\geq n+K^2,\\
 \delta_{i+r}(\pi)-\delta_i(\pi)
 &=\left(ri+\frac{r(r-1)}2\right)(\mu_\pi-n-K^2).
\end{aligned}
\end{equation}
\end{corollary}

\begin{proof}
Choose a common resolution on which the full pullbacks of \(\omega_A^{[j]}\), for \(0\leq j<r\), are invertible. The local isomorphisms \(\omega_A^{[i+r]}\simeq\omega_A^{[i]}\) identify their full pullbacks, so \(u_{i+r}=u_i\). Define
\[
 \gamma_i(A):=\frac{u_i(u_i-K)}2.
\]
The blowup calculation in the proof of Theorem~\ref{thm:rr} shows that this number is independent of the resolution. It depends only on \(A\), and the equality \(u_{i+r}=u_i\) gives its \(r\)-periodicity. Theorem~\ref{thm:rr} gives \eqref{eq:periodic}.

Lemma~\ref{prop:depth} gives \(\delta_0=\delta_1=0\), so \eqref{eq:periodic} implies \(\gamma_0=\gamma_1=0\). Lemma~\ref{prop:duality} and the identity \(\binom{1-i}{2}=\binom{i}{2}\) then give \(\gamma_i=\gamma_{1-i}\).

Taking \(i=r\) in \eqref{eq:periodic} gives the formula for \(\delta_r\). Since \(\delta_r\geq 0\) and \(\binom{r}{2}>0\), it follows that \(\mu_\pi\geq n+K^2\). Finally, subtracting \eqref{eq:periodic} for \(i\) from the same formula for \(i+r\), and using
\[
 \binom{i+r}{2}-\binom{i}{2}
 =ri+\frac{r(r-1)}2,
\]
proves \eqref{eq:period-shift}.
\end{proof}

For the projective central surface \(Y_0\) constructed in the proof of Theorem~\ref{thm:rr}, the Riemann--Roch formula takes the form
\[
 \chi(Y_0,\omega_{Y_0}^{[i]})
 =\chi(Y_0,\mathcal{O}_{Y_0})+\binom{i}{2}K_{Y_0}^2+\gamma_i(A).
\]
Since \(Y_0\) has only one singular point, analytically isomorphic to \((V,0)\), the term \(\gamma_i(A)\) is the local Riemann--Roch correction associated with \(\omega_A^{[i]}\); see Blache~\cite{Blache1995}.

\subsection{Numerical criteria}\label{sec:cover}

Recall that a \emph{strong \(\mathbb{Q}\)-Gorenstein smoothing} is the quotient of an equivariant smoothing of the index-one cover. We omit ``strong'' when referring to this condition. The condition that \(K_R\) is \(\mathbb{Q}\)-Cartier is called weakly \(\mathbb{Q}\)-Gorenstein. In the terminology of Lee--Nakayama~\cite[Definition~7.1 and Remark~7.6]{LeeNakayama2018}, these are the \(\mathbb{Q}\)-Gorenstein and naively \(\mathbb{Q}\)-Gorenstein conditions, respectively. The strong condition requires every reflexive canonical power to commute with restriction to the central fiber, or equivalently \(\delta_i(\pi)=0\) for every \(i \in \mathbb{Z}\).

Assume that \((V,0)\) has canonical index \(r>1\), and recall that \(A=\mathcal{O}_{V,0}\). The following proposition is the local analytic form of Lee--Nakayama~\cite[Lemma~7.3 and Proposition~7.11(i)]{LeeNakayama2018}. We include a proof using Lemma~\ref{prop:depth} and the cyclic-cover construction of Section~\ref{qg:sec:relative-cover}.

\begin{proposition}\label{thm:finite}
Let \(\pi\colon(\mathcal{V},0) \to (T,0)\) be a smoothing of a normal complex \(\mathbb{Q}\)-Gorenstein surface singularity \((V,0)\) of canonical index \(r>1\). Then \(\pi\) is \(\mathbb{Q}\)-Gorenstein if and only if \(\delta_i(\pi)=0\) for every \(2 \le i \le r\). When these conditions hold, the total space \((\mathcal{V},0)\) has canonical index \(r\), every \(M_i\) is maximal Cohen--Macaulay, and the natural restriction maps
\[
 M_i/tM_i \to N_i
 \quad(i \in \mathbb{Z})
\]
are isomorphisms.
\end{proposition}

\begin{proof}
Recall that \(R=\mathcal{O}_{\mathcal{V},0}\), \(A=\mathcal{O}_{V,0}=R/(t)\), \(M_i=\omega_R^{[i]}\), and \(N_i=\omega_A^{[i]}\).

Suppose that \(\delta_i(\pi)=0\) for \(2 \le i \le r\). Since \(N_r \simeq A\), the specialization sequence \eqref{eq:specialization} gives \(M_r/tM_r \simeq A\). Lifting a generator and applying Nakayama's lemma gives a surjection \(R \to M_r\). This is an isomorphism because \(R\) is a domain and \(M_r\) has rank one. Consequently,
\[
 M_{i+r} \simeq M_i
 \quad(i \in \mathbb{Z}).
\]
By Lemma~\ref{prop:depth} and \(\delta_0=\delta_1=0\), the modules \(M_0,\ldots,M_{r-1}\) are maximal Cohen--Macaulay. Periodicity gives the same assertion for every \(M_i\). Another application of Lemma~\ref{prop:depth} shows that every specialization map is an isomorphism.

Put \(s:=\indx(R)\). Since \(M_r \simeq R\), we have \(s \mid r\). Conversely, a trivialization of \(M_s\) restricts to a trivialization of the \(s\)th canonical power on \(V\setminus\{0\}\). Reflexive extension gives \(N_s \simeq A\), so \(r \mid s\). Thus \(s=r\).

Let \(q\colon(\widetilde{V},0) \to (V,0)\) be the canonical cover defined by a chosen trivialization \(N_r \simeq A\), and put \(B:=\mathcal{O}_{\widetilde{V},0}\). Choose a trivialization \(M_r \simeq R\) inducing this trivialization on the central fiber. Such a choice is possible because units of \(A\) lift to units of \(R\). Form the cyclic algebra
\[
 \widetilde{R}:=\bigoplus_{i=0}^{r-1}M_i.
\]
Its grading defines an action of \(G:=\mu_r\), with \(t\) invariant and \(\widetilde{R}^G=R\). The specialization isomorphisms identify \(\widetilde{R}/t\widetilde{R}\) with \(B\) as a \(G\)-equivariant algebra: compatibility with multiplication holds on the punctured central surface and extends by reflexivity.

The finite \(R\)-algebra \(\widetilde{R}\) is local because its quotient by \(t\) is local. It is Cohen--Macaulay because each summand \(M_i\) is maximal Cohen--Macaulay over \(R\). Since \(t\) is a nonzerodivisor and \(B\) is Gorenstein, \(\widetilde{R}\) is Gorenstein. Away from the closed point, \(R\) is regular and the cyclic cover is locally defined by \(z^r=u\), with \(u\) a unit. It is therefore \'{e}tale there. Thus \(\widetilde{R}\) is regular in codimension one and is normal by Serre's criterion. It is flat over \(T\), and its nonzero fibers are \'{e}tale covers of the smooth fibers of \(\pi\). Hence it defines the required \(G\)-equivariant smoothing of \((\widetilde{V},0)\).

Conversely, suppose that \(\pi\) is induced by a \(G\)-equivariant smoothing \(\widetilde{\pi}\colon(\widetilde{\mathcal{V}},0) \to (T,0)\) of the canonical cover, with trivial action on \(T\). Put \(\widetilde{R}:=\mathcal{O}_{\widetilde{\mathcal{V}},0}\). Since \(\widetilde{R}/t\widetilde{R} \simeq B\) and \(t\) is a nonzerodivisor, \(\widetilde{R}\) is normal and Gorenstein. Every nonidentity element of \(G\) fixes only the origin on the central fiber. Its fixed locus in \(\widetilde{\mathcal{V}}\) therefore has dimension at most one near the origin. Consequently, the quotient map is \'{e}tale in codimension one.

Choose an eigen-generator \(\eta\) of \(\omega_{\widetilde{R}}\). Equivariant adjunction identifies its character with the canonical character \(\kappa\) of \(B\), since \(t\) is invariant. On the locus where the quotient map is \'{e}tale, multiplication by \(\eta^{\otimes i}\) and descent identify the \(\kappa^{-i}\)-eigensummand of \(\widetilde{R}\) with \(M_i\).

A system of parameters of \(R\) is also a system of parameters of \(\widetilde{R}\). Since \(\widetilde{R}\) is Cohen--Macaulay, it is maximal Cohen--Macaulay as an \(R\)-module. Its eigensummands are direct summands and have the same property. They are therefore reflexive, so the preceding identifications extend across the complement of the \'{e}tale locus. Thus every \(M_i\) is maximal Cohen--Macaulay, and Lemma~\ref{prop:depth} gives \(\delta_i(\pi)=0\) for every \(i \in \mathbb{Z}\).
\end{proof}

\begin{remark}\label{prop:short-finite-criterion}
Canonical duality reduces the list of vanishing conditions: \(\pi\) is \(\mathbb{Q}\)-Gorenstein if and only if \(\delta_r=0\) and \(\delta_j=0\) for \(2 \le j \le \lceil r/2\rceil\). Indeed, the proof of Proposition~\ref{thm:finite} shows that \(\delta_r=0\) gives \(M_{i+r} \simeq M_i\) and \(\delta_{i+r}=\delta_i\). Lemma~\ref{prop:duality} then gives \(\delta_j=\delta_{r+1-j}\) for \(2 \le j<r\). The second range is empty when \(r=2\).
\end{remark}

Assume again that \(A\) is rational of canonical index \(r>1\). The following is the Riemann--Roch reformulation of the finite criterion, obtained by combining Corollary~\ref{cor:periodic} with Proposition~\ref{thm:finite}.

\begin{corollary}\label{cor:numerical}
Let \((V,0)\) be a rational normal complex surface singularity of canonical index \(r>1\), and put \(A:=\mathcal{O}_{V,0}\). A smoothing \(\pi\) of \((V,0)\) is \(\mathbb{Q}\)-Gorenstein if and only if
\[
 \mu_\pi=n+K^2,\quad
 \gamma_j(A)=0
 \quad(2 \le j<r).
\]
If \(r=2\), then
\[
 \delta_i(\pi)=\binom{i}{2}(\mu_\pi-n-K^2)
 \quad(i \in \mathbb{Z}),
\]
and \(\pi\) is \(\mathbb{Q}\)-Gorenstein if and only if \(\mu_\pi=n+K^2\).
\end{corollary}

\begin{proof}
Corollary~\ref{cor:periodic} gives
\[
 \delta_r(\pi)=\binom{r}{2}(\mu_\pi-n-K^2),\quad
 \delta_j(\pi)=\binom{j}{2}(\mu_\pi-n-K^2)+\gamma_j(A)
 \quad(2 \le j<r).
\]
Since \(r>1\), the equality \(\delta_r(\pi)=0\) is equivalent to \(\mu_\pi=n+K^2\). Under this equality, the remaining conditions \(\delta_j(\pi)=0\) for \(2 \le j<r\) are precisely \(\gamma_j(A)=0\). Proposition~\ref{thm:finite} therefore proves the first assertion.

If \(r=2\), periodicity and \(\gamma_0(A)=\gamma_1(A)=0\) give \(\gamma_i(A)=0\) for every \(i \in \mathbb{Z}\). Substitution in \eqref{eq:periodic} gives the displayed formula for \(\delta_i(\pi)\), and the first assertion gives the stated criterion.
\end{proof}

For every canonical index \(r>1\), the equality \(\mu_\pi=n+K^2\) alone characterizes the weak \(\mathbb{Q}\)-Gorenstein condition. Thus the additional vanishings of \(\gamma_j\) in Corollary~\ref{cor:numerical} distinguish the strong condition from the weak one. If \(K_R\) is \(\mathbb{Q}\)-Cartier, the classical formula of Looijenga~\cite[\S4.1(c)]{Looijenga1986} gives \(\mu_\pi=n+K^2\). Moreover, Lee--Nakayama~\cite[Proposition~7.11(iii)]{LeeNakayama2018} show that \(\indx R=\indx A=r\). We give a short proof that \(\mu_\pi=n+K^2\) is equivalent to \(K_R\) being \(\mathbb{Q}\)-Cartier, using specialization defects and Nakayama's lemma.

The following proposition gives a numerical criterion for the weak \(\mathbb{Q}\)-Gorenstein condition. We prove it directly from Corollary~\ref{cor:periodic}. For index preservation in characteristic zero, see Lee--Nakayama~\cite[Proposition~7.11(iii)]{LeeNakayama2018}.

\begin{proposition}\label{prop:weak-numerical}
Let \(\pi\colon(\mathcal{V},0) \to (T,0)\) be a smoothing of a rational normal complex surface singularity \((V,0)\) of canonical index \(r>1\). The canonical divisor \(K_{\mathcal{V}}\) of the total space is \(\mathbb{Q}\)-Cartier if and only if \(\mu_\pi=n+K^2\). When these conditions hold, \((\mathcal{V},0)\) has canonical index \(r\).
\end{proposition}

\begin{proof}
Recall that \(R=\mathcal{O}_{\mathcal{V},0}\), \(A=\mathcal{O}_{V,0}=R/(t)\), \(M_i=\omega_R^{[i]}\), and \(N_i=\omega_A^{[i]}\).

Suppose that \(\mu_\pi=n+K^2\). Corollary~\ref{cor:periodic} gives \(\delta_r(\pi)=0\). Since \(N_r \simeq A\), the specialization sequence \eqref{eq:specialization} gives \(M_r/tM_r \simeq A\). Lifting a generator and applying Nakayama's lemma gives a surjection \(R \to M_r\). This is an isomorphism because \(R\) is a domain and \(M_r\) has rank one. Thus \(K_{\mathcal{V}}\) is \(\mathbb{Q}\)-Cartier.

Put \(s:=\indx(R)\). The isomorphism \(M_r \simeq R\) gives \(s \mid r\). A trivialization of \(M_s\) restricts to a trivialization of the \(s\)th canonical power on \(V\setminus\{0\}\). Reflexive extension gives \(N_s \simeq A\), so \(r \mid s\). Consequently, \(s=r\).

Conversely, suppose that \(K_{\mathcal{V}}\) is \(\mathbb{Q}\)-Cartier, and let \(s\) be the canonical index of \(R\). The same restriction argument gives \(r \mid s\). Since \(M_s \simeq R\), Lemma~\ref{prop:depth} gives \(\delta_s(\pi)=0\). Periodicity and \(r \mid s\) give \(\gamma_s(A)=\gamma_0(A)=0\). Hence Corollary~\ref{cor:periodic} yields
\[
 0=\delta_s(\pi)=\binom{s}{2}(\mu_\pi-n-K^2).
\]
Since \(s \ge r>1\), it follows that \(\mu_\pi=n+K^2\).
\end{proof}

The specialization cokernels also have a common interpretation on the canonical covers. Under the weak \(\mathbb{Q}\)-Gorenstein hypothesis, their lengths are the character multiplicities of one finite-dimensional representation. We define this representation and then reformulate the finite and numerical criteria above.

Let \(\pi\colon(\mathcal{V},0) \to (T,0)\) be a weakly \(\mathbb{Q}\)-Gorenstein smoothing of a rational normal complex surface singularity \((V,0)\) of canonical index \(r>1\). Let \(t\) denote the pullback of a parameter on the smooth base, and recall that \(R=\mathcal{O}_{\mathcal{V},0}\), \(A=\mathcal{O}_{V,0}=R/(t)\), \(M_j=\omega_R^{[j]}\), and \(N_j=\omega_A^{[j]}\). By Proposition~\ref{prop:weak-numerical}, \(R\) has canonical index \(r\). Choose compatible trivializations of \(M_r\) and \(N_r\), and let \((\widetilde{\mathcal{V}},0)\) and \((\widetilde{V},0)\) be the resulting canonical covers. Put
\[
 \widetilde{R}:=\mathcal{O}_{\widetilde{\mathcal{V}},0}=\bigoplus_{j=0}^{r-1}M_j,
 \quad B:=\mathcal{O}_{\widetilde{V},0}=\bigoplus_{j=0}^{r-1}N_j.
\]
Let \(G:=\mu_r\), and let \(\kappa\colon G \to \mathbb{C}^*\) be the faithful canonical character of \(B\). The canonical gradings give \(G\)-actions on both algebras, with \(t\) invariant. The specialization maps induce a \(G\)-equivariant homomorphism of \(A\)-algebras \(\varphi\colon\widetilde{R}/t\widetilde{R} \to B\).

\begin{definition}\label{def:canonical-specialization-representation}
The \emph{canonical-specialization representation} of \(\pi\) is the finite-dimensional complex \(G\)-representation
\[
 \mathcal{C}_\pi:=\coker\varphi,
\]
with the action induced by the specialization homomorphism of the compatible canonical algebras.
\end{definition}

\begin{proposition}[Canonical-specialization representation criterion (CSRC)]\label{prop:canonical-specialization-representation}
With the notation above, the specialization homomorphism gives an exact sequence of \(G\)-equivariant \(A\)-modules
\[
 0 \to \widetilde{R}/t\widetilde{R} \xrightarrow{\varphi} B \to \mathcal{C}_\pi \to 0.
\]
The \(\kappa^{-j}\)-eigenspace of \(\mathcal{C}_\pi\) is the cokernel \(C_j\) in \eqref{eq:specialization}, and
\begin{equation}\label{eq:canonical-specialization-representation}
 [\mathcal{C}_\pi]=\sum_{j=0}^{r-1}\gamma_j(A)[\mathbb{C}_{\kappa^{-j}}]
 \quad\text{in }R(G).
\end{equation}
The following conditions are equivalent:
\begin{enumerate}
\item \(\pi\) is strongly \(\mathbb{Q}\)-Gorenstein;
\item \(\mathcal{C}_\pi=0\);
\item \(\widetilde{R}\) is Cohen--Macaulay;
\item \(\gamma_j(A)=0\) for every \(j\in\mathbb{Z}\).
\end{enumerate}
In the last condition, it suffices to take \(2\leq j<r\).
\end{proposition}

\begin{proof}
Proposition~\ref{prop:weak-numerical} gives \(\mu_\pi=n+K^2\). Corollary~\ref{cor:periodic} therefore gives \(\delta_j(\pi)=\gamma_j(A)\) for every \(j\), and in particular \(\delta_r=0\). Thus \(M_r/tM_r \to N_r\) is an isomorphism. A chosen trivialization of \(N_r\) lifts to one of \(M_r\), since units of \(A\) lift to units of \(R\). This gives compatible multiplications on the two cyclic algebras.

Take the direct sum of the specialization sequences \eqref{eq:specialization} for \(0\leq j<r\). The resulting injection is compatible with multiplication on the punctured central fiber. It is therefore an algebra homomorphism everywhere, since \(B\) is torsion-free over \(A\). In the canonical grading, the summands \(M_j\) and \(N_j\) have character \(\kappa^{-j}\), as in the proof of Proposition~\ref{thm:finite}. Its cokernel is consequently the direct sum of the \(C_j\) with these characters. Each \(C_j\) has finite length, and its complex dimension is \(\delta_j(\pi)=\gamma_j(A)\). This proves \eqref{eq:canonical-specialization-representation}. The character \(\kappa\) has order \(r\), so these are all the distinct irreducible characters of \(G\).

By Lemma~\ref{prop:depth}, \(\mathcal{C}_\pi=0\) if and only if every \(M_j\), for \(0\leq j<r\), is maximal Cohen--Macaulay over \(R\). Their direct sum is \(\widetilde{R}\). Since the finite local extension \(R\subseteq\widetilde{R}\) has dimension three, a system of parameters of \(R\) is also a system of parameters of \(\widetilde{R}\). Thus this module condition is equivalent to \(\widetilde{R}\) being Cohen--Macaulay. Since \(\delta_r=0\), Proposition~\ref{thm:finite} identifies the same vanishing with the strong \(\mathbb{Q}\)-Gorenstein condition. Finally, the \(r\)-periodicity of \(\gamma_j(A)\) and \(\gamma_0(A)=\gamma_1(A)=0\) give the remaining equivalence and the shorter range. \qedhere
\end{proof}

Lee--Nakayama~\cite[Fact~7.7, Lemma~7.8, and Example~7.9]{LeeNakayama2018} construct examples distinguishing the weak and strong \(\mathbb{Q}\)-Gorenstein conditions. For rational singularities whose canonical covers are not Cohen--Macaulay, see Singh~\cite[Theorem~6.1]{Singh2003}. We construct a rational surface smoothing of index three for which the weak condition does not imply Wahl's inequality.

\begin{lemma}\label{ex:lem:invariant-curve}
Let \(\zeta\) be a primitive cube root of unity, put \(E_\zeta:=\mathbb{C}/(\mathbb{Z}+\mathbb{Z}\zeta)\) and \(S:=E_\zeta\times E_\zeta\), and let \(G:=\langle g\rangle\) act by \(g(x,y)=(\zeta x,\zeta y)\). There exists a smooth connected \(G\)-invariant ample curve \(C \subset S\) containing all nine fixed points, with
\[
 C/G \simeq \mathbb{P}^1,\quad
 g(C)=7,\quad C^2=12.
\]
\end{lemma}

\begin{proof}
Let \(b\colon\widehat{S} \to S\) be the blowup of the nine fixed points. By Roulleau--Urz\'{u}a~\cite[Proposition~3.1]{RoulleauUrzua2015}, the quotient \(Y:=\widehat{S}/G\) is the blowup of \(\mathbb{P}^2\) at the twelve triple points of the dual Hesse arrangement. The quotient map is branched over the strict transforms of its nine lines.

Let \(\ell \subset Y\) be the strict transform of a general plane line. Its inverse image \(\widehat{C}\) is a smooth connected triple cover of \(\ell \simeq \mathbb{P}^1\), totally ramified at nine points. Riemann--Hurwitz gives
\[
 2g(\widehat{C})-2=3(-2)+9(3-1)=12.
\]
Moreover, \(\widehat{C}^2=3\), and \(\widehat{C}\) meets each exceptional curve of \(b\) transversely at one point. Consequently, \(C:=b(\widehat{C})\) is smooth, contains all nine fixed points, and satisfies \(C^2=3+9=12\). Its quotient is \(\ell\), and its genus is seven. An irreducible curve of positive square on an abelian surface is ample.
\end{proof}

\begin{example}\label{ex:weak-qg}
There exists a weighted homogeneous rational normal complex surface singularity \((V,0)\) of canonical index three admitting a weakly \(\mathbb{Q}\)-Gorenstein smoothing \(\pi\) whose total space has a rational singularity at the origin, with
\[
 \Delta(V,0)=0,\quad \mu_\pi=3,\quad
 \alpha_\pi=2,\quad \tau_\pi=6.
\]
Thus \(1+\mu_\pi-\tau_\pi=-2\), so the strong hypothesis in Theorem~\ref{thm:QG-conjecture} cannot be replaced by the weak condition.
\end{example}

\begin{proof}
Choose \(\zeta,S,G,C\) as in Lemma~\ref{ex:lem:invariant-curve}, and put \(L:=\mathcal{O}_S(C)\). We use the action \(g\cdot\alpha=(g^{-1})^*\alpha\). Write \(\mathbb{C}_{\zeta^j}\) for the one-dimensional representation on which \(g\) acts by multiplication by \(\zeta^j\).

\emph{Construction of the smoothing.} Give \(L\) its divisor linearization, so that the section \(f\) defining \(C\) is invariant. Define the graded algebras
\[
 \mathscr{B}:=\bigoplus_{d \ge 0}H^0(S,L^d),\quad
 \mathscr{R}:=\mathscr{B}^G,\quad
 \mathscr{A}:=\mathscr{R}/(f).
\]
Let \(\mathcal{V}\) be the complex analytic space associated with \(\Spec\mathscr{R}\), and put \(V:=\{f=0\} \subset \mathcal{V}\). At the vertices, write
\[
 R:=\mathcal{O}_{\mathcal{V},0},\quad
 A:=\mathcal{O}_{V,0}=R/(f).
\]
The punctured section cone is identified with the complement of the zero section in \(\Tot(L^{-1})\); see Watanabe~\cite[\S1]{Watanabe1981}. At a fixed point of \(S\), choose an eigen-equation \(z\) of \(C\) and a fiber coordinate \(w\) on \(L^{-1}\). Then
\[
 g\cdot z=\zeta^2z,\quad
 g\cdot w=\zeta w,\quad f=zw.
\]
Thus \(G\) acts freely on the punctured cone. For \(f\ne 0\), the fiber coordinate is determined by the base point, and the fiber is isomorphic to \((S\setminus C)/G\). In particular, every nonzero fiber is smooth. Since \(\mathscr{R}\) is torsion-free over \(\mathbb{C}[f]\), the map \(\pi:=f\colon(\mathcal{V},0) \to (\mathbb{C},0)\) is flat.

\emph{Normality and rationality.} Kodaira vanishing and Serre duality give
\[
 H^1(S,L^d)=0\quad(d\ne 0),\quad
 H^1(S,\mathcal{O}_S)=\mathbb{C}_{\zeta}^{\oplus2},\quad
 H^2(S,\mathcal{O}_S)=\mathbb{C}_{\zeta^2}.
\]
Put \(L_C:=L|_C\) and \(T_C:=\bigoplus_{d \ge 0}H^0(C,L_C^d)\). The restriction sequences give
\[
 0 \to \mathscr{B}/(f) \to T_C
 \to H^1(S,\mathcal{O}_S) \to 0,
\]
where the last term is concentrated in degree one. Taking invariants gives \(\mathscr{A}\simeq T_C^G\). The full section ring \(T_C\) is normal, so \(A\) is normal. Its positive grading defines a good \(\mathbb{C}^*\)-action.

Set \(Z:=\Tot(L^{-1})\). Since \(Z \to S\) is affine,
\[
 H^j(Z/G,\mathcal{O}_{Z/G})
 \simeq
 \left(\bigoplus_{d \ge 0}H^j(S,L^d)\right)^G
 =0\quad(j>0).
\]
The contraction \(Z/G \to \Spec\mathscr{R}\) is proper and birational. The singularities of \(Z/G\) are cyclic quotient singularities and hence rational. Resolving them preserves the displayed cohomology groups. Since \(\Spec\mathscr{R}\) is affine, the higher direct images of the structure sheaf of this resolution vanish. Thus \((\mathcal{V},0)\) is rational.

For the central curve, we have
\[
 H^1(C,\mathcal{O}_C)^G=0,\quad
 H^1(C,L_C)\simeq H^2(S,\mathcal{O}_S)
 =\mathbb{C}_{\zeta^2}.
\]
The first equality follows from \(C/G\simeq\mathbb{P}^1\), and the second follows from the restriction sequence for \(L\). For \(d\ge 2\), Serre duality gives \(H^1(C,L_C^d)=0\), since \(\deg L_C=12=2g(C)-2\). Consequently,
\[
 H^1(\Tot(L_C^{-1})/G,
 \mathcal{O}_{\Tot(L_C^{-1})/G})
 \simeq
 \left(\bigoplus_{d \ge 0}H^1(C,L_C^d)\right)^G
 =0.
\]
Resolving the cyclic quotient singularities and applying the same contraction argument proves that \((V,0)\) is rational.

\emph{Canonical powers and specialization.} Let \(\mathscr{B}_j\) be the character-\(\zeta^j\) summand of \(\mathscr{B}\), for \(j\in\mathbb{Z}/3\mathbb{Z}\). Watanabe~\cite[Theorem~2.8]{Watanabe1981} identifies the canonical module with
\[
 \omega_{\mathscr{B}}
 \simeq \mathscr{B}\otimes_{\mathbb{C}}\mathbb{C}_{\zeta}.
\]
Indeed, a nonzero holomorphic two-form on \(S\) has character \(\zeta\), and the logarithmic differential in the fiber direction has trivial character. Canonical descent therefore gives
\[
 \omega_{\mathscr{R}}\simeq\mathscr{B}_2,\quad
 \omega_{\mathscr{R}}^{[2]}\simeq\mathscr{B}_1,\quad
 \omega_{\mathscr{R}}^{[3]}\simeq\mathscr{R}.
\]
These identifications also hold after passage to analytic local rings. The canonical character has order three, and an eigenunit has trivial character upon evaluation at the vertex. Thus \(R\) has canonical index three.

Restriction to the punctured central surface and reflexive extension show that the canonical index of \(A\) divides three. It cannot be one: if \(A\) were Gorenstein, Gorensteinness would lift across the regular element \(f\) in the Cohen--Macaulay ring \(R\). Hence \(A\) also has index three.

Put \(\mathfrak{n}:=\mathscr{B}_{>0}\). By Watanabe~\cite[Proposition~2.2]{Watanabe1981},
\[
 (H^2_{\mathfrak{n}}(\mathscr{B}))_d
 \simeq H^1(S,L^d)
 =
 \begin{cases}
 \mathbb{C}_{\zeta}^{\oplus2},&d=0,\\
 0,&d\ne 0.
 \end{cases}
\]
In particular, the covering cone has depth two. Taking the character-\(\zeta\) summand computes the second local cohomology of \(\omega_{\mathscr{R}}^{[2]}\). This finite-length module is preserved by localization and flat passage to the analytic local ring. Hence
\[
 H^2_{\mathfrak{m}_R}(\omega_R^{[2]})
 \simeq \mathbb{C}^{\oplus2}.
\]
Multiplication by \(f\) raises degree by one and annihilates this module. Lemmas~\ref{prop:depth} and~\ref{prop:duality} therefore give \(\delta_2(\pi)=\alpha_\pi=2\).

\emph{Smoothing invariants.} Positive weighted dilation identifies the homotopy type of the local Milnor fiber with that of the global fiber \(f^{-1}(1)\simeq(S\setminus C)/G\). Since the action on \(S\setminus C\) is free,
\[
 e(M_\pi)=\frac{e(S)-e(C)}3
 =\frac{0-(-12)}3=4.
\]
The Gysin sequence identifies \(H^1(S\setminus C,\mathbb{Q})\) with \(H^1(S,\mathbb{Q})\), because the map \(H^0(C,\mathbb{Q}) \to H^2(S,\mathbb{Q})\) sends \(1\) to the nonzero class \([C]\). The invariant part of \(H^1(S,\mathbb{Q})\) is zero. Thus \(b_1(M_\pi)=0\). Since the Milnor fiber is Stein and has homotopy dimension at most two, \(e(M_\pi)=1+\mu_\pi\), and hence \(\mu_\pi=3\).

The singularity \((V,0)\) is quasi-homogeneous and has canonical index three, so Theorem~\ref{thm:main} gives \(\Delta(V,0)=0\). Wahl's formula \eqref{eq:wahl} then gives
\[
 \tau_\pi=1+\mu_\pi+\alpha_\pi-\Delta(V,0)=6. \qedhere
\]
\end{proof}

\section{Smoothing components and applications}\label{sec:consequences}

We compare smoothing components by their specialization of reflexive canonical powers defects and Milnor numbers, characterize strong \(\mathbb{Q}\)-Gorenstein smoothings by minimality, and derive applications to weak smoothings, Gorenstein modifications, and projective surface smoothings.

Let \((V,0)\) be a normal complex surface singularity which is not a rational double point, and put \(A:=\mathcal{O}_{V,0}\). We first record a consequence of Wahl's smoothing formulas~\cite[Proposition~2.2]{Wahl2015}.

\begin{corollary}\label{prop:component-comparison}
Let \(\pi\) and \(\pi'\) be one-parameter smoothings of \((V,0)\). Then
\begin{equation}\label{eq:alpha-component-comparison}
\begin{aligned}
 \alpha_\pi-\alpha_{\pi'} &=\mu_\pi-\mu_{\pi'},\\
 \tau_\pi-\tau_{\pi'} &=2(\mu_\pi-\mu_{\pi'}).
\end{aligned}
\end{equation}

Consequently, two smoothing components have the same dimension if and only if they have the same Milnor number. The dimensions of all smoothing components have the same parity.
\end{corollary}

\begin{proof}
Wahl~\cite[Proposition~2.2]{Wahl2015} expresses \(1+\mu_\pi-\alpha_\pi\) and \(\tau_\pi-2\alpha_\pi\) in terms of the cohomology and intersection numbers of the minimal good resolution of \((V,0)\). Both expressions are independent of the smoothing. Subtracting their values for \(\pi\) and \(\pi'\) gives the two equalities.
\end{proof}

Suppose now that \((V,0)\) is rational and has canonical index \(r>1\). Corollary~\ref{cor:periodic} gives the comparison for every reflexive canonical power:
\begin{equation}\label{eq:smoothing-difference}
 \delta_i(\pi)-\delta_i(\pi')
 =\binom{i}{2}(\mu_\pi-\mu_{\pi'})
 \quad(i \in \mathbb{Z}).
\end{equation}
Indeed, the correction terms \(\gamma_i(A)\) are independent of the smoothing. Taking \(i=-1\) recovers \eqref{eq:alpha-component-comparison}.

Wahl's smoothing formulas~\cite[Proposition~2.2]{Wahl2015}, together with the vanishing of \(\alpha_{\pi_0}\) for a strong \(\mathbb{Q}\)-Gorenstein smoothing (Wahl~\cite[Lemma~5.1]{Wahl2015}), imply minimality of its Milnor number and component dimension. For rational singularities, the comparison \eqref{eq:smoothing-difference} also proves the converse. The index-two case follows already from de Jong--van Straten~\cite[Lemma~1.2]{DeJongVanStraten1992}, since the weak and strong conditions coincide in that index; the argument below applies to arbitrary canonical index.

\begin{proposition}\label{prop:one-qg}
Let \((V,0)\) be a rational non-Gorenstein normal complex surface singularity admitting a strong \(\mathbb{Q}\)-Gorenstein smoothing \(\pi_0\). Then every one-parameter smoothing \(\pi\) of \((V,0)\) satisfies
\[
 \mu_\pi \geq \mu_{\pi_0},
\]
with equality if and only if \(\pi\) is strongly \(\mathbb{Q}\)-Gorenstein.
\end{proposition}

\begin{proof}
Since \(\pi_0\) is strongly \(\mathbb{Q}\)-Gorenstein, every \(\delta_i(\pi_0)\) vanishes. Equation~\eqref{eq:smoothing-difference} therefore gives
\[
 \delta_i(\pi)
 =\binom{i}{2}(\mu_\pi-\mu_{\pi_0})
 \quad(i \in \mathbb{Z}).
\]
Taking \(i=2\) and using \(\delta_2(\pi)\geq 0\) proves the inequality. Equality holds precisely when every \(\delta_i(\pi)\) vanishes. By Proposition~\ref{thm:finite}, this is equivalent to the strong \(\mathbb{Q}\)-Gorenstein condition.
\end{proof}

Combining the Milnor-number criterion with the dimension comparison gives the following characterization of the components of smallest dimension.

\begin{corollary}\label{cor:qg-minimal-components}
Let \((V,0)\) be a rational non-Gorenstein normal complex surface singularity admitting a strong \(\mathbb{Q}\)-Gorenstein smoothing. A one-parameter smoothing of \((V,0)\) is strongly \(\mathbb{Q}\)-Gorenstein if and only if its corresponding smoothing component has the smallest dimension among all smoothing components.
\end{corollary}

\begin{proof}
Choose a strong \(\mathbb{Q}\)-Gorenstein smoothing \(\pi_0\). By Corollary~\ref{prop:component-comparison},
\[
 \tau_\pi-\tau_{\pi_0}
 =2(\mu_\pi-\mu_{\pi_0})
\]
for every one-parameter smoothing \(\pi\). Proposition~\ref{prop:one-qg} shows that the right side is nonnegative and vanishes precisely when \(\pi\) is strongly \(\mathbb{Q}\)-Gorenstein. Since \(\tau_\pi\) is the dimension of the corresponding smoothing component, the assertion follows.
\end{proof}

The weak \(\mathbb{Q}\)-Gorenstein condition does not by itself impose Cohen--Macaulayness on the canonical cover of the total space; see Wahl~\cite[Remark~3.3]{Wahl2013}. The existence of one strong smoothing allows us to recover the strong condition for every weak smoothing by comparing their Milnor numbers.

\begin{corollary}\label{prop:weak-strong-with-smoothing}
Let \((V,0)\) be a rational non-Gorenstein normal complex surface singularity. If \((V,0)\) admits a strong \(\mathbb{Q}\)-Gorenstein smoothing, then every weakly \(\mathbb{Q}\)-Gorenstein one-parameter smoothing of \((V,0)\) is strongly \(\mathbb{Q}\)-Gorenstein. Consequently, a weakly \(\mathbb{Q}\)-Gorenstein smoothing which is not strong rules out the existence of a strong \(\mathbb{Q}\)-Gorenstein smoothing on any component.
\end{corollary}

\begin{proof}
Choose a strong \(\mathbb{Q}\)-Gorenstein smoothing \(\pi_0\), and let \(\pi\) be weakly \(\mathbb{Q}\)-Gorenstein. Write \(n\) for the number of exceptional components of the minimal good resolution and \(K\) for its numerical canonical cycle. Proposition~\ref{prop:weak-numerical} gives
\[
 \mu_\pi=n+K^2=\mu_{\pi_0}.
\]
Proposition~\ref{prop:one-qg} therefore shows that \(\pi\) is strongly \(\mathbb{Q}\)-Gorenstein. The last assertion is the contrapositive.
\end{proof}

Dao--Iyama--Takahashi--Wemyss~\cite[Theorem~1.3]{DaoIyamaTakahashiWemyss2020} characterize the existence of Gorenstein modifications by the \(\mathbb{Q}\)-Cartier property of the canonical divisor and the Gorenstein property of local canonical covers. In the setting below, the preceding corollary and Proposition~\ref{prop:one-qg} turn this criterion into minimality of the Milnor number.

\begin{corollary}\label{prop:gm-minimal-milnor}
Let \((V,0)\) be a rational non-Gorenstein normal complex surface singularity admitting a strong \(\mathbb{Q}\)-Gorenstein smoothing \(\pi_0\). Let \(\pi\colon(\mathcal{V},0)\to(T,0)\) be another one-parameter smoothing, and put \(R:=\mathcal{O}_{\mathcal{V},0}\). Then \(R\) admits a Gorenstein modification in the sense of \cite{DaoIyamaTakahashiWemyss2020} if and only if \(\mu_\pi=\mu_{\pi_0}\). Under these equivalent conditions, if \(r\) is the canonical index of \((V,0)\), the algebra
\[
 \End_R\left(\bigoplus_{j=0}^{r-1}\omega_R^{[j]}\right)
\]
is a Gorenstein modification of \(R\). If \(\mu_\pi>\mu_{\pi_0}\), then \(R\) admits no noncommutative crepant resolution.
\end{corollary}

\begin{proof}
The ring \(R\) is a three-dimensional Cohen--Macaulay normal local domain with a canonical module and contains \(\mathbb{C}\). If \(R\) admits a Gorenstein modification, then \(K_R\) is \(\mathbb{Q}\)-Cartier by Dao--Iyama--Takahashi--Wemyss~\cite[Theorem~3.5(2)]{DaoIyamaTakahashiWemyss2020}. Thus \(\pi\) is weakly \(\mathbb{Q}\)-Gorenstein. Corollary~\ref{prop:weak-strong-with-smoothing} shows that \(\pi\) is strong, and Proposition~\ref{prop:one-qg} gives \(\mu_\pi=\mu_{\pi_0}\).

Conversely, suppose that \(\mu_\pi=\mu_{\pi_0}\). Proposition~\ref{prop:one-qg} shows that \(\pi\) is strongly \(\mathbb{Q}\)-Gorenstein. By Proposition~\ref{thm:finite}, \(R\) has canonical index \(r\), and its canonical cover
\[
 \widetilde{R}:=\bigoplus_{j=0}^{r-1}\omega_R^{[j]}
\]
is Gorenstein. By Dao--Iyama--Takahashi--Wemyss~\cite[Theorem~3.5(1)]{DaoIyamaTakahashiWemyss2020}, \(\End_R(\widetilde{R})\) is a Gorenstein modification of \(R\). Finally, every noncommutative crepant resolution is a Gorenstein modification, so the last assertion follows.
\end{proof}

The square of the canonical divisor is preserved in a strong \(\mathbb{Q}\)-Gorenstein smoothing of a projective surface. When the singularities are rational and each admits a strong smoothing, the local Milnor-number criterion gives a converse and measures the drop in \(K^2\) for an arbitrary smoothing.

\begin{corollary}\label{prop:global-qg-numerical}
Let \(Y\) be a normal projective complex surface with rational singularities \(P_1,\dotsc,P_s\), each admitting a strong \(\mathbb{Q}\)-Gorenstein smoothing of Milnor number \(\mu_i^0\). Let \(g\colon\mathcal{Y}\to T\) be a flat projective one-parameter smoothing over a smooth pointed curve \((T,0)\), with central fiber \(Y\). Write \(Y_t:=g^{-1}(t)\), and let \(\mu_i\) be the Milnor number of the induced smoothing at \(P_i\). Then, for \(t\ne0\) sufficiently close to \(0\),
\[
 K_Y^2-K_{Y_t}^2
 =\sum_{i=1}^s(\mu_i-\mu_i^0)\geq0.
\]
Equality holds if and only if \(g\) is strongly \(\mathbb{Q}\)-Gorenstein near \(Y\). Equivalently, writing \(e\) for the topological Euler characteristic,
\[
 e(Y_t)\geq e(Y)+\sum_{i=1}^s\mu_i^0,
\]
with equality under the same condition.
\end{corollary}

\begin{proof}
At each non-Gorenstein point \(P_i\), Proposition~\ref{prop:one-qg} gives \(\mu_i\geq\mu_i^0\), with equality precisely when the induced smoothing is strongly \(\mathbb{Q}\)-Gorenstein. A rational Gorenstein surface singularity is a rational double point, and every smoothing of such a point is Gorenstein with the same Milnor number. Thus the same conclusion holds at every \(P_i\).

Let \(p\colon X\to Y\) be the minimal good resolution. Write
\[
 K_X=p^*K_Y+\sum_{i=1}^sK_i,
\]
where \(K_i\) is the numerical canonical cycle supported over \(P_i\), and let \(n_i\) be the number of exceptional components over \(P_i\). The divisor \(K_Y\) is \(\mathbb{Q}\)-Cartier because the singularities of \(Y\) are rational. Proposition~\ref{prop:weak-numerical} gives \(\mu_i^0=n_i+K_i^2\) at every non-Gorenstein point. At a rational double point, \(K_i=0\) and \(\mu_i^0=n_i\). Rationality gives \(\chi(\mathcal{O}_X)=\chi(\mathcal{O}_Y)\) and \(e(X)=e(Y)+\sum_i n_i\). Since the exceptional cycles over distinct points are disjoint and orthogonal to \(p^*K_Y\), Noether's formula yields
\[
 12\chi(\mathcal{O}_Y)
 =K_X^2+e(X)
 =K_Y^2+e(Y)+\sum_{i=1}^s\mu_i^0.
\]
Flatness gives \(\chi(\mathcal{O}_{Y_t})=\chi(\mathcal{O}_Y)\), and replacement of the singular points by their Milnor fibers gives \(e(Y_t)=e(Y)+\sum_i\mu_i\). Applying Noether's formula to \(Y_t\) and subtracting proves the identity and inequalities.

Equality holds precisely when every induced local smoothing is strong. By Proposition~\ref{thm:finite}, every reflexive canonical power then specializes at the non-Gorenstein points. The same holds at the Gorenstein points and on the smooth locus. These conditions are local, and the finitely many local canonical indices have a common multiple. They therefore give the strong \(\mathbb{Q}\)-Gorenstein condition for \(g\) near \(Y\). The converse follows by applying the same local criterion.
\end{proof}

\section{Application to rational homology disk smoothings}
\label{sec:qhd}

We show that every rational homology disk smoothing is \(\mathbb{Q}\)-Gorenstein; Theorem~\ref{thm:qhd-qg}. Applying Theorem~\ref{thm:QG-conjecture}, we then obtain quasi-homogeneity and reduce the analytic classification to the known weighted homogeneous cases.

Let \(\pi\colon(\mathcal{V},0) \to (\Delta,0)\) be a rational homology disk (QHD) smoothing of a normal complex surface singularity \((V,0)\), that is, a smoothing whose Milnor fiber has the rational homology of the \(4\)-disk. Durfee's formula then gives \(p_g(V,0)=0\), so \((V,0)\) is rational; see also Stipsicz--Szab\'{o}--Wahl~\cite[Proposition~2.4(1)]{SSW2008}.

Wahl~\cite[Proposition~3.1]{Wahl2013} proved that the total space of every rational homology disk smoothing has a \(\mathbb{Q}\)-Cartier canonical divisor. As noted in Wahl~\cite[Remark~3.3]{Wahl2013}, this does not immediately give an equivariant smoothing of the index-one cover of the central singularity. He established this stronger conclusion in the weighted homogeneous case; Wahl~\cite[Theorem~3.4]{Wahl2013}. We prove it without assuming weighted homogeneity, by showing that every reflexive canonical power specializes.

\begin{theorem}\label{thm:qhd-qg}
Let \(\pi\colon(\mathcal{V},0) \to (\Delta,0)\) be a rational homology disk smoothing of a normal complex surface singularity \((V,0)\). Put \(R:=\mathcal{O}_{\mathcal{V},0}\) and \(A:=\mathcal{O}_{V,0}\), and let \(t\) denote the pullback of a coordinate on \(\Delta\), so that \(A=R/(t)\). Then \(\pi\) is strongly \(\mathbb{Q}\)-Gorenstein, and \(R\) and \(A\) have the same canonical index. For every \(i \in \mathbb{Z}\), the natural map
\begin{equation*}%\label{eq:qhd-rr-base-change}
 \omega_R^{[i]}/t\omega_R^{[i]} \to \omega_A^{[i]}
\end{equation*}
is an isomorphism, and \(\omega_R^{[i]}\) is maximal Cohen--Macaulay.
\end{theorem}

\begin{proof}
Fix \(i \in \mathbb{Z}\), and choose a good resolution \(p\colon(X,E) \to (V,0)\) dominating the minimal good resolution, on which the full pullback of \(\omega_A^{[i]}\) is a line bundle. Write
\[
 E=\sum_{a=1}^n E_a,\quad
 \mathscr{L}_i \simeq \mathcal{O}_X(iK_X+Z_i),\quad
 u_i:=iK+Z_i,
\]
where \(Z_i\) is an integral exceptional divisor and \(K\) is the numerical canonical cycle.

By Stipsicz--Szab\'{o}--Wahl~\cite[Proposition~2.4(2), (4)]{SSW2008}, the minimal good resolution lattice embeds in a negative diagonal lattice of the same rank. Under this embedding, the numerical canonical cycle maps to an integral characteristic element, and its square is the negative of the rank. These properties persist under point blowups. Indeed, a point blowup \(b\) adds an orthogonal class \(e\) with \(e^2=-1\), and the numerical canonical cycle becomes \(b^*K+e\). Thus the chosen resolution lattice also embeds in a negative diagonal lattice of rank \(n\), in which \(K\) is integral and characteristic and \(K^2=-n\).

Choose a diagonal basis \(e_1,\dotsc,e_n\) with \(e_a\cdot e_b=-\delta_{ab}\), and write \(K=\sum_a k_a e_a\). Since \(K\) is characteristic, every \(k_a\) is odd. The equality \(K^2=-n\) gives \(\sum_a k_a^2=n\), so every \(k_a\) is \(1\) or \(-1\). Changing the signs of the basis vectors, we may assume that \(K=\sum_a e_a\).

Since \(\mu_\pi=0\) and \(K^2=-n\), Theorem~\ref{thm:rr} gives
\[
 \delta_i=\frac{u_i(u_i-K)}2.
\]
Both \(K\) and the image of \(Z_i\) are integral in the diagonal lattice. Hence \(u_i=\sum_a u_{i,a}e_a\) with \(u_{i,a} \in \mathbb{Z}\), and
\begin{equation*}%\label{eq:qhd-rr-defect-vanishing}
 0 \leq \delta_i
 =-\frac12\sum_{a=1}^n u_{i,a}(u_{i,a}-1)
 \leq 0.
\end{equation*}
Thus \(\delta_i=0\). Since \(i\) was arbitrary, every reflexive canonical power specializes. Lemma~\ref{prop:depth} also shows that every \(\omega_R^{[i]}\) is maximal Cohen--Macaulay.

The rational singularity \((V,0)\) has finite canonical index \(r\). On its minimal good resolution, \(K^2=-n<0\), so it is not a rational double point. Since a rational Gorenstein surface singularity is a rational double point~\cite[pp.~134--136]{Artin1966}, \((V,0)\) is non-Gorenstein and \(r>1\). Proposition~\ref{thm:finite} now gives index preservation and a \(G\)-equivariant smoothing of the index-one cover.
\end{proof}

Since \((V,0)\) is non-Gorenstein, we can apply Theorem~\ref{thm:QG-conjecture}.

\begin{corollary}\label{cor:qhd-components}
Every normal surface singularity admitting a rational homology disk smoothing is rational and quasi-homogeneous. Every such smoothing \(\pi\) occurs over a one-dimensional smoothing component and satisfies \(\alpha_\pi=0\) and \(\Delta(V,0)=0\).
\end{corollary}

\begin{proof}
By Theorem~\ref{thm:qhd-qg}, the smoothing is strongly \(\mathbb{Q}\)-Gorenstein. Since \(\mu_\pi=0\) and a smoothing component has positive dimension, Theorem~\ref{thm:QG-conjecture} gives
\[
 1 \leq \tau_\pi \leq 1.
\]
Hence \(\tau_\pi=1\), and the equality criterion gives quasi-homogeneity. Theorem~\ref{thm:qhd-qg} and Lemma~\ref{prop:duality} give \(\alpha_\pi=\delta_{-1}=0\). Wahl's formula~\eqref{eq:wahl} therefore gives
\[
 \Delta(V,0)=1+\mu_\pi-\tau_\pi+\alpha_\pi=0. \qedhere
\]
\end{proof}

Corollary~\ref{cor:qhd-components} reduces the classification to the weighted homogeneous singularities studied by Stipsicz--Szab\'{o}--Wahl~\cite{SSW2008} and Bhupal--Stipsicz~\cite{BhupalStipsicz2011}. Their resolution graphs are the cyclic quotient graphs and the star-shaped graphs listed in Bhupal--Stipsicz~\cite[Figures~1 and~2]{BhupalStipsicz2011}; see Fowler~\cite[Theorem~1]{Fowler2013} and Wahl~\cite[\S1]{Wahl2021}. The cyclic quotient graphs and the valency-three graphs in Figure~1 are taut by Laufer~\cite{Laufer1973}, so each determines a unique analytic type. For the valency-four graphs, Fowler~\cite[Theorem~5]{Fowler2013} determines the admissible cross ratios in these three families.

For the cross ratios below, let \(x_1,x_2,x_3,x_4\) be the ordered attachment points on the central curve and set \(\lambda:=\varphi(x_4)\), where \(\varphi\) is the unique isomorphism to \(\mathbb{P}^1\) sending \(x_1,x_2,x_3\) to \(0,1,\infty\), respectively. In Fowler's basic models of types \((3,3,3;4)\), \((2,4,4;3)\), and \((2,3,6;2)\), we order the neighboring curves as \((L_{02},L_{03},L_{04},L_{01})\), \((L_{01},L_{02},L_{03},L_{05})\), and \((L_1,A_2,L_4,A_1)\), respectively; see Fowler~\cite[Examples~3.5.6, 3.5.10, and~3.5.14]{Fowler2013}.

\begin{corollary}[Complete analytic classification]\label{thm:qhd-classification}
The normal complex surface singularities admitting a rational homology disk smoothing are precisely the following:
\begin{enumerate}
\item the cyclic quotient singularities
\[
 \frac1{p^2}(1,pq-1), p>q>0, \gcd(p,q)=1;
\]
\item the weighted homogeneous singularities determined by the taut graphs in Bhupal--Stipsicz~\cite[Figure~1]{BhupalStipsicz2011};
\item the weighted homogeneous singularities with graphs in Bhupal--Stipsicz~\cite[Figure~2]{BhupalStipsicz2011}, with the cross ratio \(\lambda=e^{\pm\pi i/3}\) for \((3,3,3;4)\), \(\lambda=2\) for \((2,4,4;3)\), and \(\lambda=9\) for \((2,3,6;2)\), under the convention above.
\end{enumerate}
\end{corollary}

\begin{remark}\label{rem:qhd-csrc-application}
Theorem~\ref{thm:qhd-qg} and Corollary~\ref{cor:qhd-components} show that every normal surface singularity admitting a rational homology disk smoothing is rational and weighted homogeneous with \(\Delta(V,0)=0\), and that every such smoothing \(\pi\) is strongly \(\mathbb{Q}\)-Gorenstein with \(\tau_\pi=1\). The strong \(\mathbb{Q}\)-Gorenstein condition implies that the canonical-specialization representation of Proposition~\ref{prop:canonical-specialization-representation} vanishes. Expressing the vanishing of its character multiplicities in terms of the Seifert data of weighted homogeneous surface singularities and combining the resulting equations with Hodge inequalities yields explicit formulas for the Seifert data in many cases. Park--Shin~\cite{ParkShinQHD} combine these equations with the properties above and mainly algebraic-geometric arguments to recover the Bhupal--Stipsicz~\cite{BhupalStipsicz2011} classification of weighted homogeneous surface singularities admitting rational homology disk smoothings.
\end{remark}


\begin{thebibliography}{99}
\bibitem{Ambro2016}
F.~Ambro, \textit{Cyclic covers and toroidal embeddings}, Eur. J. Math. \textbf{2} (2016), 9--44, \href{https://doi.org/10.1007/s40879-015-0084-y}{doi:10.1007/s40879-015-0084-y}.

\bibitem{Artin1966}
M.~Artin, \textit{On isolated rational singularities of surfaces}, Amer. J. Math. \textbf{88} (1966), no.~1, 129--136, \href{https://doi.org/10.2307/2373050}{doi:10.2307/2373050}.

\bibitem{AtiyahBott1967}
M.~F. Atiyah and R.~Bott, \textit{A Lefschetz fixed point formula for elliptic complexes. I}, Ann. of Math. (2) \textbf{86} (1967), no.~2, 374--407, \href{https://doi.org/10.2307/1970694}{doi.org/10.2307/1970694}; \textit{II. Applications}, \textbf{88} (1968), no.~3, 451--491, \href{https://doi.org/10.2307/1970721}{doi.org/10.2307/1970721}.

\bibitem{AvendanoCamachoVorobiev2011}
M.~Avenda\~{n}o Camacho and Y.~M. Vorobiev, \textit{Homological equations for tensor fields and periodic averaging}, Russian J. Math. Phys. \textbf{18} (2011), no.~3, 243--257, \href{https://doi.org/10.1134/S1061920811030010}{doi.org/10.1134/S1061920811030010}.

\bibitem{BPSS2026}
M. Beke, H. Park, D. Shin, A. I. Stipsicz, \textit{Normal complex surface singularities with rational homology disk smoothings}, \href{https://doi.org/10.48550/arXiv.2610.20991}{arXiv:2610.20991}.

\bibitem{BergnerGraf2020}
H.~Bergner and P.~Graf, \textit{The Lipman--Zariski conjecture in genus one higher}, Forum Math. Sigma \textbf{8} (2020), e21, 16 pp., \href{https://doi.org/10.1017/fms.2020.19}{doi.org/10.1017/fms.2020.19}.

\bibitem{BhupalStipsicz2011}
M.~Bhupal and A.~I. Stipsicz, \textit{Weighted homogeneous singularities and rational homology disk smoothings}, Amer. J. Math. \textbf{133} (2011), no.~5, 1259--1297, \href{https://doi.org/10.1353/ajm.2011.0036}{doi:10.1353/ajm.2011.0036}.

\bibitem{Blache1995}
R.~Blache, \textit{Riemann--Roch theorem for normal surfaces and applications}, Abh. Math. Sem. Univ. Hamburg \textbf{65} (1995), 307--340, \href{https://doi.org/10.1007/BF02953338}{doi:10.1007/BF02953338}.

\bibitem{BuchweitzGreuel1980}
R.-O.~Buchweitz and G.-M.~Greuel, \textit{The Milnor number and deformations of complex curve singularities}, Invent. Math. \textbf{58} (1980), no.~3, 241--281, \href{https://doi.org/10.1007/BF01390254}{doi.org/10.1007/BF01390254}.

\bibitem{DaoIyamaTakahashiWemyss2020}
H.~Dao, O.~Iyama, R.~Takahashi and M.~Wemyss, \textit{Gorenstein modifications and \(\mathbb{Q}\)-Gorenstein rings}, J. Algebraic Geom. \textbf{29} (2020), no.~4, 729--751, \href{https://doi.org/10.1090/jag/760}{doi:10.1090/jag/760}.

\bibitem{DeJongVanStraten1992}
T.~de Jong and D.~van Straten, \textit{A construction of \(\mathbb{Q}\)-Gorenstein smoothings of index two}, Internat. J. Math. \textbf{3} (1992), no.~3, 341--347, \href{https://doi.org/10.1142/S0129167X92000126}{doi:10.1142/S0129167X92000126}.

\bibitem{DuYau2010}
R.~Du and S.~S.-T.~Yau, \textit{Local holomorphic De Rham cohomology}, Comm. Anal. Geom. \textbf{18} (2010), no.~2, 365--374, \href{https://doi.org/10.4310/CAG.2010.v18.n2.a4}{doi:10.4310/CAG.2010.v18.n2.a4}.

\bibitem{EbelingGuseinZadeSeade2004}
W.~Ebeling, S.~M. Gusein-Zade, and J.~Seade, \textit{Homological index for 1-forms and a Milnor number for isolated singularities}, Internat. J. Math. \textbf{15} (2004), no.~9, 895--905, \href{https://doi.org/10.1142/S0129167X04002624}{doi.org/10.1142/S0129167X04002624}.

\bibitem{EsnaultViehweg1992}
H.~Esnault and E.~Viehweg, \textit{Lectures on vanishing theorems}, DMV Seminar, vol.~20, Birkh\"{a}user, Basel, 1992, \href{https://doi.org/10.1007/978-3-0348-8600-0}{doi:10.1007/978-3-0348-8600-0}.

\bibitem{FavreRuggiero2014}
C.~Favre and M.~Ruggiero, \textit{Normal surface singularities admitting contracting automorphisms}, Ann. Fac. Sci. Toulouse Math. (6) \textbf{23} (2014), no.~4, 797--828, \href{https://doi.org/10.5802/afst.1425}{doi.org/10.5802/afst.1425}.

\bibitem{FlennerZaidenberg1994}
H.~Flenner and M.~Zaidenberg, \textit{\(\mathbb{Q}\)-acyclic surfaces and their deformations}, in \textit{Classification of algebraic varieties} (L'Aquila, 1992), Contemp. Math., vol.~162, Amer. Math. Soc., Providence, RI, 1994, pp.~143--208, \href{https://doi.org/10.1090/conm/162/01532}{doi.org/10.1090/conm/162/01532}.

\bibitem{Fowler2013}
J.~R. Fowler, \textit{Rational homology disk smoothing components of weighted homogeneous surface singularities}, Ph.D. thesis, University of North Carolina at Chapel Hill, 2013, \href{https://cdr.lib.unc.edu/concern/dissertations/b2773w797} {UNC Digital Repository}.

\bibitem{GrauertRiemenschneider1970}
H.~Grauert and O.~Riemenschneider, \textit{Verschwindungss\"{a}tze f\"{u}r analytische Kohomologiegruppen auf komplexen R\"{a}umen}, Invent. Math. \textbf{11} (1970), 263--292, \href{https://doi.org/10.1007/BF01403182} {doi:10.1007/BF01403182}.

\bibitem{Greuel1980}
G.-M.~Greuel, \textit{Dualit\"{a}t in der lokalen Kohomologie isolierter Singularit\"{a}ten}, Math. Ann. \textbf{250} (1980), 157--173, \href{https://doi.org/10.1007/BF02599793} {doi:10.1007/BF02599793}.

\bibitem{GreuelSteenbrink1983}
G.-M.~Greuel and J.~H.~M. Steenbrink, \textit{On the topology of smoothable singularities}, Proc. Sympos. Pure Math. \textbf{40}, Part~1 (1983), 535--545, \href{https://doi.org/10.1090/pspum/040.1/713090} {doi:10.1090/pspum/040.1/713090}.

\bibitem{JorderExtension}
C.~Jörder, \textit{A weak version of the Lipman--Zariski conjecture}, Math. Z. \textbf{278} (2014), nos.~3--4, 893--899, \href{https://doi.org/10.1007/s00209-014-1337-z} {doi:10.1007/s00209-014-1337-z}.

\bibitem{JorderPoincare2014}
C.~Jörder, \textit{On the Poincar\'{e} Lemma for reflexive differential forms}, preprint, 2014, \href{https://arxiv.org/abs/1401.7495}{arxiv:1401.7495}.

\bibitem{KollarShepherdBarron1988}
J.~Koll\'{a}r and N.~I. Shepherd-Barron, \textit{Threefolds and deformations of surface singularities}, Invent. Math. \textbf{91} (1988), no.~2, 299--338, \href{https://doi.org/10.1007/BF01389370} {doi:10.1007/BF01389370}.

\bibitem{Konno2008}
K.~Konno, \textit{On the cohomological cycle of a normal surface singularity}, RIMS K\^{o}ky\^{u}roku Bessatsu \textbf{B9} (2008), 143--151, %\url{https://www.kurims.kyoto-u.ac.jp/~kenkyubu/bessatsu/open/B9/pdf/B09_007.pdf}.

\bibitem{Laufer1973}
H.~B. Laufer, \textit{Taut two-dimensional singularities}, Math. Ann. \textbf{205} (1973), 131--164, \href{https://doi.org/10.1007/BF01350842} {doi:10.1007/\allowbreak BF01350842}.

\bibitem{LeeNakayama2018}
Y.~Lee and N.~Nakayama, \textit{Grothendieck duality and \(\mathbb{Q}\)-Gorenstein morphisms}, Publ. Res. Inst. Math. Sci. \textbf{54} (2018), 517--648, \href{https://doi.org/10.4171/PRIMS/54-3-3}{doi:10.4171/PRIMS/54-3-3}.

\bibitem{LimaPereiraEtAl2024}
B.~K. Lima-Pereira, J.~J. Nu\~{n}o-Ballesteros, B.~Or\'{e}fice-Okamoto, and J.~N. Tomazella, \textit{The Bruce--Roberts numbers of a function on an ICIS}, Q. J. Math. \textbf{75} (2024), no.~1, 31--50, \href{https://doi.org/10.1093/qmath/haad038} {doi:10.1093/qmath/haad038}.

\bibitem{Lipman1969}
J.~Lipman, \textit{Rational singularities with applications to algebraic surfaces and unique factorization}, Publ. Math. Inst. Hautes \'{E}tudes Sci. \textbf{36} (1969), 195--279, \href{https://doi.org/10.1007/BF02684604}{doi:10.1007/BF02684604}.

\bibitem{Looijenga1986}
E.~Looijenga, \textit{Riemann--Roch and smoothings of singularities}, Topology \textbf{25} (1986), no.~3, 293--302, \href{https://doi.org/10.1016/0040-9383(86)90045-5} {doi:10.1016/0040-9383(86)90045-5}.

\bibitem{LooijengaSteenbrink1985}
E.~Looijenga and J.~H.~M.~Steenbrink, \textit{Milnor number and Tjurina number of complete intersections}, Math. Ann. \textbf{271} (1985), no.~1, 121--124, \href{https://doi.org/10.1007/BF01455800} {doi:10.1007/BF01455800}.

\bibitem{LooijengaWahl1986}
E.~Looijenga and J.~Wahl, \textit{Quadratic functions and smoothing surface singularities}, Topology \textbf{25} (1986), no.~3, 261--291, \href{https://doi.org/10.1016/0040-9383(86)90044-3} {doi:10.1016/0040-9383(86)90044-3}.

\bibitem{Matsuki2002}
K.~Matsuki, \textit{Introduction to the Mori Program}, Universitext, Springer-Verlag, New York, 2002, \href{https://doi.org/10.1007/978-1-4757-5602-9} {doi:10.1007/978-1-4757-5602-9}.

\bibitem{Milnor}
J.~Milnor, \textit{Singular points of complex hypersurfaces}, Annals of Mathematics Studies, no.~61, Princeton University Press, Princeton, NJ, 1968, \href{https://doi.org/10.1515/9781400881819}{doi.org/10.1515/9781400881819}.

\bibitem{Mueller}
G.~M\"{u}ller, \textit{Symmetries of surface singularities}, J. London Math. Soc. (2) \textbf{59} (1999), no.~2, 491--506, \href{https://doi.org/10.1112/S0024610799007188} {doi:10.1112/S0024610799007188}.

\bibitem{NagyNemethi2019}
J.~Nagy and A.~N\'{e}methi, \textit{The Abel map for surface singularities I: generalities and examples}, Math. Ann. \textbf{375} (2019), 1427--1487, \href{https://doi.org/10.1007/s00208-019-01873-w}{doi:10.1007/s00208-019-01873-w}.

\bibitem{Okuma1997}
T.~Okuma, \textit{The plurigenera of Gorenstein surface singularities}, Manuscripta Math. \textbf{94} (1997), 187--194, \href{https://doi.org/10.1007/BF02677846}{doi.org/10.1007/BF02677846}.

\bibitem{ParkShinNonlinear}
H.~Park and D.~Shin, \textit{Nonlinear neighborhoods and arbitrarily negative Wahl defects}, \href{https://doi.org/10.48550/arXiv.2610.21093}{arXiv.2610.21093}.

\bibitem{ParkShinQHD}
H.~Park and D.~Shin, \textit{Weighted homogeneous surface singularities admitting rational homology disk smoothings}, preprint 2026.

\bibitem{ParkShinTaut}
H.~Park and D.~Shin, \textit{The classification of taut surface singularities via representations from logarithmic vector fields}, preprint 2026.

\bibitem{PopescuPampuSeade2009}
P.~Popescu-Pampu and J.~Seade, \textit{A finiteness theorem for dual graphs of surface singularities}, Internat. J. Math. \textbf{20} (2009), no.~8, 1057--1068, \href{https://doi.org/10.1142/S0129167X09005649}{doi:10.1142/S0129167X09005649}.

\bibitem{PopescuPampu2011}
P.~Popescu-Pampu, \textit{Numerically Gorenstein surface singularities are homeomorphic to Gorenstein ones}, Duke Math. J. \textbf{159} (2011), no.~3, 539--559, \href{https://doi.org/10.1215/00127094-1433412}{doi:10.1215/00127094-1433412}.

\bibitem{Reid1987}
M.~Reid, \textit{Young person's guide to canonical singularities}, in \textit{Algebraic Geometry, Bowdoin, 1985} (Brunswick, Maine, 1985), Proc. Sympos. Pure Math. \textbf{46}, Part~1, Amer. Math. Soc., Providence, RI, 1987, 345--414, \href{https://doi.org/10.1090/pspum/046.1/927963} {doi:10.1090/pspum/046.1/927963}.

\bibitem{RoulleauUrzua2015}
X.~Roulleau and G.~Urz\'{u}a, \textit{Chern slopes of simply connected complex surfaces of general type are dense in \([2,3]\)}, Ann. of Math. (2) \textbf{182} (2015), no.~1, 287--306, \href{https://doi.org/10.4007/annals.2015.182.1.6} {doi:10.4007/annals.2015.182.1.6}.

\bibitem{Saito1971}
K.~Saito, \textit{Quasihomogene isolierte Singularit\"{a}ten von Hyperfl\"{a}chen}, Invent. Math. \textbf{14} (1971), no.~2, 123--142, \href{https://doi.org/10.1007/BF01405360} {doi:10.1007/BF01405360}.

\bibitem{Saito1974}
K.~Saito, \textit{Einfach-elliptische Singularit\"{a}ten}, Invent. Math. \textbf{23} (1974), no.~3, 289--325, \href{https://doi.org/10.1007/BF01389749}{doi:10.1007/\allowbreak BF01389749}.

\bibitem{SchejaWiebe1980}
G.~Scheja and H.~Wiebe, \textit{Zur Chevalley-Zerlegung von Derivationen}, Manuscripta Math. \textbf{33} (1980/81), no.~2, 159--176, \href{https://doi.org/10.1007/BF01316974} {doi:10.1007/BF01316974}.

\bibitem{Serre1977}
J.-P. Serre, \textit{Linear Representations of Finite Groups}, translated from the French by L.~L. Scott, Graduate Texts in Mathematics, vol.~42, Springer-Verlag, New York--Heidelberg, 1977, \href{https://doi.org/10.1007/978-1-4684-9458-7}{doi.org/10.1007/978-1-4684-9458-7}.

\bibitem{Singh2003}
A.~K. Singh, \textit{Cyclic covers of rings with rational singularities}, Trans. Amer. Math. Soc. \textbf{355} (2003), no.~3, 1009--1024, \href{https://doi.org/10.1090/S0002-9947-02-03186-0} {doi:10.1090/S0002-9947-02-03186-0}.

\bibitem{Stacks}
The Stacks Project Authors, \textit{The Stacks Project}, \href{https://stacks.math.columbia.edu} {https://stacks.math.columbia.edu}.

\bibitem{Steenbrink1995}
J.~H.~M. Steenbrink, \textit{Monodromy and weight filtration for smoothings of isolated singularities}, Compositio Math. \textbf{97} (1995), no.~1--2, 285--293, \href{https://www.numdam.org/item/CM_1995__97_1-2_285_0/} {Numdam: CM\_1995\_\_97\_1-2\_285\_0}.

\bibitem{SSW2008}
A.~I. Stipsicz, Z.~Szab\'{o}, and J.~Wahl, \textit{Rational blowdowns and smoothings of surface singularities}, J. Topol. \textbf{1} (2008), no.~2, 477--517, \href{https://doi.org/10.1112/jtopol/jtn009} {doi:10.1112/jtopol/jtn009}.

\bibitem{VanStratenSteenbrink1985}
D.~van Straten and J.~H.~M.~Steenbrink, \textit{Extendability of holomorphic differential forms near isolated hypersurface singularities}, Abh. Math. Sem. Univ. Hamburg \textbf{55} (1985), 97--109, \href{https://doi.org/10.1007/BF02941491}{doi:10.1007/BF02941491}.

\bibitem{Vosegaard2002}
H.~Vosegaard, \textit{A characterization of quasi-homogeneous complete intersection singularities}, J. Algebraic Geom. \textbf{11} (2002), no.~3, 581--597, \href{https://doi.org/10.1090/S1056-3911-02-00298-9} {doi:10.1090/S1056-3911-02-00298-9}.

\bibitem{Wahl1981}
J.~Wahl, \textit{Smoothings of normal surface singularities}, Topology \textbf{20} (1981), no.~3, 219--246, \href{https://doi.org/10.1016/0040-9383(81)90001-X} {doi:10.1016/0040-9383(81)90001-X}.

\bibitem{Wahl1985}
J.~Wahl, \textit{A characterization of quasi-homogeneous Gorenstein surface singularities}, Compositio Math. \textbf{55} (1985), no.~3, 269--288, \href{http://www.numdam.org/item/CM_1985__55_3_269_0/} {Numdam: CM\_1985\_\_55\_3\_269\_0}.

%%\bibitem{Wahl2011}
%%J.~Wahl, \textit{On rational homology disk smoothings of valency 4 surface singularities}, Geom. Topol. \textbf{15} (2011), no.~2, 1125--1156, \href{https://doi.org/10.2140/gt.2011.15.1125}{doi:10.2140/gt.2011.15.1125}.

\bibitem{Wahl2013}
J. Wahl, \textit{Log-terminal smoothings of graded normal surface singularities}, Michigan Math. J. \textbf{62} (2013), 475--489. DOI: 10.1307/mmj/1378757884.

\bibitem{Wahl2015}
J.~Wahl, \textit{Milnor and Tjurina numbers for smoothings of surface singularities}, Algebraic Geometry \textbf{2} (2015), no.~3, 315--331, \href{https://doi.org/10.14231/AG-2015-014}{doi:10.14231/AG-2015-014}.

\bibitem{Wahl2021}
J.~Wahl, \textit{Complex surface singularities with rational homology disk smoothings}, in \textit{Singularities and Their Interaction with Geometry and Low Dimensional Topology}, Trends in Mathematics, Birkh\"{a}user/Springer, Cham, 2021, 259--287, \href{https://doi.org/10.1007/978-3-030-61958-9_12} {doi:10.1007/978-3-030-61958-9\_12}.

\bibitem{Watanabe1981}
K.~Watanabe, \textit{Some remarks concerning Demazure's construction of normal graded rings}, Nagoya Math. J. \textbf{83} (1981), 203--211, \href{https://doi.org/10.1017/S0027763000019498}{doi:10.1017/S0027763000019498}.


\end{thebibliography}
\end{document}